\documentclass[a4paper,10pt]{amsart}
\usepackage{etex}

\usepackage[latin1]{inputenc}
\usepackage{amsmath, amsthm, amssymb}
\usepackage{amscd}
\usepackage[dvips]{graphicx}
\usepackage[all]{xy}
\usepackage{enumerate}
\usepackage{hyperref}

\usepackage{caption}
\usepackage{subcaption}
\usepackage[rightcaption]{sidecap}

\usepackage[usenames,dvipsnames]{pstricks}
\usepackage{epsfig}
\usepackage{pst-grad} % For gradients
\usepackage{pst-plot} % For axes
\usepackage{pstricks-add}
\usepackage{pst-solides3d}

\usepackage[margin=2.5cm]{geometry}

\usepackage{color}

\newtheorem{theorem}{Theorem}[section]
\newtheorem{corollary}[theorem]{Corollary}
\newtheorem{proposition}[theorem]{Proposition}
\newtheorem{definition}[theorem]{Definition}
\newtheorem{lemma}[theorem]{Lemma}
\newtheorem{claim}[theorem]{Claim}
\newtheorem{conjecture}[theorem]{Conjecture}
\newtheorem*{theorem*}{Theorem}
\newtheorem*{proposition*}{Proposition}
\newtheorem*{definition*}{Definition}
\newtheorem*{lemma*}{Lemma}
\newtheorem*{claim*}{Claim}
\newtheorem*{corollary*}{Corollary}

\newtheorem{thmintro}{Theorem}

\theoremstyle{definition}
\newtheorem{convention}{Convention}
\newtheorem{example}{Example}
\newtheorem{question}{Question}

\theoremstyle{remark}
\newtheorem{rem}[theorem]{Remark}
\newtheorem*{rem*}{Remark}

\newcommand\M{\widetilde{M}}

\newcommand{\wt}[1]{\widetilde{#1}}

\newcommand\R{\mathbb R}
\newcommand\Z{\mathbb Z}
\newcommand\N{\mathbb N}

\newcommand\Hyp{\mathbb H}

\newcommand\eps{\varepsilon}

\DeclareMathOperator{\Cl}{Cl}

\newcommand\flot{ \phi^{t} }
\newcommand\hflot{ \tilde{ \phi}^t }

\newcommand\orb{ \mathcal O }
\newcommand\leafs{ \mathcal L ^{s} }
\newcommand\leafu{ \mathcal L ^{u} }

\newcommand\fs{\mathcal F^{s} }
\newcommand\hfs{\widetilde{\mathcal F}^{s} }
\newcommand\fu{\mathcal F^{u} }
\newcommand\hfu{\widetilde{\mathcal F}^{u} }

\newcommand{\al}[1]{\widetilde{\alpha_{#1} } }

\newcommand{\FH}{\mathcal{FH}}

\newcommand{\Conj}{\mathrm{CCl}}

\title[Counting orbits in free homotopy classes]{Counting periodic orbits of Anosov flows in free homotopy classes}
\author{Thomas Barthelm\'e}
\address{The Pennsylvania State University, State College, PA 16802}
\email{thomas.barthelme@psu.edu}
\urladdr{sites.google.com/site/thomasbarthelme}

\author{Sergio R.\ Fenley} 
\address{Florida State University, Tallahassee, FL 32306 \and \newline \indent Princeton University, Princeton, NJ 08540, USA}
\email{fenley@math.fsu.edu}


\begin{document}
 
\maketitle

\begin{abstract}
 The main result of this article is that if a $3$-manifold $M$ supports an Anosov flow, then the number of conjugacy classes in the fundamental group of $M$ grows exponentially fast with the length of the shortest orbit representative, hereby answering a question raised by Plante and Thurston in 1972. 
In fact we show that, when the flow is transitive, the exponential growth rate is exactly the topological entropy of the flow.
We also show that
taking {\em only} the shortest orbit representatives in each conjugacy classes 
still yields Bowen's version of the measure of maximal entropy.
These results are achieved by obtaining counting results on the growth rate of the number of periodic orbits 
inside a \emph{free homotopy class}. In the first part of the article, we also construct many examples of Anosov flows having some finite and some infinite free homotopy classes of periodic orbits,
 and we also give a characterization of algebraic Anosov flows as the only $\R$-covered Anosov flows up to orbit equivalence
that do not admit at least one infinite free homotopy class of periodic orbits.
\end{abstract}

\section{Introduction}

A classical and fundamental
problem in dynamical systems is to count the number of closed orbits of the system
with respect to the period. For Anosov flows, Margulis in his thesis \cite{MargulisThesis}, and independently (and, more generally, for Axiom A flows) Bowen \cite{Bowen:periodic_orbits} showed that the number of orbits grows exponentially with the period. In fact, Margulis gave an asymptotic formula for the growth of the 
number of closed orbits as a function of the period
for weak-mixing flows and later Parry and Policott \cite{ParryPollicott} gave a formula for the general case.
In particular, the topological entropy, an ubiquitous quantity in dynamical systems that measures the complexity of the flow, appears in the asymptotics of the counting function \cite{Bowen:periodic_orbits,MargulisThesis}. Moreover, Bowen \cite{Bowen:periodic_orbits} showed that the unique invariant measure of maximal entropy is supported by periodic orbits, that is, it can be obtained as the normalization of the sum of the Lebesgue probability measures supported on periodic orbits. Bowen's and Margulis' work have been essential in the theory of hyperbolic dynamical systems.

At the same time that Bowen's work appeared, Plante and Thurston \cite{PlanteThurston} proved that if a manifold $M$ supports a \emph{codimension one} Anosov flow (i.e., such that one of the strong foliations of the flow is of dimension one), then $\pi_1(M)$ has exponential growth. 
In that paper, they also asked the following question:

\begin{question}[Plante, Thurston \cite{PlanteThurston}]
Suppose that $M$ supports a codimension one Anosov flow.
Does the number of \emph{conjugacy classes} in $\pi_1(M)$ grow exponentially fast with the length of 
a shortest orbit representative? 
\end{question}
Given Bowen's result, Plante and Thurston remark that a positive answer to this question can be obtained by giving a low upper bound on the growth of the number of orbits of period less than
$t$ as a function of $t$ \emph{inside a free homotopy class}.

To the best of our knowledge, it appears that no one has yet managed to answer 
Question 1 in any setting, nor obtained any results on the number of orbits inside a free homotopy class.

The main goal of this article is to give a positive answer to Plante and Thurston's question in the case of $3$-manifolds. In fact we will obtain more information
since we also get a (coarse) estimate of the growth rate of the number of conjugacy classes, as well as an equidistribution property of a shortest orbit representative (see Theorem \ref{thmintro:exponential_growth_conjugacy_classes} below).

Before stating our results, we review what is known about freely homotopic
periodic orbits.
A {\em free homotopy class} is a maximal collection of closed orbits
that are pairwise freely homotopic to each other.
First, it is easy to note that, in the case of algebraic flows, or more generally flows that are 
orbit equivalent to either a suspension of an Anosov diffeomorphism or the geodesic flow of a negatively curved metric, then the answer 
to Plante and Thurston's question is clearly yes. Indeed, in that case, there is at most one periodic orbit in each free homotopy class (or two in the case of geodesic flows if one takes our definition of free homotopy that forgets about the direction of an orbit, see Convention \ref{convention1}). Hence Bowen's or Margulis' work directly implies the result.

Plante and Thurston knew of the existence of Anosov flows admitting really distinct orbits in the same free homotopy class. 
But non trivial
explicit examples were not constructed for more than ten years afterwards. 
The ``trivial" examples are obtained 
as finite lifts of the geodesic flow on the unit tangent
bundle of a hyperbolic surface. These manifolds are Seifert fibered
(see definition in the section \ref{sec:background_prelim})
and finite covers of any order can be obtained
by unrolling the Seifert fibers. 
Then for each natural number $n$
one can obtain examples where every free homotopy 
class has $2n$ elements.
These examples are in some sense artificial, for example all
orbits in a given free homotopy class have exactly the same 
length in the lifted metric.
Plante and Thurston
did not know whether there is an upper bound on the number of orbits in an
arbitrary  free homotopy class. In 1994, 
the second author \cite{Fen:AFM} constructed examples of Anosov flows in $3$-manifolds such that \emph{every} periodic orbit is
freely homotopic to \emph{infinitely} many distinct orbits. 
In particular, for any enumeration of the orbits in the
free homotopy class, the lengths of the orbits diverge to
infinity.
It follows that counting orbits inside a free homotopy class is a natural, and also non-trivial question at least for some Anosov flows.

One goal of this article is to show that infinite free homotopy 
classes are very common amongst Anosov flows. An Anosov flow in a $3$-manifold is called ${\R}$-{\em covered} if 
the stable (or equivalently the unstable) foliation lifts to a foliation
in the universal cover that has leaf space homeomorphic to the reals.
A vast amount of such flows exists \cite{Barbot:VarGraphees,Fen:AFM}.
We will show in this article that, when one considers
$\R$-covered Anosov flows 
then a flow is either orbit equivalent to a finite cover
of an algebraic Anosov flow,
or it admits an enormous amount of free homotopy classes with infinitely many distinct orbits 
(see Theorem \ref{thmintro:finite_hom_class}). For general Anosov flows, even if the same result does not hold, it 
is not hard to construct examples with some infinite free homotopy classes. 
So Plante and Thurston's question is not trivial for ``most'' Anosov flows on $3$-manifolds.

Before stating more precisely our results, there are a few more remarks that one should make.

First, conjecturally, Plante and Thurston's question is trivial in dimension at least 4. Indeed, Verjovsky conjecture \cite{Ver:codim1} states that any \emph{codimension one} Anosov flow in dimension at least 4 is orbit
equivalent to a suspension of an Anosov diffeomorphism, so a free homotopy class contains at most one orbit. 
Verjovsky conjecture is still open in full generality, but 
it has been proven if the fundamental group  is solvable \cite{Plante:verjovsky_conjecture}, or given some smoothness conditions on the Anosov splitting \cite{Ghys:codim_one,Simic:codim_one}.

Second, consider the question of counting
periodic orbits inside a free homotopy class for a generic Anosov flow, that is for an Anosov flow in higher codimension. 
One must note that nothing is known about the topology of
manifolds admitting Anosov flows in higher codimension. For instance, in higher codimension, we do not know whether a periodic orbit has to be homotopically non trivial. In particular, no one even knows whether $\mathbb{S}^5$ supports (or, presumably, does not) an Anosov flow. Not knowing the answer to this most basic question does not bode well for trying to understand fine properties of free homotopy classes.
More explicitly all the techniques used in this article for Anosov flows
in $3$-manifolds completely break down in higher codimension, because we do not yet have 
any of the understanding of free homotopy classes that we have for $3$-manifolds.

Finally, a problem that attracted a lot of attention in the past was to give counting results for the number of periodic orbits of an Anosov flow inside a fixed \emph{homology} class. Amongst others, Katsuda and Sunada \cite{KatSun} (in the case of a surface with an hyperbolic metric), Phillips and Sarnak \cite{PhillipsSarnak} (for geodesic flows in negative sectional curvature), Sharp \cite{Sharp:closed_orbits_homology_classes}, and Babillot and Ledrappier \cite{BabillotLedrappier} gave precise estimates for the asymptotic of the number of periodic orbits 
of period less than a given real number
in a homology class. Sharp's and Babillot--Ledrappier's asymptotics holds for Anosov flows on any manifold, provided the flow is homologically 
full (i.e., if every homology class admits at least one periodic orbit) or a suspension. In particular, there are no assumptions on the dimension of the manifold.

Given all the results on counting closed orbits
inside a homology class, the lack of counting results inside free homotopy classes
 seems even more surprising. 
However, it must be mentioned that the tools used in the homological
setting rely in part on deep number theoretical results, whereas our tools are only topological and geometric in nature. This difference in available tools
also affects the results: while they obtained 
precise asymptotics in the homological
setting, we only get relatively coarse upper and lower bounds.

We can now present more carefully the results of this article. From now on, we will always be in a $3$-manifold setting.

\subsection{Statement of results}

Our main result about the growth of conjugacy class and the equidistribution of a shortest orbit in 
a conjugacy class is the following (see Theorem \ref{thm:counting_conjugacy_classes} for a more precise version):
\begin{thmintro} \label{thmintro:exponential_growth_conjugacy_classes}
Let $\flot$ be an Anosov flow on $M^3$.
Then the number of conjugacy classes in $\pi_1(M)$ grows exponentially fast with the length of a 
shortest representative closed orbit in the conjugacy class.

Moreover, if the flow is transitive, then the exponential growth rate is given by the topological entropy of the flow.
That is, if we write $\Cl(h)$ for the conjugacy class of an element $h \in \pi_1(M)$, $\alpha_{\Cl(h)}$ for a shortest periodic orbit in the conjugacy class $\Cl(h)$ (if such a periodic orbit exists in that class), and
\[
 \Conj(t):= \left\{ \Cl(h) \mid h \in \pi_1(M), \; l\left( \alpha_{\Cl(h)}\right) <t \right\},
\]
 then we have
\begin{equation*}
 \limsup_{t \rightarrow +\infty} \frac{1}{t} \log \sharp \Conj(t) = h_{\textrm{top}},
\end{equation*}
where $h_{\textrm{top}}$ is the topological entropy of the flow.

Furthermore, the Bowen-Margulis measure $\mu_{BM}$ of the transitive flow $\flot$ (i.e., measure of maximal entropy) can be obtained as
\[
 \mu_{BM} = \lim_{t\rightarrow +\infty} \frac{1}{\sharp \Conj(t)}\sum_{\Cl(h) \in \Conj(t)} \delta_{\alpha_{\Cl(h)}},
\]
where $\delta_{\alpha_{\Cl(h)}}$ is the Lebesgue probability measure supported on $\alpha_{\Cl(h)}$.
\end{thmintro}

As we will explain later, we obtain this result by counting orbits inside a free homotopy class. But in order to do that, 
we first need to understand free homotopy classes as well as possible. This is what the first part of this article is concerned with, and it owes a lot to the good understanding of the topology of Anosov flows in $3$-manifolds obtained through the cumulative work of Barbot and the second author \cite{Bar:CFA,Bar:MPOT,Barbot:VarGraphees,Bar:PAG,BarbotFenley1,BarbotFenley2,Fen:AFM,Fen:QGAF,Fenley:Incompressible_tori,Fen:SBAF}.
We first obtain a classification of $\R$-covered flows such that all of their free homotopy classes are finite:
\begin{thmintro} \label{thmintro:finite_hom_class}
 Let $\flot$ be an $\R$-covered Anosov flow on a closed $3$-manifold $M$. Suppose that every periodic orbit of $\flot$ is freely homotopic to \emph{at most} a finite number of other periodic orbits.
Then either $\flot$ is orbit equivalent to a suspension or, up to finite cover, $\flot$ is orbit equivalent to the geodesic flow of a negatively curved surface.
\end{thmintro}
Most of the pieces of the proof of this result were essentially known, 
thanks to the cumulative work of (mainly) Verjovsky, Ghys, Barbot and the second author.
The relevance of this result is that it shows that infinite free homotopy classes are
extremely common. In fact for $\R$-covered Anosov flows the nonexistence of infinite 
free homotopy classes is extremely rare.

An essential tool in this article will be the JSJ decomposition of a $3$-manifold.
In our setting it roughly states that any manifold supporting an
Anosov flow has a decomposition by embedded tori into pieces
that are either Seifert fibered or hyperbolic (see detailed
description in section \ref{section:good_JSJ}).

We also construct examples of 
contact Anosov flows (so, in particular, $\R$-covered, see \cite{Bar:PAG}) 
on manifolds admitting all possible types of JSJ decompositions,
 by doing Foulon--Hasselblatt \cite{FouHassel:contact_anosov} surgery on geodesic flows. By the result above, all these flows have some (in fact, infinitely many) infinite free homotopy classes.

Theorem \ref{thmintro:finite_hom_class} above is only true for $\R$-covered Anosov flows. 
In some sense it is not very surprising that this result does not hold
for non-transitive Anosov flows. But it turns out that it does not
even hold when the flow is transitive:
\begin{thmintro} \label{thmintro:non_algebraic_finite_free_homotopy}
 There exists (a large family of) non-algebraic transitive Anosov flows such that
every periodic orbit is freely homotopic to at most finitely many others.
\end{thmintro}

In fact, we will prove that all the examples of pseudo-Anosov flows constructed in \cite{BarbotFenley2}, called totally periodic, are such that every periodic orbit is freely homotopic to at most finitely many others, and many of these examples are Anosov and transitive. 

One should also point out that, up until now, all the known examples of Anosov flows with all their free homotopy classes finite were on graph-manifolds, i.e., manifolds so 
that their JSJ decomposition has only Seifert-fibered pieces. However, we also construct examples of 
transitive Anosov flows  on manifolds containing an atoroidal piece and so that
all free homotopy classes are finite.

The first step in order to obtain Theorem \ref{thmintro:finite_hom_class} is to use a JSJ decomposition of the manifold which is well adapted to the flow. A {\em modified JSJ decomposition} is one such that every torus of the 
decomposition is 
weakly embedded and \emph{quasi-transverse} to the flow, i.e., transverse except possibly for a finite number of periodic orbits of the flow where the flow is tangent to the entire orbit. 
Thanks to works of Barbot \cite{Barbot:VarGraphees} (for the $\R$-covered case) and Barbot and Fenley \cite{BarbotFenley1} (for the general case), 
any Anosov flow admits a modified JSJ decomposition (see section \ref{section:good_JSJ}). 
Using modified JSJ decompositions, we can prove that in general certain types of pieces
in the JSJ decomposition \emph{cannot} admit an infinite free homotopy class:
\begin{thmintro} \label{thmintro:no_periodic_piece}
 Suppose that $\FH(\alpha)$ is an \emph{infinite} free homotopy class of a
periodic orbit of an Anosov flow on $M$. Then no orbit of $\FH(\alpha)$ can cross 
a Seifert-fibered piece on which the flow is periodic, except, possibly, when the piece is a
twisted $I$-bundle over the Klein bottle.
In addition any Seifert fibered piece of the modified JSJ decomposition can only contain a
bounded number of orbits of $\FH(\alpha)$.
\end{thmintro}

Note that there are infinitely many examples of non $\R$-covered Anosov flows that admit some 
infinite free homotopy classes of orbits, but it seems difficult to come up with a topological 
criterion that would detect such a feature.

Given that infinite free homotopy classes are very common,
we now  move towards a proof of Theorem \ref{thmintro:exponential_growth_conjugacy_classes}.
This is the second part of this article.

To prove Theorem \ref{thmintro:exponential_growth_conjugacy_classes} we will show that lengths of orbits in a free homotopy class grow
{\emph{at least}} at a certain rate. The proof of this will be heavily dependent
on the geometric type of the manifold or of the pieces of the JSJ decomposition
of the manifold and how the geometric type of the pieces relates with the flow lines.

To obtain an upper bound on the number of orbits inside a free homotopy class with
period less than a given real number, we need two preliminary key lemmas.

The first result (Proposition \ref{prop:free_homotopy_class_to_strings}) shows that a free homotopy class can be split into a \emph{uniformly bounded} number of special orbits plus a \emph{uniformly bounded} number of what we call 
a \emph{string of orbits}. 
Strings of orbits are particular subsets of free homotopy classes: roughly speaking they are such that they do not involve non separated stable/unstable
leaves when lifted to the universal cover.
These strings of lozenges come naturally equipped with an indexation by $\N$.
The indexation is given by a {\em chain of lozenges} when lifted to the
universal cover (see definition in the next section).
The index is given by placement of the lift of the
periodic orbit as a corner of a lozenge
in this infinite chain of lozenges.

The second result (Lemma \ref{lem:distance_greater_Ai}) says that orbits
 inside a string of orbits, when lifted to the universal cover, are 
at least linearly far apart with respect to the indexation.

These two results, while not technically difficult given what is already known, are what allows us to bring in geometry into the picture. Using (either directly or indirectly) some hyperbolic properties of the metric inside a piece of the JSJ decomposition, we can obtain a lower bound on the growth of the period of orbits inside a string of orbits. Using the fact that orbits in a string  are also \emph{at most} linearly far apart (Lemma \ref{lem:upper_bound_distance_in_string}), we also get an upper bound.
\begin{thmintro} \label{thmintro:period_growth}
Let $\{\alpha_i\}_{i \in \N}$ be an infinite string of orbits, with the indexation chosen so that $\alpha_0$ is one of
the shortest orbit in the conjugacy class. Then the growth of the period is at least:
\begin{enumerate}
 \item Exponential in $i$ if the manifold is hyperbolic;
 \item Quadratic in $i$ if the $\{\alpha_i\}_{i \in \N}$ intersect an atoroidal piece;
 \item Linear in $i$ if $\{\alpha_i\}_{i \in \N}$ goes through two consecutive Seifert-fibered pieces.
\end{enumerate}

Moreover, the growth of the period is at most exponential in $i$, independently of the topology of $M$.
\end{thmintro}

Notice that for the lower bound, there is a very strong dependence on the geometric type of the manifold or the piece 
of the JSJ decomposition.

This result can then be translated in terms of a counting result inside free homotopy classes thanks to Proposition \ref{prop:free_homotopy_class_to_strings}:

\begin{thmintro} \label{thmintro:counting_orbits}
Let $\flot$ be an Anosov flow on a $3$-manifold $M$, and $\FH(\alpha_0)$ be the free homotopy class of a closed orbit $\alpha_0$ of $\flot$.
\begin{enumerate}
 \item If $M$ is hyperbolic, then there exists a uniform constant $A_1>0$ and a constant $C_{1}$ depending on $\FH(\alpha_0)$ (or equivalently on 
$\alpha_0$) such that, for $t$ big enough,
\begin{equation*}
 \sharp \{ \alpha \in \FH(\alpha_0) \mid l(\alpha) <t \} \leq A_1 \log(t) + C_{1}.
\end{equation*}
 \item If the JSJ decomposition of $M$ is such that no decomposition torus bounds a Seifert-fibered piece on both sides (so in particular, if all the pieces are atoroidal), then there exists a constant $C_{1}$ depending on $\FH(\alpha_0)$ such that, for $t$ big enough,
\begin{equation*}
 \sharp \{ \alpha \in \FH(\alpha_0) \mid l(\alpha) <t \} \leq  C_{1}\sqrt{t}.
\end{equation*}

 \item Otherwise, there exist constants $A_1>0$ and $B_1\geq 0$, such that, for $t$ big enough, 
\begin{equation*}
 \sharp \{ \alpha \in \FH(\alpha_0) \mid l(\alpha) <t \} \leq  A_1t + B_1.
\end{equation*}
Furthermore, if $M$ is a graph manifold, then $A_1$ and $B_1$ can be chosen independently of $\FH(\alpha_0)$.
\end{enumerate}
So, in any case, the growth of the number of orbits inside a free homotopy class is at most linear in the period --- but a priori with constants depending on the particular free homotopy class.

Moreover, independently of the topology of $M$, the growth of the number of orbits inside an infinite free homotopy class is at least logarithmic in the period. More precisely, there exists a uniform constant $A_2>0$ and a constant $C_{2}$ depending on $\FH(\alpha_0)$ such that, if $\FH(\alpha_0)$ is infinite, then for any $t$

\begin{equation*}
 \sharp \{ \alpha \in \FH(\alpha_0) \mid l(\alpha) <t \} \geq  \frac{1}{A_2} \log(t) - C_{2}.
\end{equation*}
\end{thmintro}

Theorem \ref{thmintro:counting_orbits} is not yet enough to get Theorem \ref{thmintro:exponential_growth_conjugacy_classes}. Indeed, we need to show that we have a \emph{uniform} upper bound for the rate of growth of
number of orbits inside a free homotopy class with respect to the
period. That is, we need the constants in the previous theorem to be independent of the chosen free homotopy class. We manage to do that, at the cost of getting worse rates of growth:
\begin{thmintro} \label{thmintro:Uniform_control_growth_rate}
Let $\flot$ be an Anosov flow on a $3$-manifold $M$. There exist constants $A_1, A_2, A_3>0$ and $t_0>0$, depending only on the flow and $M$, such that for any periodic orbit $\alpha_0$
\begin{enumerate}
 \item If $M$ is a graph manifold, then for $t >t_0$,
\begin{equation*}
 \sharp \{ \alpha \in \FH(\alpha_0) \mid l(\alpha) <t \} \leq  A_1t.
\end{equation*}

 \item If $M$ is hyperbolic, then for $t >t_0$,
\begin{equation*}
 \sharp \{ \alpha \in \FH(\alpha_0) \mid l(\alpha) <t \}\leq A_2\sqrt{t} \log(t)
\end{equation*}

 \item Otherwise, for $t >t_0$,
\begin{equation*}
 \sharp \{ \alpha \in \FH(\alpha_0) \mid l(\alpha) <t \}\leq  A_3\sqrt{t} e^{\frac{\sqrt{t}}{2} \log (t)}
\end{equation*}
\end{enumerate}

So, independently of the topology of $M$, we always have, for $t >t_0$,
\begin{equation*}
 \sharp \{ \alpha \in \FH(\alpha_0) \mid l(\alpha) <t \} \leq  A_3\sqrt{t} e^{\frac{\sqrt{t}}{2} \log (t)}.
\end{equation*}
\end{thmintro}

The first two results of 
Theorem \ref{thmintro:exponential_growth_conjugacy_classes} are an almost direct corollary of Theorem \ref{thmintro:Uniform_control_growth_rate}, thanks to the counting results of Bowen \cite{Bowen:periodic_orbits} and Margulis \cite{MargulisThesis}.
To prove the equidistribution result, i.e., that the measure of maximal entropy can be obtained by a limit of 
sums of measures supported on the shortest orbit in a free homotopy class, we also use a result of Kifer \cite{Kifer:Large_deviations94} (following an idea of Babillot and Ledrappier \cite{BabillotLedrappier}).

We stress that the last three theorems establish a deep connection between counting orbits 
in infinite free homotopy classes and the JSJ decomposition of the
manifold. In addition, they establish a connection with the particular
geometry in the pieces of the JSJ decomposition. It is worth noting
that this is the first instance establishing such a connection.
This relationship does not appear in the aforementioned counting results in homology
classes and general counting of orbits of Anosov flows.
It follows that the results of this paper establish an important
connection between counting results and entropy on the one
hand and the topology and geometry of the $3$-manifold on the
other hand.

Finally, we also deduce from Theorem \ref{thmintro:period_growth} the following result about quasigeodesicity of $\R$-covered Anosov flows, which generalizes a result of the second author in \cite{Fen:AFM}.
\begin{thmintro} \label{thmintro:not_quasigeodesic}
 Let $\flot$ be a $\R$-covered Anosov flow on a closed $3$-manifold $M$. If $M$ admits an atoroidal piece in its JSJ decomposition, i.e., if $M$ is not a graph-manifold, then $\flot$ is \emph{not} quasigeodesic.
\end{thmintro}
We also conjecture that all the $\R$-covered Anosov flows on graph-manifolds are quasigeodesic flows.

\subsection{Structure of the paper}

In section \ref{sec:background_prelim}, we cover the background material 
needed for
this article about Anosov flows and their topology.
We also prove some new results describing free homotopy classes that are essential for the rest of the article. In particular, we prove the key results Proposition \ref{prop:free_homotopy_class_to_strings} and Lemma \ref{lem:distance_greater_Ai}.

In section \ref{sec:flows_toroidal_manifolds} we describe the Foulon--Hasselblatt surgery and use it to construct a number of contact Anosov flows on manifolds with all possible types of JSJ decompositions.

In section \ref{sec:classifying_flows_free_homotopy}, we study what the existence or nonexistence of infinite free homotopy classes implies for the topology of the manifold. In particular, we prove Theorems \ref{thmintro:finite_hom_class}, \ref{thmintro:non_algebraic_finite_free_homotopy} and \ref{thmintro:no_periodic_piece} (Theorem \ref{thm:finite_homotopy_class}, 
Corollaries \ref{cor:totally_periodic} and \ref{thm:no_periodic_piece} respectively).

Section \ref{sec:more_examples} describes how one can use recent work of B\'eguin, Bonatti, and Yu \cite{BBY} to build \emph{non} $\R$-covered Anosov flows on manifolds with all possible types of
JSJ decompositions and both finite and infinite free homotopy classes.

Theorem \ref{thmintro:period_growth} (Theorems \ref{thm:upper_bound_length_growth} and \ref{thm:length_growth}) is then proved in section \ref{section:period_growth_strings}, as well as more precise results giving some explicit control of the constants. The proof of that result is split over three subsections (\ref{subsec:Hyperbolic_case}, \ref{subsec:atoroidal_piece}, and \ref{subsec:two_Seifert_pieces}), one for each topological type.

In section \ref{section:consequences}, we derive the consequences of Theorem \ref{thmintro:period_growth}. That is, we first derive Theorem \ref{thmintro:counting_orbits} (Theorem \ref{thm:counting_orbits}) and Theorem \ref{thmintro:Uniform_control_growth_rate} (Theorem \ref{thm:Uniform_control_growth_rate}). We then explain how to use the latter to finally prove Theorem \ref{thmintro:exponential_growth_conjugacy_classes} (Theorem \ref{thm:counting_conjugacy_classes}, Corollaries \ref{cor:counting_conjugacy_equal_topological_entropy} and \ref{cor:equidistribution}).

Finally, in section \ref{sec:quasigeodesic}, we obtain Theorem \ref{thmintro:not_quasigeodesic} (Theorem \ref{thm:quasigeodesic}) as yet another consequence of the work done in section \ref{section:period_growth_strings}.

\section{Background and preliminary results} \label{sec:background_prelim}

\subsection{Generalities on Anosov flows}

An Anosov flow is defined as follows

\begin{definition} \label{def:Anosov}
 Let $M$ be a compact manifold and $\flot \colon M \rightarrow M$ a $C^{1}$ flow on $M$. The flow $\flot$ is called Anosov if there exists a splitting of the tangent bundle ${TM =  \R\cdot X \oplus E^{ss} \oplus E^{uu}}$ preserved by $D\flot$ and two constants $a,b >0$ such that:
\begin{enumerate}
 \item $X$ is the generating vector field of $\flot$;
 \item For any $v\in E^{ss}$ and $t>0$,
    \begin{equation*}
     \lVert D\flot(v)\rVert \leq be^{-at}\lVert v \rVert \, ;
    \end{equation*}
  \item For any $v\in E^{uu}$ and $t>0$,
    \begin{equation*}
     \lVert D\phi^{-t}(v)\rVert \leq be^{-at}\lVert v \rVert\, .
    \end{equation*}
\end{enumerate}
In the above, $\lVert \cdot \rVert$ is any Riemannian (or Finsler) metric on $M$.
\end{definition}

Clearly this definition makes sense for $M$ of any dimension and $E^{ss}, E^{uu}$ of
any positive dimension. The results of this article deal with $M$ of dimension $3$,
so we restrict to this dimension from now on.

The subbundle $E^{ss}$ (resp. $E^{uu}$) is called the \emph{strong stable distribution} (resp.
\emph{strong unstable distribution}). It is a classical result of 
Anosov (\cite{Anosov}) that $E^{ss}$, $E^{uu}$, $\R\cdot X \oplus E^{ss}$ and $\R\cdot X \oplus E^{uu}$ are integrable
and are continuous. We denote by $\mathcal{F}^{ss}$, $\mathcal{F}^{uu}$, $\fs$ and $\fu$ the respective foliations and we call them the strong stable, strong unstable, stable and unstable foliations.

All of these foliations, as well as the flow, lift to the universal cover $\wt M$ of $M$, and we denote the lifts by $\hflot$, $\hfs$, $\hfu$, $\wt{\mathcal{F}}^{ss}$ and $\wt{\mathcal{F}}^{uu}$.
We then define the orbit and leaf spaces of the flow in the following way:
\begin{itemize}
 \item The \emph{orbit space} of $\flot$ is the quotient space of $\M$ by the relation ``being on the
same orbit of $\hflot$ ''. We denote it by $\orb$.
 \item The \emph{stable} (resp. \emph{unstable}) \emph{leaf space} of $\flot$ is the quotient of 
$\M$ by the relation ``being on the same leaf of $\hfs$ (resp. $\hfu$)''. We denote them by $\leafs$ and $\leafu$ respectively.
\end{itemize}
The stable and unstable foliations project to two transverse $1$-dimensional foliations of the orbit space $\orb$. We will keep the same notations for the foliations on $\orb$ or on $\wt M$ and hope it will not be the source of any confusion.

The orbit space $\orb$ is always homeomorphic to $\R^2$ \cite{Bar:CFA,Fen:AFM}, but in general
the leaf spaces are 
not Hausdorff. 
So the leaf spaces are examples of simply connected
non-Hausdorff $1$-manifolds.
Therefore we make the following 

\begin{definition}
 An Anosov flow is called $\R$-covered if its stable leaf space $\leafs$, or equivalently, its unstable leaf space $\leafu$ is homeomorphic to $\R$.
\end{definition}
The proof that one leaf space being Hausdorff implies that the other is can be found in \cite{Bar:CFA,Fen:AFM}.

A very important fact about $\R$-covered flows is the following:
\begin{theorem}[Barbot \cite{Bar:CFA}]
Let $\flot$ be a $\R$-covered Anosov flow on a closed $3$-manifold $M$.
\begin{itemize}
 \item either no leaf of $\hfs$ intersect every leaf of $\hfu$ (and vice-versa),
 \item or $\flot$ is orbit equivalent to a suspension of an Anosov diffeomorphism
\end{itemize} 
\end{theorem}

Suppose that $\phi^t$ is $\R$-covered but not orbit equivalent
to a suspension. We say in that case that $\phi^t$ is \emph{skewed}.
The previous result implies that 
the structure of the orbit space and the stable and unstable foliations are particularly nice: 
Start with a leaf $\lambda^s \in \leafs$. Then the set 
\[
I^u(\lambda^s) \ :=  \ \lbrace \lambda^u \in \leafu \mid \lambda^u \cap \lambda^s \neq \emptyset\rbrace
\]
is an open, non-empty, connected and bounded set in $\leafu \simeq \R$. Hence it admits an upper and lower bound. Let $\eta^s(\lambda^s) \in \leafu$ be the upper bound and $\eta^{-u}(\lambda^s)\in \leafu$ be the lower bound. Similarly, for any $\lambda^u \in \leafu$, define $\eta^u(\lambda^u)$ and $\eta^{-s}(\lambda^u)$ as, respectively, the upper and lower bounds in $\leafs$ of the set of stable leafs that intersects $\lambda^u$. We have the following result (see Figure \ref{fig:R-covered_case}):
\begin{proposition}[Fenley \cite{Fen:AFM}, Barbot \cite{Bar:CFA,Bar:PAG}] \label{prop:eta_s_eta_u}
 Let $\flot$ be a skewed $\R$-covered Anosov flow in a $3$-manifold $M$,
where $\fs$ is transversely orientable. Then, the functions $\eta^s \colon \leafs \rightarrow \leafu$ and $\eta^u \colon \leafu \rightarrow \leafs$ are H\"older-homeomorphisms and $\pi_1(M)$-equivariant. We have $(\eta^u)^{-1} = \eta^{-u}$, and $(\eta^s)^{-1} = \eta^{-s}$.
 Furthermore, $\eta^u \circ \eta^s$ and $\eta^s \circ \eta^u$ are strictly increasing homeomorphisms and we can define $\eta \colon \orb  \rightarrow \orb$ by
\begin{equation*}
\eta(o):= \eta^u \left( \hfu(o)\right) \cap \eta^s\left(\hfs(o) \right) .
\end{equation*}
\end{proposition}

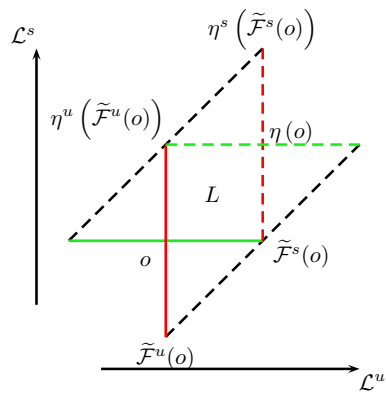
\begin{figure}[h]
\begin{center}
\scalebox{0.85}{
\begin{pspicture}(-0.5,-0.5)(6,6)
\psline[linewidth=0.04cm,linestyle=dashed](2,0.5)(5,3.5)
\psline[linewidth=0.04cm,linestyle=dashed](0.5,2)(3.5,5)
\psline[linewidth=0.04cm,arrowsize=0.05cm 2.0,arrowlength=1.4,arrowinset=0.4]{->}(1,0)(5,0)
\psline[linewidth=0.04cm,arrowsize=0.05cm 2.0,arrowlength=1.4,arrowinset=0.4]{->}(0,1)(0,5)
\rput(5.2,-0.2){$\leafu$}
\rput(-0.2,5.2){$\leafs$}
 \psline[linewidth=0.04cm,linecolor=green](0.5,2)(3.5,2)
\rput(4.1,1.8){$\hfs(o)$}
 \psline[linewidth=0.04cm,linecolor=red](2,0.5)(2,3.5)
\rput(2,0.2){$\hfu(o)$}
\put(1.6,1.6){$o$}
\psline[linewidth=0.04cm,linecolor=red,linestyle=dashed](3.5,2)(3.5,5)
 \rput(3.5,5.3){$\eta^s\left(\hfs(o)\right)$}
\psline[linewidth=0.04cm,linecolor=green,linestyle=dashed](2,3.5)(5,3.5)
\rput(1.1,3.9){$\eta^u\left(\hfu(o)\right)$}
\put(3.6,3.6){$\eta\left(o\right)$ }
\put(2.6,2.6){$L$}
\end{pspicture}}
\end{center}
 \caption{The orbit space in the $\R$-covered case} \label{fig:R-covered_case}
\end{figure}

If $\fs$ is not transversely orientable the homeomorphisms $\eta^s, \eta^u$ are 
twisted $\pi_1(M)$-equivariant.

When Anosov flows are non $\R$-covered flows, we can make the following definition:
\begin{definition} \label{def:branching_leaf}
A leaf of $\hfs$ or $\hfu$ is called a non-separated leaf if it is non-separated
in its respective leaf space ($\leafs$ or $\leafu$)
from a distinct leaf.
 A leaf of $\fs$ or $\fu$ is called a \emph{branching} leaf if it is the projection of a non-separated leaf 
in $\hfs$ or $\hfu$.
\end{definition}

Non $\R$-covered Anosov flows are generally more complicated than $\R$-covered ones, but we do have the following nice result, that will be quite useful for us:
\begin{theorem}[Fenley \cite{Fen:SBAF}, Theorem F] \label{thm:finite_branching}
 For any Anosov flow on a $3$-manifold, there are only a finite number of branching leaves.
\end{theorem}

\subsection{Modified JSJ decompositions} \label{section:good_JSJ}

A $3$-manifold $M$ is irreducible if every embedded sphere bounds a ball \cite{He}.
A fundamental result of Jaco-Shalen and Johannson states that $3$-manifolds are decomposed
into simple pieces. A $3$-manifold $N$ is {\em Seifert fibered} if it has a foliation
by circles \cite{He,Ep}. A $3$-manifold $N$ is {\em atoroidal} if every $\pi_1$-injective
map from the torus $f: T^2 \rightarrow N$ is homotopic into the boundary of $N$.
A Seifert manifold usually has many $\pi_1$-injective tori that are not homotopic to
the boundary, so these two types of manifolds are opposites.
The Jaco-Shalen-Johannson decomposition theorem also called the JSJ, or torus decomposition states the 
following:

\begin{theorem} Let $M$ be a compact, irreducible, orientable $3$-manifold. Then
there is a finite collection $\{ T_j \}$ of $\pi_1$-injective, embedded tori which
cut $M$ into
pieces $\{ P_i \}$ so that 
the closure of each component $P_i$ of $M - \cup T_j$ is either Seifert fibered or atoroidal.
In addition except for a very small and completely specified class of simple 
manifolds, the decomposition (in other  words the $T_j$ or the $P_i$  up to isotopy)
is unique if the collection $\{ T_j \}$ is minimal.
\end{theorem}

Any manifold supporting an Anosov flow is irreducible \cite{Ro}. However it may not
be orientable, but we will be able to lift to an orientable double cover as explained
later.

\begin{definition}{(Birkhoff annulus)}{}
Let $\phi^t$ be an Anosov flow in $M^3$. A Birkhoff annulus is an a priori
only immersed annulus $A$, so that the interior of $A$ is transverse to the
flow and the boundary of $A$ is a union of orbits of the flow (possibly 
the same orbit).
\end{definition}

A $\pi_1$-injective, a priori only immersed
 torus $T$ in $M$ is said to be \emph{quasi-transverse} to the flow $\phi^t$
if $T$ is a finite  union of Birkhoff annuli.
An embedded $\pi_1$-injective torus is always homotopic to one that is
either transverse or to a quasi-transverse torus that is 
{\em weakly embedded} \cite[Theorem 6.10]{BarbotFenley1}. Weakly embedded means that the torus is embedded outside the
tangent orbits. 
Unless the flow $\phi^t$ is orbit equivalent to a suspension then the torus is
always homotopic to a quasi-transverse torus. 
Moreover, almost always this quasi-transverse torus is unique up to homotopy along the orbits of $\flot$ and
unique up to isotopy in the complement of the tangent orbits.
In particular the tangent orbits are completely determined by the isotopy class
of the torus.
There is a special case when there is more than one (\cite[Lemma 5.5]{BarbotFenley1}), in which case the torus 
is associated
with a scalloped region (see Definition \ref{def:scalloped} in section \ref{subsec:lozenges}). In this case up to flow homotopy there are exactly two Birkhoff
tori homotopic to $T$.
This is related to the study of ${\Z}^2$-invariant chains of lozenges, as will be explained further on in this article.
In this case the torus is homotopic to two essentially distinct Birkhoff tori, in
particular the boundary orbits are not the same for the two tori. In addition the torus is then also isotopic to another torus that is transverse to the
flow.
So, in summary, the following result describes what we call the modified JSJ decomposition.
\begin{theorem}[\cite{BarbotFenley1}, Sections 5 and 6] 
Let $\flot$ be an Anosov flow in $M$ orientable, which is not orbit equivalent to a suspension Anosov flow.
Let $\{ T'_j \}$ be a collection
of disjoint, embedded tori given by the JSJ decomposition theorem.
Then each torus $T'_j$ is homotopic to a weakly embedded quasi-transverse torus $T_j$. %theorem 6.10 
In case $T_j$ is not unique up to flow homotopy then $T'_j$ is also isotopic to a transverse
torus, which will then be denoted by $T_j$. 

Moreover, the collection $\{ T_j \}$ is also {\em weakly embedded}, that is, embedded outside
the union of the orbits tangent to the $T_j$ that are quasi-transverse.

With these choices the tori $T_j$ are unique up to flow homotopy and
unique up to flow isotopy outside the tangent orbits.
The closure of the complementary components $P_i$ of $\cup T_j$ are called the 
\emph{pieces of the modified JSJ decomposition.}

Furthermore, if $P_i$ is not a manifold, then there are arbitrarily small 
neighborhoods of $P_i$ that are representatives of the corresponding 
piece $P'_j$ of the torus decomposition of $M$.
\end{theorem}
The fact that $P_i$ may not always be a submanifold is due to the possible collapsing of tangent orbits in the union of the tori $T_j$. For example it could be that
two distinct ``boundary'' components $T_j$ and $T_k$ of $P_i$ have a common
tangent orbit $\gamma$ (and this is quite common as can be seen in \cite{BarbotFenley1}). Then, along $\gamma$, the piece
$P_i$ is not a manifold with boundary, since two ``sheets'' of the boundary 
of $P_i$ intersect at $\gamma$.

In addition, notice that to ensure the flow uniqueness of the $T_j$, we need to choose
the transverse tori in the case that there are two essentially distinct
quasi-transverse tori homotopic to a given $T'_j$.

Let us now describe how the flow intersects a piece $P_i$: An orbit intersecting $\partial P_i$ intersects
it either in the tangential or transverse part of $\partial P_i$. If it is tangent then it
is \emph{entirely} contained in $T_j$ for some $j$ and so entirely contained 
in $\partial P_i$. Otherwise it either enters or exits $P_i$. Hence the fact that
$P_i$ may not be a manifold only affects the orbits that are entirely contained in
$\partial P_i$.

Throughout the article we will use modified JSJ decompositions.

\vskip .1in
\noindent
{\bf {Setup}} $-$
All the counting questions we consider in this article are left unmodified by passing to finite covers, modulo
changing some of the constants. 
This will be carefully explained later on in the article.
Therefore
 we will always implicitly take a finite cover where $M$ is orientable
 if necessary, and we will mostly consider modified JSJ decompositions.
\vskip .1in

\begin{convention}
 We invariably think of $\gamma$ in $\pi_1(M)$ as both
a covering translation of $\widetilde M$ and as a homotopy class of curves in $M$.
\end{convention}

\begin{definition}[intersecting a piece, crossing a piece]
Let $P$ be a piece of the torus decomposition of $M$ and let $P'$ be an associated
piece of a modified JSJ decomposition. We say that a periodic orbit 
$\alpha$ intersects $P$ if $\alpha$ is either a tangent orbit in $P'$ or if
it intersects $\partial P'$ transversely.
In the second case we in addition say that $\alpha$ crosses the piece $P$.
We may also refer to this as $\alpha$ intersects or crosses $P'$, the associated
piece of the modified JSJ decomposition.
\end{definition}

Notice that $P$ is defined up to isotopy and $P'$ is defined up to homotopy along
flow lines and isotopy outside the tangent orbits. Therefore $\alpha$ intersects
$P$ or crosses $P$ independently of the particular modified JSJ
representative $P'$ and depends only 
on the isotopy class of $P$.

\begin{definition}[periodic piece, free piece]
With respect to an Anosov flow, a Seifert fibered piece $S$ of the torus
decomposition of $M$ can have one of two possible behaviors:
\begin{itemize}
 \item Either there exists a Seifert fibration of $S$ and up to powers 
there exists a periodic orbit in $M$ which is freely homotopic to a regular fiber of $S$
in this Seifert fibration; in which case the piece
is called \emph{periodic};
 \item Or no periodic orbit is freely homotopic to a regular fiber (even up to powers)
of any Seifert fibration of $S$;
and the piece $S$ is then called \emph{free}. 
\end{itemize}
Note that, $S$ is periodic if and only if there is a Seifert fibration of $S$ such
that if $h\in \pi_1(S)$ represents a regular fiber of $S$, 
then $h$ does not act freely in at least one 
of the leaf spaces of stable/unstable leaves in $\wt M$.
\end{definition}

The most classical example of a free piece is the geodesic flow of a negatively curved surface. More generally, all the flows that we will construct in section \ref{sec:flows_toroidal_manifolds} 
are free on each of their Seifert-fibered pieces thanks to Barbot's result in 
\cite{Barbot:VarGraphees}.
But Anosov flows with periodic pieces are far from uncommon either. 
A lot of examples were constructed and studied in \cite{BarbotFenley1,BarbotFenley2} by Barbot and the second author.

The added technicality in the statement about some Seifert fibration is that in
some exceptional cases there is more than one Seifert fibration in $S$. This happens
non trivially 
for example if $S$ is a twisted $I$-bundle over the torus or the Klein bottle.

\vskip .1in

We will later (in particular in section \ref{section:period_growth_strings}) need the following two lemmas that describe the connection between free homotopy classes and a modified JSJ decomposition. Throughout this article, we use the following convention for our definition of freely homotopic orbit:

\begin{convention} \label{convention_free_homo}
 We say that two orbits $\alpha$ and $\beta$ of a flow on $M$ are \emph{freely homotopic} if they are homotopic as non-oriented curves in $M$. In other words, if $g\in \pi_1(M)$ is a representative of the orbit $\alpha$, then $\alpha$ and $\beta$ are freely homotopic if and only if $\beta$ is represented by $g^{\pm 1}$.

Up to powers this
 is also equivalent to saying that there exist lifts $\wt \alpha$ and $\wt \beta$ of $\alpha$ and $\beta$ to $\wt M$ such that $g$ stabilizes $\wt \alpha$ and $\wt \beta$.
\end{convention}

\begin{lemma} \label{lem:cutting_orbits_in_pieces}
Let $\flot$ be an Anosov flow on an orientable  $3$-manifold $M$. Let $M = \cup_j N_j$ be a 
modified JSJ decomposition. Let $\alpha_0$ be a periodic orbit and $\FH(\alpha_0)$ its free homotopy class. 
Suppose that some orbit $\beta \in \FH(\alpha_0)$ cross a piece $N_k$.
Then all the orbits $\alpha \in \FH(\alpha_0)$ also cross  $N_k$.

In addition, if there exists a connected component $\beta_{1}$ of $\beta \cap N$ 
between two boundary torus $T_1$ and $T_2$ (where we also allow $T_1=T_2$), then, 
for any $\alpha \in \FH(\alpha_0)$, there exists a connected component $\alpha_{1}$ of 
$\alpha \cap N$ between $T_1$ and $T_2$ that is in the same free homotopy class as $\beta_{1}$
modulo boundary.

Furthermore, the free homotopy between two segments of orbits can always be realized inside the pieces of the decomposition that the orbits crosses.
\end{lemma}

\begin{proof}
 Given the modified JSJ decomposition $M=  \cup_j N_j$, we construct the graph $G$ dual to it in the following way:
\begin{itemize}
 \item The vertices $v_j$ corresponds to the interior of $N_j$.
 \item Two vertices $v_i, v_j$ are joint by an edge if $N_i$ and $N_j$ share a common torus boundary.
\end{itemize}

This graph $G$ lifts to a tree $\wt G$ that is dual to the lift of the 
JSJ decomposition of $M$ to its universal cover $\wt M$. The vertices of $\wt G$ are copies of the universal cover of some $N_j$, and the edges corresponds to lifts of the decomposition tori.
This graph is used a lot in $3$-manifold theory \cite{He}.

The piece $N_k$ crossed by $\beta$ is fixed throughout the proof.

Since $\FH(\alpha_0)$ represents a free homotopy class, there exist a coherent lift $\wt{\FH}(\alpha_0)$ of $\FH(\alpha_0)$ to the universal cover such that the collection of orbits $\wt{\alpha} \in \wt{\FH}(\alpha_0)$ is exactly  the set of orbits of $\widetilde \phi^t$ that are individually left invariant by the same element $\gamma\in \pi_1(M)$.
We claim that the action of $\gamma$ on $\wt G$ is of one of the following types:
\begin{enumerate}[(i)]
 \item Either $\gamma$ acts freely on the tree $\widetilde G$ and $\gamma$
acts as a translation on an
unique axis.
 \item Or $\gamma$ acts freely on the set of vertices of $\wt G$ but fixes an edge, and moreover this fixed edge is unique.
 \item Or $\gamma$ has fixed vertices in $\wt G$, but does not leave invariant
three consecutive edges forming a linear subtree of $\wt G$.
\end{enumerate}

The general theory of group actions on trees (see for instance \cite{Serre}) states that there are three possibilities:
1) $\gamma$ acts freely on $\wt G$ and so $\gamma$ has a unique axis where it acts
as a translation, 2) $\gamma$ does not fix any vertex of $\wt G$, 
but leaves invariant an edge, so $\gamma^2$ fixes at least two points (this is what is
called an ``inversion of an edge") --- clearly the invariant edge is unique, 3) $\gamma$
has a fixed point. Given this, 
the only thing that needs to be justified in the above classification 
is case (iii). We have to show that $\gamma$ cannot fix 3 edges of $\wt G$ forming a
linear subtree. Suppose by way of contradiction that this is not true and let
$\wt T_1, \wt T_2, \wt T_3$ denote the lifts of the tori corresponding to the
edges in question. Let $\wt N_a, \wt N_b$ be the lifts of the pieces of the JSJ
decomposition corresponding to the vertices between $\wt T_1, \wt T_2$ and
$\wt T_2, \wt T_3$ respectively.
Projecting to $M$ this means that $\gamma$ has a representative in the torus
$T_1$ (projection of $\wt T_1$ to $M$) and also $T_2$.
This implies that in $N_a$ there is a cylinder or annulus from $T_1$ to $T_2$. Notice that
$T_1$ may be the same torus as $T_2$, but in this case, the annulus cannot be homotoped into $T_1$ 
or $\wt T_1$ would be equal to $\wt T_2$. In other words there is an essential annulus in 
$N_a$.
If the piece $N_a$ is atoroidal then it is acylindrical. This is because $N_a$ is in
fact hyperbolic and has boundary made up of tori, hence it is acylindrical (see for instance \cite{Thurston_book}), so this cannot happen. If $N_a$ is Seifert this can only happen
if $\gamma$ is up to powers a representative of the Seifert fiber in $N_a$. In the same 
way $\gamma$ fixes $\wt T_3$ so $N_b$ is a Seifert piece and
$\gamma$ up to powers represents the Seifert fiber in $N_b$. Then, up to powers, $\gamma$
represents the regular fiber in both $N_a$ and $N_b$. This is disallowed by the minimality
requirement in the JSJ decomposition. This proves that $\gamma$ cannot leave invariant
3 edges forming a linear subtree of $\wt G$ and yields possibility (iii) above.

In order to prove the lemma,
we need to understand a bit better the projection of the class $\FH(\alpha_0)$ onto the graph.
A lift $\wt{\alpha}$ of an orbit $\alpha \in \FH(\alpha_0)$ to the universal cover projects to a path on $\wt G$.
The path is continuous, but the projection is not: when $\wt \alpha$ crosses one
 of the lifts of the $T_j$ the projection immediately goes from a point, to an edge 
(at the intersection point with $T_j$),
to a point.
Suppose that $\alpha$ intersects a torus $T$ of the modified JSJ decomposition.

If it is contained in $T$ then the projection of $\wt \alpha$ is an edge in $\wt G$.

If the intersection is transverse and $\wt T, \wt T'$ are two lifts intersected by
$\wt \alpha$ then they are distinct. This is because either $\wt T$ is transverse
to the flow, in which case this is obvious or $\wt T$ is the lift of a
quasi-transverse torus. In that case this fact is proved in \cite{BarbotFenley2}. The idea of the proof is just that if $\wt T$ and $\wt T'$ were the same, then one could build a closed transversal to say the stable foliation by following $\wt \alpha$, closing it up along $\wt T$ and making that closed path transverse to the stable foliation by pushing it along the strong unstable foliation.
%To see this, consider the curve $c$, obtained by taking the portion of $\alpha$ that runs from one intersection 
%with $T$ to another intersection, and take $\bar c$ to be a close path obtained by taking $c$ and closing it along $T$. 
%Then $\bar c$ is not homotopically trivial (indeed, $\bar c$ can be deformed to be transverse to the 
%stable foliation and hence has to be topologically non-trivial by a classical argument 
%about closed curves transverse to Reebless codimension one foliations \cite{No}).
Therefore, a lift $\wt{\alpha}$ of $\alpha$ cannot hit twice the same lift of $T$, and hence
the projection to $\wt G$ intersects an edge at most once. It follows
that this projection of $\wt{\alpha}$ has to be an infinite path.

So the projection to $\wt G$ of an orbit $\wt \alpha$
is either a vertex, or an edge for the special case of the periodic orbits tangent to a decomposition torus, 
or is an infinite path in $\wt G$.

We can now establish part of the lemma. Recall that $\gamma$ is the element of $\pi_1(M)$ that fixes all the orbits $\wt \alpha$ in the coherent lift $\wt{\FH}(\alpha_0)$.

If $\gamma$ fixes only one vertex $\wt N_i$ (corresponding to case (ii) above), 
then the orbits in $\wt{\FH}(\alpha_0)$ are all included in $\wt N_i$. 
This case is not possible as the orbit $\beta \in \FH(\alpha_0)$ crosses the piece $N_k$,
which implies that the projection of $\widetilde \beta$ to $\widetilde G$ is
an infinite path as seen above.
Similarly case (iii) cannot happen either, again by the same reason.

It follows that the 
the action of $\gamma$ on $\wt G$ is free (of type (i) above).
So all 
the $\al i$ project to the axis of $\gamma$ and the lift of the decomposition tori cuts each 
orbits in $\wt{\FH}(\alpha_0)$ into freely homotopic connected pieces.
This is because $\al i$ cannot intersect a lift of one of $\{ T_j \}$ more than once,
so the projection to $\wt G$ is {\em exactly} the axis of $\gamma$ acting 
on $\wt G$ and intersects a lift of a piece $N_j$ in a connected arc. Hence all the $\alpha\in \FH(\alpha_0)$ crosses the piece $N_k$.

Moreover, if $\beta_1$ is a segment of $\beta \cap N_k$  between two boundary tori $T_1$ and $T_2$, then a lift $\wt\beta_1$ will connect a lift $\wt T_1$ of $T_1$ to a lift $\wt T_2$ of $T_2$. Call $\wt N_k$ the lift of $N_k$ containing $\wt \beta_1$. By the argument above, any $\alpha \in \FH(\alpha_0)$ has a coherent lift $\wt \alpha$ that contains a segment intersecting $\wt T_1$ and $\wt T_2$. Call that segment $\wt\alpha_1$. Notice that this segment $\wt \alpha_1$ is uniquely determined, because $\wt \alpha$ intersects $\wt T_1$ only once.

We define $\alpha_1$ to be the projection of $\wt \alpha_1$ on $M$. All we have left to do now is to show that $\alpha_1$ is freely homotopic to $\beta_1$ relative to the boundary with a free homotopy that stays inside the piece $N_k$.

The lift $\wt N_k$ is simply connected, because the tori in the torus decomposition
are $\pi_1$-injective.
Since $\wt T_1, \wt T_2$ are path connected we can connect the points of $\wt \alpha_1, \wt \beta_1$
in $\wt T_1$ with an arc $a_1$ and similarly with an arc $a_2$ in $\wt T_2$. This produce 
a closed loop $a_1 \cup \wt \alpha_1 \cup a_2 \cup \wt \beta_1$ (we are not paying attention
to orientation along the arcs here).
Since $\wt N_k$ is simply connected this loop is null homotopic in $\wt N_k$ and
project to a closed loop in $N$ which is null homotopic in $N$.
Notice that the arcs $a_1$ and $a_2$ are well defined up to homotopy with endpoints fixed. Hence $\beta_1$ and $\alpha_1$ are freely homotopic in $N_k$ relative to the boundary.

If we keep doing this for all the other components of $\beta - \cup \{ T_j \}$, so
that at each step the arcs $a_1$ are chosen to be equal to a previously chosen
arc on the same lift $\wt T_1$, then this produces a free homotopy from any segment of $\beta$ to a corresponding segment of $\alpha$ as claimed in the lemma.

This finishes the proof of Lemma \ref{lem:cutting_orbits_in_pieces}.
\end{proof}

\begin{lemma} \label{lem_free_hom_orbits_inside_piece} 
Suppose that $\alpha$ and $\beta$ are contained in a piece $N$ of
the torus decomposition and that they are in the same free homotopy class.
Let $H$ be a free homotopy between them.
Then we can choose a free homotopy from $\alpha$ to $\beta$ entirely contained in $N$,
unless, possibly, the image of $H$ intersects a periodic Seifert piece.
\end{lemma}

\begin{proof}{}
Fix a piece $N$ of the decomposition, two orbits $\alpha$, $\beta$ in $N$ 
and a free homotopy $H$ (in $M$) between them.  Suppose that $H$ is not already contained in $N$.
As the tori in the torus decomposition are two sided,
we can choose the free homotopy to be in general position with respect to the boundary tori
of the modified JSJ decomposition. 

Let $c$ be the intersection between the image in $M$ of the free homotopy $H$ and the boundary tori of $N$. All the connected components of $c$ are closed paths on one of the boundary tori.
We first deal with the components of $c$ that are homotopically trivial on the tori. 
We can modify the homotopy $H$ in the following manner: 
For each such connected component $c_i$ of $c$, starting with the innermost (since it bounds a disk
in the torus), we consider the disc on the torus that $c_i$ bounds. 
The homotopy $H$ (or more precisely the connected component of $H$ outside of $N$ that bounds $c_i$) 
together with that disc forms a sphere. 
Since $M$ is aspherical (because the universal cover of a 
$3$-manifold supporting an Anosov flow is always $\R^3$ \cite{Ver:codim1}), 
the sphere that we obtained bounds a ball. We can hence modify $H$ by replacing its part 
inside of $c_i$ by a disc on the torus and then modifying it slightly to eliminate this intersection
with the union of the tori in the modified JSJ decomposition.
 Doing this process on all the homotopically trivial connected components of $c$ eliminates
all such intersections.

If that process removed all the connected components of $c$, then the homotopy is in $N$ and we are done. Otherwise,  
we can assume that any remaining component of $c$ is homotopically non trivial in the
particular torus. Consider a sub-annulus, outside of $N$, of the free homotopy between consecutive such
intersections, call it $A$. This annulus $A$ has image in a piece $V$ of the modified JSJ decomposition.
 If this annulus is homotopic into the boundary of $V$ we modify $H$ so that the annulus $A$ is replaced by its homotopic image inside the boundary of $V$.

If $A$ is not homotopic into the boundary, then it is an essential annulus in
$V$. As seen before this implies that $V$ is Seifert (not atoroidal) and the core of the
annulus is freely homotopic to a regular fiber, up to powers. Since this core is also
freely homotopic to a periodic orbit of $\phi^t$ this implies that $V$ is a periodic Seifert
piece. Any further sub-annulus so that no boundary is either $\alpha$ or $\beta$
has to be homotopic into that boundary of $V$ as proved in the previous lemma.

Recall that $V$ and $N$ are distinct pieces of the JSJ decomposition.
So if the free homotopy $H$ cannot be modified to be contained in $N$ it follows that
\begin{itemize}
\item
There is a homotopy (still denoted by $H$) made up of $3$ annuli:
$A_1, A, A_2$, where:

\item
 $A_1$ is a free homotopy in $N$ from $\alpha$ to 
a curve $\gamma_1$ in a boundary torus $T_1$ of $N$,

\item $A$ is a free homotopy contained in the periodic Seifert piece $V$ and

\item $A_2$ is an annulus in $N$ from a curve $\gamma_2$ contained in
a boundary component $T_2$ of $N$ to $\beta$.
\end{itemize}

Notice that both $\gamma_1$ and $\gamma_2$ are isotopic to regular fibers in their respective
tori. So if $T_1 = T_2$ then the homotopy $H$ can be modified to be entirely contained
in $N$.
Therefore we can assume that $T_1 \not = T_2$, and this finishes the proof of the lemma.
\end{proof}

\subsection{Lozenges and separation constant} \label{subsec:lozenges}

\begin{definition}
 A \emph{lozenge} $L$ in $\orb$ is an open subset of $\orb$ such that (see Figure \ref{fig:a_lozenge}):\\
There exist two points $\alpha,\beta \in \orb$ and four half leaves $A \subset \hfs(\alpha)$, $B \subset \hfu(\alpha)$, $C \subset \hfs(\beta)$ and $D \subset \hfu(\beta)$ satisfying:
\begin{itemize}
 \item For any $\lambda^s \in \leafs$, $\lambda^s \cap B \neq \emptyset$ if and only if $\lambda^s \cap D\neq \emptyset$,
 \item For any $\lambda^u \in \leafu$, $\lambda^u \cap A \neq \emptyset$ if and only if $\lambda^u \cap C \neq \emptyset$,
 \item The half-leaf $A$ does not intersect $D$ and $B$ does not intersect $C$.
\end{itemize}
Then,
\begin{equation*}
 L := \lbrace p \in \orb \mid \hfs(p) \cap B \neq \emptyset, \; \hfu(p) \cap A \neq \emptyset \rbrace.
\end{equation*}
The points $\alpha$ and $\beta$ are called the \emph{corners} of $L$ and $A,B,C$ and $D$ are called the \emph{sides}.
\end{definition}

\begin{figure}[h]
\begin{center}
\scalebox{0.85}{
\begin{pspicture}(0,-1.97)(3.92,1.97)
\psbezier[linewidth=0.04](0.02,0.17)(0.88,0.03)(0.74114233,-0.47874346)(1.48,-1.11)(2.2188578,-1.7412565)(2.38,-1.73)(3.26,-1.95)
\psbezier[linewidth=0.04](0.0,0.27)(0.96,0.3105634)(0.75286174,0.37057108)(1.78,0.77)(2.8071382,1.169429)(2.66,1.47)(3.36,1.45)
\psbezier[linewidth=0.04](0.6,1.95)(1.2539726,1.8871263)(1.1265805,1.3646309)(2.0345206,0.7973154)(2.9424605,0.23)(3.2249315,0.2543258)(3.9,0.31)
\psbezier[linewidth=0.04](0.52,-1.33)(1.48,-1.33)(1.3597014,-0.9703507)(2.3,-0.63)(3.2402985,-0.28964934)(3.14,0.05)(3.84,0.23)
\psdots[dotsize=0.16](1.98,0.85)
\psdots[dotsize=0.16](1.46,-1.11)
\usefont{T1}{ptm}{m}{n}
\rput(1.6145313,-0.06){$L$}
\usefont{T1}{ptm}{m}{n}
\rput(1.4,-1.38){$\alpha$}
\rput(2,1.2){$\beta$}

\rput(0.6,-0.6){$A$}
\rput(3,-0.6){$B$}
\rput(0.6,0.6){$D$}
\rput(3,0.6){$C$}
\end{pspicture}}
\end{center}
\caption{A lozenge with corners $\alpha$, $\beta$ and sides $A,B,C,D$} \label{fig:a_lozenge}
\end{figure}
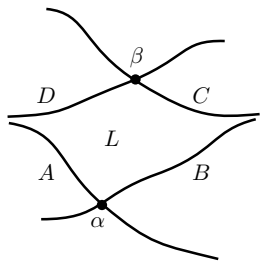

A lozenge with a periodic corner (i.e., such that its corner orbits project to periodic orbits in $M$) corresponds to a \emph{Birkhoff annulus}, that is, an annulus which is transverse to the flow, except for its two boundary components which are orbits of the flow. 
This means  that the lift of the interior of any Birkhoff annulus to $\widetilde M$
projects to a lozenge.
We say that the Birkhoff annulus projects to this lozenge.
Conversely starting with any lozenge, we can construct a Birkhoff annulus that projects to it (see \cite{Bar:MPOT}). A \emph{Birkhoff torus} is a torus obtained as a finite union of Birkhoff annuli.

\begin{definition} \label{def:chain_and_strings}
 A \emph{chain of lozenges} is a connected union of lozenges such that two consecutive lozenges always share a corner or a side.

 A \emph{string of lozenges} is a connected union of lozenges that only share corners.
\end{definition}

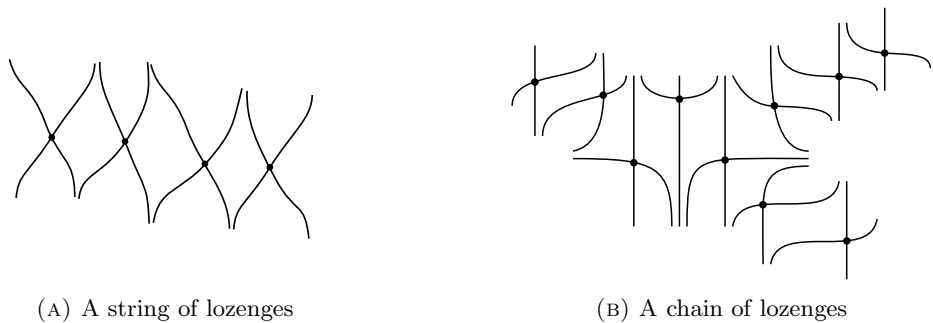
\begin{figure}[h]
\begin{subfigure}[b]{0.45\textwidth}
\centering
  \scalebox{0.55} { 
\begin{pspicture}(-3.8,-1)(4,4.4)
\rput{90}{
\psdots[dotsize=0.16](2.439,2.61)
\psbezier[linewidth=0.04](0.38,0.18)(1.24,0.04)(1.1011424,-0.46874347)(1.84,-1.1)(2.5788577,-1.7312565)(2.74,-1.72)(3.62,-1.94)
\psbezier[linewidth=0.04](0.36,0.28)(1.32,0.32056338)(1.1128618,0.38057107)(2.14,0.78)(3.1671383,1.1794289)(3.56,1.48)(4.26,1.46)
\psbezier[linewidth=0.04](0.96,1.96)(1.6139725,1.8971263)(1.4865805,1.3746309)(2.3945205,0.80731547)(3.3024607,0.24)(3.5849316,0.26432583)(4.26,0.32)
\psbezier[linewidth=0.04](0.24,-1.64)(1.16,-1.62)(1.7197014,-0.9603507)(2.66,-0.62)(3.6002986,-0.27964935)(3.5,0.06)(4.2,0.24)
\psdots[dotsize=0.16](2.34,0.86)
\psdots[dotsize=0.16](1.805,-1.05)
\psbezier[linewidth=0.04](0.98,3.46)(1.5739726,3.4171262)(1.5465806,3.1746309)(2.4545205,2.6073155)(3.3624606,2.04)(3.64,1.6)(4.22,1.58)
\psbezier[linewidth=0.04](1.0,2.06)(1.96,2.06)(1.8397014,2.4196494)(2.78,2.76)(3.7202985,3.1003506)(3.62,3.44)(4.32,3.62)
\psbezier[linewidth=0.04](0.22,-1.74)(0.8139726,-1.7828737)(0.7865805,-2.0253692)(1.6945206,-2.5926845)(2.6024606,-3.16)(2.88,-3.6)(3.46,-3.62)
\psbezier[linewidth=0.04](0.0,-3.54)(0.94,-3.4394367)(0.54,-3.22)(1.44,-2.74)(2.34,-2.26)(2.86,-2.04)(3.56,-2.06)
\psdots[dotsize=0.16](1.72,-2.605)
%  \psbezier[linewidth=0.04,linestyle=dashed,dash=0.16cm 0.16cm](0.06,-2.94)(0.54,-2.9)(0.8125543,-2.2401693)(1.7,-1.84)(2.5874457,-1.4398307)(3.08,-0.86)(4.04,-1.06)
}
\end{pspicture}
% \begin{pspicture}(-0.4,-3.8)(4.8,4)
% \psdots[dotsize=0.16](2.439,2.61)
% \psbezier[linewidth=0.04](0.38,0.18)(1.24,0.04)(1.1011424,-0.46874347)(1.84,-1.1)(2.5788577,-1.7312565)(2.74,-1.72)(3.62,-1.94)
% \psbezier[linewidth=0.04](0.36,0.28)(1.32,0.32056338)(1.1128618,0.38057107)(2.14,0.78)(3.1671383,1.1794289)(3.56,1.48)(4.26,1.46)
% \psbezier[linewidth=0.04](0.96,1.96)(1.6139725,1.8971263)(1.4865805,1.3746309)(2.3945205,0.80731547)(3.3024607,0.24)(3.5849316,0.26432583)(4.26,0.32)
% \psbezier[linewidth=0.04](0.24,-1.64)(1.16,-1.62)(1.7197014,-0.9603507)(2.66,-0.62)(3.6002986,-0.27964935)(3.5,0.06)(4.2,0.24)
% \psdots[dotsize=0.16](2.34,0.86)
% \psdots[dotsize=0.16](1.805,-1.05)
% \psbezier[linewidth=0.04](0.98,3.46)(1.5739726,3.4171262)(1.5465806,3.1746309)(2.4545205,2.6073155)(3.3624606,2.04)(3.64,1.6)(4.22,1.58)
% \psbezier[linewidth=0.04](1.0,2.06)(1.96,2.06)(1.8397014,2.4196494)(2.78,2.76)(3.7202985,3.1003506)(3.62,3.44)(4.32,3.62)
% \psbezier[linewidth=0.04](0.22,-1.74)(0.8139726,-1.7828737)(0.7865805,-2.0253692)(1.6945206,-2.5926845)(2.6024606,-3.16)(2.88,-3.6)(3.46,-3.62)
% \psbezier[linewidth=0.04](0.0,-3.54)(0.94,-3.4394367)(0.54,-3.22)(1.44,-2.74)(2.34,-2.26)(2.86,-2.04)(3.56,-2.06)
% \psdots[dotsize=0.16](1.72,-2.605)
%  \psbezier[linewidth=0.04,linestyle=dashed,dash=0.16cm 0.16cm](0.06,-2.94)(0.54,-2.9)(0.8125543,-2.2401693)(1.7,-1.84)(2.5874457,-1.4398307)(3.08,-0.86)(4.04,-1.06)
% \end{pspicture}}
}
  \caption{A string of lozenges}
\label{fig:string_lozenges} 
\end{subfigure}
% \quad
\begin{subfigure}[b]{0.45\textwidth}
\centering
\scalebox{.5} % Change this value to rescale the drawing.
{
\begin{pspicture}(0,-3.62)(11.02,3.62)
% \psgrid
\psbezier[linewidth=0.04](6.8,2.6)(6.8,1.6)(6.8,-0.2)(7.8,-0.2)
\psline[linewidth=0.04cm](5.6,1.8)(5.6,-2.2)
\psbezier[linewidth=0.04](5.8,1.8)(6.0,1.4)(6.301685,1.0393515)(6.8,1.0)(7.298315,0.96064854)(8.4,1.0)(8.4,0.6)
\psbezier[linewidth=0.04](7.0,2.6)(7.0,2.0)(7.5016847,1.8393514)(8.0,1.8)(8.498315,1.7606486)(9.6,1.8)(9.6,1.4)
\psline[linewidth=0.04cm](8.6,0.6)(8.6,3.2)
\psbezier[linewidth=0.04](8.8,3.2)(8.8,2.6)(9.301684,2.4393516)(9.8,2.4)(10.298315,2.3606486)(11.0,2.4)(11.0,2.0)
\psline[linewidth=0.04cm](9.8,1.4)(9.8,3.6)
\psbezier[linewidth=0.04](7.8,-0.4)(4.8,-0.4)(4.6,-0.2)(4.6,-2.2)
\psline[linewidth=0.04cm](4.4,1.8)(4.4,-2.2)
\psbezier[linewidth=0.04](5.4,1.8)(5.4,1.0)(3.4,1.0)(3.4,1.8)
\psline[linewidth=0.04cm](3.2,1.8)(3.2,-2.2)
\psbezier[linewidth=0.04](4.2,-2.2)(4.2,-0.6)(3.8,-0.4)(1.6,-0.4)
\psbezier[linewidth=0.04](5.8,-2.2)(6.0,-1.0)(8.4,-2.2)(8.6,-1.0)
\psbezier[linewidth=0.04](7.8,-0.6)(6.4,-0.6)(6.6,-1.1777778)(6.6,-3.2)
\psbezier[linewidth=0.04](6.8,-3.2)(7.0,-2.0)(9.4,-3.2)(9.6,-2.0)
\psline[linewidth=0.04cm](8.8,-3.6)(8.8,-1.0)
\psbezier[linewidth=0.04](3.0,1.8)(3.0,1.0)(0.8,1.4)(0.8,0.2)
\psbezier[linewidth=0.04](1.6,-0.2)(2.6,0.0)(2.4,1.6)(2.4,2.4)
\psbezier[linewidth=0.04](2.2,2.4)(2.2,1.6)(0.0,2.0)(0.0,1.0)
\psline[linewidth=0.04cm](0.6,0.2)(0.6,2.6)
\psdots[dotsize=0.2](9.8,2.4)
\psdots[dotsize=0.2](8.6,1.78)
\psdots[dotsize=0.2](6.9,1)
\psdots[dotsize=0.2](5.6,-0.44)
\psdots[dotsize=0.2](6.595,-1.62)
\psdots[dotsize=0.2](8.8,-2.58)
\psdots[dotsize=0.2](4.4,1.18)
\psdots[dotsize=0.2](3.2,-0.51)
\psdots[dotsize=0.2](2.4,1.28)
\psdots[dotsize=0.2](0.6,1.63)
\end{pspicture} 
}
\caption{A chain of lozenges} 
\label{fig:chain_of_lozenges}
\end{subfigure}
\caption{Chain and string of lozenges} \label{fig:chain_and_string_lozenges}
\end{figure}

Lozenges sharing sides are particular:
\begin{lemma} \label{lem:Lozenges_share_side}
 If two lozenges share a side, then two of their other sides are on non-separated leaves.
\end{lemma}
\begin{proof}
The two leaves abutting to the shared side are not separated since any leaf in the neighborhood of one of them is in the neighborhood of the other. See Figure \ref{fig:two_lozenges_sharing_side}. \qedhere

\begin{figure}[h]
\scalebox{0.8}{
\begin{pspicture}(0,-2.38)(8.32,2.38)
\psbezier[linewidth=0.04](0.62,-0.32)(1.16,-0.72)(1.2611424,-0.56874347)(2.0,-1.2)(2.7388575,-1.8312565)(3.26,-1.6)(3.8,-2.08)
\psbezier[linewidth=0.04,linecolor=red](0.4,-1.74)(1.32,-1.72)(1.4597017,-1.2803507)(2.48,-0.86)(3.5002983,-0.4396493)(3.34,-0.42)(3.34,0.06)
\psdots[dotsize=0.16](1.94,-1.16)
\psbezier[linewidth=0.04](3.46,0.16)(3.54,-0.4)(4.261142,-0.16874346)(5.02,-0.8)(5.778858,-1.4312565)(5.86,-1.44)(6.76,-1.64)
\psbezier[linewidth=0.04](3.88,-1.98)(3.6,-1.5)(4.28,-1.44)(5.18,-0.96)(6.08,-0.48)(6.6,-0.26)(7.3,-0.28)
\psbezier[linewidth=0.04](4.08,1.56)(4.14,0.96)(4.741142,1.3112565)(5.5,0.68)(6.258858,0.0487435)(6.34,0.04)(7.24,-0.16)
\psbezier[linewidth=0.04,linecolor=red](3.6,0.22)(4.24,-0.18)(3.88,0.28)(4.74,0.78)(5.6,1.28)(5.12,0.96)(5.76,1.4)
\psdots[dotsize=0.16](5.05,0.95)
\psdots[dotsize=0.16](5.2,-0.95)
\psbezier[linewidth=0.04,linestyle=dashed,dash=0.16cm 0.16cm](0.66,-1.96)(1.78,-1.58)(2.86,-1.2)(3.2,-0.9)(3.54,-0.6)(3.4319272,-0.32908273)(3.58,-0.18)(3.7280726,-0.030917287)(3.9420671,-0.036735725)(4.26,-0.06)(4.577933,-0.08326428)(5.06,0.74)(5.68,1.0)
\end{pspicture}}
\caption{Two lozenges sharing a side. The red leaves are not separated.}
\label{fig:two_lozenges_sharing_side}
\end{figure}
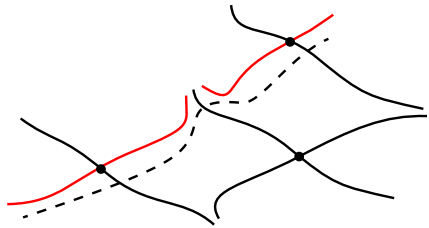

\end{proof}

It follows from Theorem \ref{thm:finite_branching}
that only a finite number of lozenges, up to deck transformations, can share a side. 

A periodic orbit can in general be the corner of anything from $0$ to $4$ lozenges, but translating the previous lemma to corners gives:
\begin{lemma}\label{lem:corner_more_than_2_lozenges}
 Suppose that $\wt \alpha$ is the corner of $3$ or $4$ lozenges, then the opposite corners are on non-separated leaves.
\end{lemma}
So up to deck transformations, there are only a finite number of orbits that can be the corner of more than 2 lozenges. 

Another fact can also limit the number of lozenges abutting to a particular orbit:
\begin{lemma} \label{lem:orbit_inside_lozenge}
 Suppose that $\wt \alpha$ is an orbit inside a lozenge $L$. Then $\wt \alpha$ is the corner of at most two lozenges.
\end{lemma}
\begin{proof}
If $\wt \alpha$ is an orbit inside a lozenge $L$, then at least two of the quadrants that the stable and unstable leaves of $\wt \alpha$ defines cannot be part of a lozenge as can be seen in Figure \ref{fig:orbit_inside_lozenge}: The quadrant containing the red leaves cannot be part of a lozenge, since otherwise two stable leaves (and two unstable leaves) would intersect. The other two quadrants can however define lozenges, as can be seen with the blue leaves in Figure \ref{fig:orbit_inside_lozenge}.
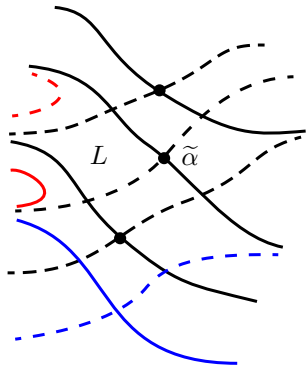
\begin{figure}[h]
 \begin{pspicture}(0,-2.38)(3.94,2.38)
\psbezier[linewidth=0.04](0.04,0.58)(0.9,0.44)(0.7611423,-0.06874346)(1.5,-0.7)(2.2388577,-1.3312565)(2.4,-1.32)(3.28,-1.54)
\psbezier[linewidth=0.04,linestyle=dashed,dash=0.17638889cm 0.10583334cm](0.02,0.68)(0.98,0.7205634)(0.7728617,0.7805711)(1.8,1.18)(2.8271382,1.5794289)(2.68,1.88)(3.38,1.86)
\psbezier[linewidth=0.04](0.62,2.36)(1.2739726,2.2971263)(1.1465805,1.7746309)(2.0545206,1.2073154)(2.9624608,0.6399999)(3.2449315,0.6643258)(3.92,0.72)
\psbezier[linewidth=0.04,linestyle=dashed,dash=0.17638889cm 0.10583334cm](0.0,-1.16)(1.22,-1.18)(1.3797014,-0.5603507)(2.32,-0.22)(3.2602985,0.1203507)(3.16,0.46)(3.86,0.64)
\psdots[dotsize=0.16](2.0,1.26)
\psdots[dotsize=0.16](1.48,-0.7)
\psbezier[linewidth=0.04,linestyle=dashed,dash=0.16cm 0.16cm](0.1,-0.34)(1.06,-0.32)(1.7266773,-0.009415701)(2.06,0.38)(2.3933227,0.7694157)(2.92,1.42)(3.76,1.44)
\psbezier[linewidth=0.04](0.28,1.58)(1.28,1.48)(1.32,0.9)(1.84,0.54)(2.36,0.18)(2.92,-0.68)(3.62,-0.82)
\psbezier[linewidth=0.04,linecolor=red,linestyle=dashed,dash=0.16cm 0.16cm](0.32,1.48)(1.1,1.22)(0.5,0.96)(0.06,0.92)
\psdots[dotsize=0.16](2.06,0.36)
\psbezier[linewidth=0.04,linecolor=red](0.12,-0.28)(0.88,-0.2)(0.26,0.24)(0.04,0.2)
\psbezier[linewidth=0.04,linecolor=blue](0.12,-0.46)(1.46,-0.82)(1.08,-2.36)(3.02,-2.36)
\psbezier[linewidth=0.04,linecolor=blue,linestyle=dashed,dash=0.16cm 0.16cm](3.56,-0.92)(2.52,-0.9)(2.08,-1.04)(1.82,-1.36)(1.56,-1.68)(0.98,-1.92)(0.1,-2.04)
\rput(2.4,0.4){$\widetilde \alpha$}
\rput(1.2,0.4){$L$}
\end{pspicture} 
\caption{An orbit in a lozenge cannot be the corner of more than two lozenges}
\label{fig:orbit_inside_lozenge}
\end{figure}
\end{proof}

\begin{definition} \label{def:scalloped}
 A \emph{scalloped} chain of lozenges is a chain of lozenges such that either each consecutive lozenges share an unstable side or each consecutive lozenges share a stable side.

A \emph{scalloped region} is a scalloped chain of lozenges together with the sides in between two consecutive lozenges.
\end{definition}

A property of freely homotopic orbits that will be essential for us is the following
\begin{proposition}[Fenley \cite{Fen:SBAF}] \label{prop:freely_hom_are_corners}
 If $\alpha$ and $\beta$ are two freely homotopic orbits then coherent lifts $\wt{\alpha}$ and $\wt{\beta}$ are corners of a chain of lozenges.
\end{proposition}
This proposition is the reason why we choose to forget the orientation when talking about free homotopy class (Convention \ref{convention_free_homo}): It is easy to see that if $\wt{\alpha}$ and $\wt{\beta}$ are the corners of a lozenge, 
then, up to powers, $\alpha = \pi(\wt\alpha)$ is represented by $g\in \pi_1(M)$ and 
$\beta = \pi(\wt\beta)$ is represented by $g^{-1}$. So, forgetting orientation allows us to have a full chain of lozenge associated to a free homotopy class instead of just half of the corners. Moreover, the difference between the number of orbits in a free homotopy class and in a oriented free homotopy class is by a factor of $2$. So it would not change any of the counting results in section \ref{section:consequences}, but would only change the constants. 

In fact the same is true if one lifts the flow to a finite cover of $M$, modulo changing the
constants involved. In particular if needed we can lift to a cover where both
$\fs$ and $\fu$ are transversely orientable (in which case $M$ is orientable
as well). Then given orbits $\alpha, \beta$ of
$\phi^t$  and integers $n, m$ not both zero so that
$\alpha^n$ freely homotopic to $\beta^m$, it follows that
$\alpha$ is either freely homotopic to $\beta$ or to $\beta^{-1}$
as {\em {oriented}} curves.

We now introduce some terminology that will be needed later on.

\begin{definition}[minimum distance, Hausdorff distance] Let $A, B$ be two disjoint closed sets in a metric 
space $Z$. 

The minimum distance between $A$ and $B$ is 
the infimum of $d(a,b)$ where $a$ is in $A$ and $b$ is in $B$.

The Hausdorff distance between $A$ and $B$ is the infimum of $r > 0$ so that $A$ is contained in 
the $r$ neighborhood of $B$ and vice versa. This infimum could be infinite.
\end{definition}

We will use later that the minimum and hence the
Hausdorff distance (for a given metric on $M$) between two corners of a lozenge is bounded below.
\begin{lemma} \label{lem:separation_constant}
 There exists $A>0$, depending only on the flow, such that, if the minimum
distance between two stable leaves $\lambda_1, \lambda_2 \in \hfs$ is less than $A$, then there exists an unstable leaf $l^u \in \hfu$ intersecting both $\lambda_1$ and $\lambda_2$. The same statement stays true with the same $A$ when switching the roles of stable and unstable.

Moreover, if the minimum distance between two orbits $\alpha$ and $\beta$ of the lifted flow $\hflot$ is less than $A$, then the stable leaf through $\alpha$ intersect the unstable through $\beta$ and vice-versa.
\end{lemma}

%\begin{proof}
This lemma is a simple consequence of the product structure of the foliations and the compactness of the manifold.
% Since $M$ is compact, there exists $A>0$, such that $M$ can be covered by flow boxes of sides $A$. 
%Lifting these flow boxes to the universal cover, we have that if the Hausdorff-distance between two stable leaves 
%$\lambda_1, \lambda_2 \in \hfs$ is less than $A$, then the two leaves intersect the same flow box, hence there 
%exists an unstable leaf intersecting both $\lambda_1$ and $\lambda_2$.
%In the same way, if two orbits have Hausdorff-distance less than $A$, then they must intersect the same flow box, hence the conclusion.
%\end{proof}

A far less obvious fact that we will also need later on is that the Hausdorff distance between two corners of a lozenge is also bounded from above (see \cite[Corollary 5.3]{Fenley:qg_PA_hyperbolic_manifolds}):
\begin{proposition}[Fenley] \label{prop:distance_corners_bounded_above}
Let $\alpha$, $\beta$ be two freely homotopic orbits such that they admits lifts $\wt\alpha$ and $\wt \beta$ that are the corners of the same lozenge.
Then, there exists $B>0$ depending only on the flow and on the manifold such that there exists an homotopy $H$ from $\alpha$ to $\beta$ that moves each point by a distance at most $B$.
\end{proposition}

\subsection{From free homotopy class to strings of lozenges} \label{subsec:strings_lozenges}

In order to obtain our counting results in section \ref{section:consequences} about free homotopy classes, we will consider some subsets of free homotopy classes which are easier to work with.

We fix some terminology first. Let $\flot$ be an Anosov flow on a $3$-manifold $M$, and $\alpha$ a closed orbit of $\flot$. Let  $\mathcal{FH}(\alpha)$ be the free homotopy class of $\alpha$. A \emph{coherent lift} of $\FH(\alpha)$ can be defined in the following way: Let $g$ be an element of the fundamental group that represents $\alpha$ (so any other element of the conjugacy class of $g$ would also represent $\alpha$). A coherent lift of $\FH(\alpha)$ is the set of all the lifts of orbits in $\FH(\alpha)$ that are invariant under $g$.
Notice that there may be distinct orbits in a coherent lift of $\FH(\alpha)$ that project to the same
orbit in $\FH(\alpha)$.

 By previous results of the second author \cite{Fen:SBAF}, recalled in Proposition \ref{prop:freely_hom_are_corners}, a coherent lift of $\mathcal{FH}(\alpha)$ to the universal cover is composed of the corners of a chain of lozenges.

\begin{definition} \label{def:string_of_orbits}
 We say that $\{\alpha_i\}_{i\in I}$ is a \emph{string of orbits} in $\mathcal{FH}(\alpha)$, if it satisfies to the following conditions:
\begin{itemize}
 \item All the $\alpha_i$ are distinct and contained in $\FH(\alpha)$;
 \item For a coherent lift of $\mathcal{FH}(\alpha)$, the orbits $\{\alpha_i\}_{i\in I}$ are the projections of the corners of a string of lozenges $\{\al i\}$ (see Definition \ref{def:chain_and_strings} above);
 \item Each $\al i$ is the corner of at most two lozenges in $\widetilde M$.
 \item Here $I$ is an interval in $\Z$, which could be finite, isomorphic to $\N$ or
	$\Z$ itself.
\end{itemize}
\end{definition}
There are several slightly different types of string of orbits:
\begin{itemize}
 \item  A string of orbits $\{\alpha_i\}$ is \emph{infinite} if it is indexed by $i\in \N$. We call it \emph{bi-infinite} if it is indexed by $\Z$.
 \item A string of orbits $\{\alpha_i\}$ is \emph{finite and periodic} if it is finite 
but the collection $\{ \alpha_i \}$ is the projection of corners of an \emph{infinite} string of lozenges.
In other words the collection $\{\al i\}_{i \in \Z}$ is infinite, but there is an element $h\in \pi_1(M)$ and a integer $k>0$ 
such that $h \cdot \al i = \widetilde \alpha_{i+k}$.  Note that all the orbits in a periodic string are non-trivially freely homotopic to themselves (up to powers).
 \item A string of orbits $\{\alpha_i\}$ is \emph{finite and non-periodic} 
otherwise. In other words the string $\{ \alpha_i \}$ is 
\emph{finite}, and it is not the projection
of an infinite string $\{\al i\}, i \in \N$.
\end{itemize}

\begin{example}
 Suppose that $\flot$ is $\R$-covered and that $\fs$ is transversely orientable.
Let $\alpha$ be a periodic orbit. Choose $\wt \alpha$ a lift of $\alpha$ and, set $\alpha_i = \pi \left( \eta^i (\wt \alpha) \right)$ (where $\eta$ is the map on the orbit space defined in Proposition \ref{prop:eta_s_eta_u}). Then $\{\alpha_i\}$ is either a finite periodic string of orbits or a bi-infinite string of orbits. 
In addition the free homotopy class of $\alpha$ is exactly the collection $\{\alpha_i\}$.
\end{example}

\begin{figure}[h]
 \scalebox{0.6}{
\begin{pspicture}(0,-4)(11.02,3.62)
% \psgrid
\psbezier[linewidth=0.04,linecolor=blue](6.8,2.6)(6.8,1.6)(6.8,-0.2)(7.8,-0.2)
\psline[linewidth=0.04cm](5.6,1.8)(5.6,-2.2)
\psbezier[linewidth=0.04,linecolor=blue](5.8,1.8)(6.0,1.4)(6.301685,1.0393515)(6.8,1.0)(7.298315,0.96064854)(8.4,1.0)(8.4,0.6)
\psbezier[linewidth=0.04,linecolor=blue](7.0,2.6)(7.0,2.0)(7.5016847,1.8393514)(8.0,1.8)(8.498315,1.7606486)(9.6,1.8)(9.6,1.4)
\psline[linewidth=0.04cm,linecolor=blue](8.6,0.6)(8.6,3.2)
\psbezier[linewidth=0.04,linecolor=blue](8.8,3.2)(8.8,2.6)(9.301684,2.4393516)(9.8,2.4)(10.298315,2.3606486)(11.0,2.4)(11.0,2.0)
\psline[linewidth=0.04cm,linecolor=blue](9.8,1.4)(9.8,3.6)
\psbezier[linewidth=0.04](7.8,-0.4)(4.8,-0.4)(4.6,-0.2)(4.6,-2.2)
\psline[linewidth=0.04cm](4.4,1.8)(4.4,-2.2)
\psbezier[linewidth=0.04](5.4,1.8)(5.4,1.0)(3.4,1.0)(3.4,1.8)
\psline[linewidth=0.04cm,linecolor=red](3.2,1.8)(3.2,-2.2)
\psbezier[linewidth=0.04,linecolor=red](4.2,-2.2)(4.2,-0.6)(3.8,-0.4)(1.6,-0.4)
\psbezier[linewidth=0.04,linecolor=green](5.8,-2.2)(6.0,-1.0)(8.4,-2.2)(8.6,-1.0)
\psbezier[linewidth=0.04,linecolor=green](7.8,-0.6)(6.4,-0.6)(6.6,-1.1777778)(6.6,-3.2)
\psbezier[linewidth=0.04,linecolor=green](6.8,-3.2)(7.0,-2.0)(9.4,-3.2)(9.6,-2.0)
\rput(2.2,-0.8){\psbezier[linewidth=0.04,linecolor=green](6.8,-3.2)(7.0,-2.0)(8.6,-3.2)(8.6,-2.0)}
\psline[linewidth=0.04cm,linecolor=green](8.8,-4)(8.8,-1.0)
\psline[linewidth=0.04cm,linecolor=green](9.8,-4)(9.8,-2)
\psbezier[linewidth=0.04,linecolor=red](3.0,1.8)(3.0,1.0)(0.8,1.4)(0.8,0.2)
\psbezier[linewidth=0.04,linecolor=red](1.6,-0.2)(2.6,0.0)(2.4,1.6)(2.4,2.4)
\psbezier[linewidth=0.04,linecolor=red](2.2,2.4)(2.2,1.6)(0.0,2.0)(0.0,1.0)
\psline[linewidth=0.04cm,linecolor=red](0.6,0.2)(0.6,2.6)
\psdots[dotsize=0.2](9.8,-3.4)
\psdots[dotsize=0.2](9.8,2.4)
\psdots[dotsize=0.2](8.6,1.78)
\psdots[dotsize=0.2](6.9,1)
\psdots[dotsize=0.2](5.6,-0.44)
\psdots[dotsize=0.2](6.595,-1.62)
\psdots[dotsize=0.2](8.8,-2.58)
\psdots[dotsize=0.2](4.4,1.18)
\psdots[dotsize=0.2](3.2,-0.51)
\psdots[dotsize=0.2](2.4,1.28)
\psdots[dotsize=0.2](0.6,1.63)
\end{pspicture} }
\caption{Three different strings of lozenges inside a chain of lozenges}
\label{fig:string_inside_chain_lozenge}
\end{figure}
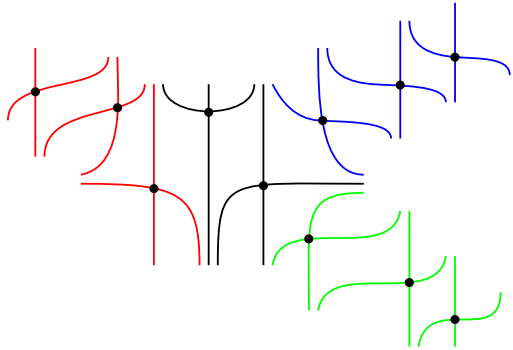

This means that if the flow $\phi^t$ is $\R$-covered, then the free homotopy classes are quite simple. In the general case however, things tend to be more complicated (see Figure \ref{fig:string_inside_chain_lozenge}). Fortunately, we have the following result.
\begin{proposition} \label{prop:free_homotopy_class_to_strings}
 Let $\alpha$ be a closed orbit of an Anosov flow on a $3$-manifold. The free homotopy class $\mathcal{FH}(\alpha)$ can be decomposed in the following way:
\begin{itemize}
 \item A finite part $\mathcal{FH}_{\textrm{finite}}(\alpha)$,
 \item A finite number of disjoint strings of closed orbits (that could be infinite, finite and periodic or just finite).
\end{itemize}
Moreover, there exists a uniform bound (i.e., depending only on the manifold and the flow) on the number of elements in $\mathcal{FH}_{\textrm{finite}}(\alpha)$. And there exists a uniform bound on the number of different strings that a free homotopy class can contain.
\end{proposition}

In fact, the statement about the uniform bounds can be made even stronger, as we will see in the proof: Except for a finite number of free homotopy classes, each free homotopy class is either a finite, infinite, or bi-infinite string of orbits. We also want to emphasize that we do not claim that there exists a uniform bound on the number of orbits inside a \emph{finite} free homotopy class (see after Theorem \ref{thm:Uniform_control_growth_rate} for a discussion of that point), but just a bound on the parts of a free homotopy class that are not strings of orbits.

The very important consequence of this proposition for this article is the following:
counting orbits inside a free homotopy class is the same, up to a change of constants, 
as counting orbits in an infinite string.

\begin{proof}
 Let $\wt{\mathcal{FH}}(\alpha)$ be a coherent lift of the free homotopy class and $g\in \pi_1(M)$ be the common stabilizer of all the lifted orbits.

 By Proposition \ref{prop:freely_hom_are_corners} the elements of $\wt{\mathcal{FH}}(\alpha)$ are all corners of a chain of lozenges. Moreover, an orbit is a corner of three or more lozenges if and only if the adjacent corners are on branching leaves (see Lemma \ref{lem:corner_more_than_2_lozenges}).

From $\wt{\mathcal{FH}}(\alpha)$ we construct a graph $(V,E)$ in the following way:
\begin{itemize}
 \item The vertices are the orbits;
 \item Two vertices are joined by an edge if they are the two corners of a lozenge.
\end{itemize}
Note that, even though we will not be using that fact here, the graph defined here is in fact a tree (\cite[Proposition 2.12]{BarbotFenley1} or \cite{Fen:SBAF}).

The stabilizer $G$ of $\widetilde{FH}(\alpha)$ contains $g$ and
acts on the graph $(V,E)$. 
With the assumption that $M$ is orientable and $\fs$ is transversely orientable,
then an element $h \in G$ has
a fixed point if and only if $h = g^n$ for some $n$.
We define the quotient graph $(V', E')$ by applying the following rules:
\begin{itemize}
 \item The vertices $v_1, v_2$ are identified if there exists $h \in \pi_1(M)$
such that $v_1 = h \cdot v_2$;
 \item Two edges are identified if their corresponding lozenges are sent
one onto the other by an element of $G$.
\end{itemize}
Note that in the new graph $(V', E')$, some edges might go from a vertex to itself.

It is now easy to see that the graph $(V', E')$ has at most a finite number of vertices of degree strictly greater than $2$. Indeed, each vertex in this graph of degree $>2$ is associated to an orbit which is a corner of at least $3$ lozenges. Hence by Lemma \ref{lem:corner_more_than_2_lozenges}, its neighboring vertices have to be on non-separated leaves, but there are only a finite number of non-separated leaves up to deck transformation (see Theorem \ref{thm:finite_branching}).

Notice that $(V',E')$ is connected. So removing all the vertices of degree $>2$ from $(V', E')$ gives a finite number of 
infinite connected components plus a finite number of finite connected components. 
Let $S$ be one of these connected components. The only way that $S$ can fail to project to a string of orbits is if some of the lozenges representing the edges in $S$ share sides. But two lozenges share sides if and only if the two opposite corners are on non-separated leaves (see Lemma \ref{lem:Lozenges_share_side}). So removing all the corners on non-separated leaves and their adjacent corners leaves only strings of orbits.

So we define $\mathcal{FH}_{\textrm{finite}}(\alpha)$ as the set of all the orbits on non-separated leaves plus their adjacent orbit, i.e., the orbits that comes from corners adjacent to the one on non-separated leaves.
Clearly, by construction, $\mathcal{FH}(\alpha) \smallsetminus \mathcal{FH}_{\textrm{finite}}(\alpha)$ consist of a finite number of strings of orbits.

The uniform bounds come from the fact that there are a finite number of branching leaves in $M$. Hence, there are only a finite number of free homotopy class that are not just a finite non-periodic, finite periodic, infinite, or bi-infinite string of orbits. The fact that we have uniform bounds on the number of different strings is therefore immediate.
\end{proof}

A particularly useful property for us is that strings of orbits that are finite and periodic are actually fairly special, in the sense that they are forced to stay in some topologically limited part of the manifold $M$:
\begin{proposition} \label{prop:finite_periodic_are tori_or_Seifert}
 Let $\{\alpha_i\}$ be a finite periodic string of orbits. Then $\{\alpha_i\}$ is a complete free homotopy class.
In addition they are either entirely contained in a Seifert piece of the modified JSJ decomposition, or are the orbits on one of the quasi-transverse decomposition tori.
\end{proposition}

\begin{proof}
 Let $\al i$ be a coherent lift of $\{\alpha_i\}$. Let $g\in \pi_1(M)$ be the common stabilizer of the $\al i$ and $h \in \pi_1(M)$ such that $h \cdot \al i = \widetilde \alpha_{i+k}$. First, applying $h^n$, $n\in \Z$, to $\alpha_0$ shows that the indexation $i$ needs to be bi-infinite, and since all the $\al i$ are, by definition, assumed to be the corners of at most two lozenges, the part $\FH_{\textrm{finite}}(\alpha_0)$ has to be empty and $\{\alpha_i\} = \FH(\alpha_0)$, which proves the first part of our claim.

Now, since $\alpha_0$ is freely homotopic to itself, there exists a $\pi_1$-injective
immersed torus that contains $\alpha_0$. Using Gabai's version of the Torus Theorem \cite{Gabai:conv_groups_are_Fuchsian}, we see that this immersed torus is either embedded or the manifold is (a special case of) Seifert-fibered. If it is the second case, we are done, and if the torus is embedded, then it can be isotoped inside a Seifert piece or to one of the modified JSJ decomposition tori (see section \ref{section:good_JSJ} or \cite{BarbotFenley1}), which proves the claim.
\end{proof}

We defined strings of orbits in no small part in order to have the following Lemma, that we will use time and time again in section \ref{section:period_growth_strings}. But before stating it we introduce the following convention that we will use for the remainder of this article since it simplifies notations for us:
\begin{convention} \label{convention1}
 If $\{\alpha_i\}$ is a finite periodic, non-periodic, infinite, or bi-infinite string of orbits, we choose the indexation 
so that $\alpha_0$ is one of the shortest orbits in the string and split the string in two so that $i$ is always taken to be non-negative.
\end{convention}

Notice that there are only finitely orbits in the string in $\FH(\alpha)_0)$ that can be the shortest 
in the string.
From now on, a string of orbit will always refer to the result of applying the convention above to a finite or infinite string of orbits.

\begin{lemma} \label{lem:distance_greater_Ai}
There exists $A>0$, depending only on the flow, such that, if $\{\alpha_i\}$ is a string of orbits and $\{\al i\}$ is a coherent lift, then, for all $i$,
\[
 d(\widetilde \alpha_0, \widetilde \alpha_i) \geq Ai.
\]
\end{lemma}
Here $d$ is the minimum distance between $\widetilde \alpha_0$ and $\widetilde \alpha_i$.
In this result, we use Convention \ref{convention1} so that we can write $i$ instead of $|i|$.
Notice that in this and in the following result we do not need to assume that $M$ is orientable
or any hypothesis on $\fs, \fu$.

\begin{proof}
 This is just a consequence of Lemma \ref{lem:separation_constant}. Let $\al i$ be a coherent lift of the $\alpha_i$. There exists a uniform constant $A>0$ such that, since the stable leaf of $\al 1$ does not intersect the unstable leaf of 
$\widetilde \alpha_0$, $d(\widetilde \alpha_1, \widetilde \alpha_0) \geq A$. Moreover, we can choose $A$ such that the minimum distance between the stable leaves of 
$\widetilde \alpha_i$ and $\widetilde \alpha_{i+2}$ is at least $2A$, 
because no unstable leaf intersects both the stable leaf of $\widetilde \alpha_i$ and 
$\widetilde \alpha_{i+2}$.

So, using the facts that $\wt M \simeq \R^3$, that each leaves of the lifted flow is homeomorphic to $\R^2$ and that the stable leaf of $\al i$ separates $\wt M$ in two pieces, one containing $\widetilde \alpha_{i-1}$ and the other 
$\widetilde \alpha_{i+1}$, we immediately obtain
\[
 d(\widetilde \alpha_0, \widetilde \alpha_i)  \geq Ai. \qedhere
\]
\end{proof}

And, using Proposition \ref{prop:distance_corners_bounded_above} instead of Lemma \ref{lem:separation_constant}, we get an upper bound:
\begin{lemma} \label{lem:upper_bound_distance_in_string}
 There exists $B>0$, depending only on the flow, such that, if $\{\alpha_i\}$ is a string of orbits, then, for all $i$, there exists an homotopy $H_i$ between $\alpha_0$ and $\alpha_i$ that moves points a distance at most $Bi$.
\end{lemma}

\section{Examples of $\R$-covered Anosov flows on toroidal manifolds} \label{sec:flows_toroidal_manifolds}

Obviously examples of $\R$-covered Anosov flows include suspensions of Anosov diffeomorphism and 
geodesic flows of negatively curved surfaces. But there are many more examples.
For instance, the second author \cite{Fen:AFM} constructed examples of $\R$-covered Anosov flows on hyperbolic manifolds. On the other hand, Barbot proved in \cite{Barbot:VarGraphees} that the examples constructed by Handel and Thurston in \cite{HandelThurston} are $\R$-covered Anosov flows on graph manifolds, i.e., manifolds such that all their pieces in their JSJ decomposition are Seifert-fibered.

But there also exists $\R$-covered Anosov flows on manifolds admitting all sorts of torus decomposition, i.e., 
with any number of Seifert fibered pieces and atoroidal pieces, including examples with only atoroidal pieces. For instance, the second author constructed in \cite{Fenley:diversified_behavior} examples of $\R$-covered Anosov flow on manifolds with some Seifert and some atoroidal pieces. We give here a slightly different construction and note that it can also yield manifolds with only atoroidal pieces.

Our construction will be based on the Foulon--Hasselblatt surgery described in \cite{FouHassel:contact_anosov}. 
One of the great advantages of that surgery is that it yields a contact Anosov flow, i.e., an Anosov flow such that its generating vector field is the Reeb field of a contact form. This is helpful in our setting because Barbot showed in \cite{Bar:PAG} that contact Anosov flows are $\R$-covered.

Note that the results in \cite{FouHassel:contact_anosov} essentially imply the existence 
of $\R$-covered Anosov flows on manifolds with various torus decompositions, but this was not 
explicitly stated there.

The Foulon--Hasselblatt surgery is a Dehn surgery done on a tubular neighborhood of an \emph{$E$-transverse Legendrian knot}. A \emph{Legendrian knot} in a contact manifold is a closed curve tangent to the contact structure. By definition, such a curve is always transverse to the flow. It is called \emph{$E$-transverse} if it is also transverse to the strong stable and strong unstable subbundles.

$E$-transverse Legendrian knots are very common. For instance, if $\flot$ is the geodesic flow of a negatively curved surface $\Sigma$, then one can take a closed geodesic $(c(t),\dot{c}(t)) \subset T^1\Sigma$ and rotate the tangent vectors by $\pi/2$. The curve $(c(t),\dot{c}(t)+\pi/2)$ is then a $E$-transverse Legendrian knot.

We can now paraphrase the Foulon--Hasselblatt construction in one theorem (see Theorem 4.2 in \cite{FouHassel:contact_anosov}).
\begin{theorem}[Foulon, Hasselblatt \cite{FouHassel:contact_anosov}]
Let $\flot$ be a contact Anosov flow on a $3$-manifold $M$. Suppose that $\gamma$ is a simple, i.e., without self-intersection, $E$-transverse Legendrian knot. Then, for any small tubular neighborhood $U$ of $\gamma$, half of the Dehn surgeries on $U$ yields a manifold $N$ that supports a contact Anosov flow $\psi^t$.

Moreover, the orbits of $\flot$ that never enters the surgery are still orbits of the new flow $\psi^t$ and the contact form of $\flot$ and $\psi^t$ are the same on $M\smallsetminus U = N\smallsetminus U$.
\end{theorem}

In particular, the Foulon--Hasselblatt surgery can be performed, either simultaneously or recursively, on a finite number of disjoint, simple, $E$-transverse Legendrian knots. Indeed, an $E$-transverse Legendrian knot that does not enter the surgery is still $E$-transverse and Legendrian for the surgered flow.

The reason we can do only half of the Dehn surgeries is that a certain positivity condition needs to be 
satisfied in order for the proof that the surgered flow is Anosov to work (see the proof of Theorem 4.3 in \cite{FouHassel:contact_anosov} or Sections 2.3 and 2.4 in \cite{Bar:HDR})

We can now explain how to build an $\R$-covered Anosov flow such that its torus decomposition consists of one Seifert-fibered piece and one atoroidal one.

Let $\Sigma_3$ be a genus $3$ surface equipped with a hyperbolic metric, and $\varphi_0^t$ its geodesic flow. Let $c_1$ be a geodesic on $\Sigma_3$ that splits $\Sigma_3$ into two subsurfaces $\Sigma_1$, of genus $1$, and $\Sigma_2$, of genus $2$. Let $c_2$ be a geodesic that fills $\Sigma_2$ and does not intersect $c_1$. Here, ``fills'' means that any geodesic representative of $\pi_1(\Sigma_2)$, except for $c_1$, intersects $c_2$.
Now let $\gamma_1$ and $\gamma_2$ be the $E$-transverse Legendrian knots in $T^1\Sigma_3$ obtained by rotating the direction vector of the geodesics $(c_1, \dot{c}_1)$ and $(c_2, \dot{c}_2)$ by $\pi/2$.

\begin{claim}
 For infinitely many Foulon--Hasselblatt surgeries on $\gamma_1$ and $\gamma_2$, the resulting manifold $M$ as a torus decomposition with one Seifert-fibered piece and one atoroidal piece.
\end{claim}

\begin{proof}
 Let $N$ be the manifold obtained after a Foulon--Hasselblatt surgery on $\gamma_1$. Then $N$ is a graph-manifold consisting of two Seifert-fibered spaces $N_1$ and $N_2$ homeomorphic to respectively $T^1\Sigma_1$ and $T^1\Sigma_2$ (see \cite{HandelThurston} or \cite[Theorem 6.2]{FouHassel:contact_anosov}).
Now, if $U_{\gamma_2}$ is a tubular neighborhood of $\gamma_2$, then $N_2\smallsetminus U_{\gamma_2}  \simeq T^1\Sigma_2 \smallsetminus U_{\gamma_2}$ is hyperbolic (see \cite[Appendix B]{FouHassel:contact_anosov} or \cite{Fenley:diversified_behavior}). Hence, all but a finite number of Dehn surgeries on $U_2$ will yield an hyperbolic manifold \cite{Thurston_3manifolds_kleinian_groups}. Therefore, for infinitely many Foulon--Hasselblatt surgeries on $\gamma_2$ in $N$, the surgered manifold $M$ will have a torus decomposition consisting of one atoroidal piece and a Seifert-fibered piece homeomorphic to $N_1$.
\end{proof}

To build a contact flow on a manifold with two atoroidal pieces, we can start with $\Sigma_4$ a surface of genus $4$, choose $c_1$, $c_2$ and $c_3$ three non-intersecting geodesics such that: $c_1$ splits $\Sigma_4$ in two surfaces of genus $2$, and $c_2$ and $c_3$ each fills one of the split surfaces. Doing Foulon--Hasselblatt surgery on the Legendrian knots obtained from $c_1$, $c_2$ and $c_3$ will almost always give a contact Anosov flow on a manifold with two atoroidal pieces.

It should be clear from that construction how one can build a contact Anosov flow on a manifold with any sort of JSJ decomposition. So in summary, we have:
\begin{theorem}
 There exist contact Anosov flows (so, in particular, $\R$-covered Anosov flow) on manifolds with their torus decomposition consisting of any number of Seifert-fibered pieces and any number of atoroidal pieces (including only atoroidal pieces or only Seifert pieces).
\end{theorem}

\section{Classifying flows via their free homotopy classes} \label{sec:classifying_flows_free_homotopy}
In this section, we first prove Theorem \ref{thmintro:finite_hom_class}.

\begin{theorem} \label{thm:finite_homotopy_class}
 Let $\flot$ be a $\R$-covered Anosov flow on a closed $3$-manifold $M$. Suppose that every periodic orbit of $\flot$ is
freely homotopic to \emph{at most} a finite number of other periodic orbits.
Then either $\flot$ is orbit equivalent to a suspension or $\flot$ is orbit equivalent to a finite cover of the geodesic flow of a negatively curved surface.
\end{theorem}

Theorem \ref{thm:finite_homotopy_class} is a consequence of the following:
\begin{theorem}\label{thm:finite_implies_seifert_or_torus}
 Let $\flot$ be a $\R$-covered Anosov flow on a closed, orientable $3$-manifold $M$ and
suppose that $\flot$ is not orbit equivalent to a suspension. 
Suppose that $\fs$ is transversely orientable.
Let  $\alpha$ be a periodic orbit of $\flot$.
Then $\alpha$ has only finitely many periodic orbits in its free homotopy class, if and only if $\alpha$ is 
either isotopic into one of 
the tori of the 
JSJ decomposition, or isotopic to a curve contained in a Seifert-fibered piece of the JSJ decomposition.
\end{theorem}

\begin{proof}
Since $\phi^t$ is $\R$-covered and not orbit equivalent to a suspension, then 
$\phi^t$ has the skewed type as explained in section 2. There are no branching leaves and
hence any chain of lozenges is in fact a string of lozenges.
Since $\phi^t$ is skewed each lift $\widetilde \alpha$ of $\alpha$ generates
an infinite string of lozenges ${\mathcal{C}}$ in $\widetilde M$.
Since this is a string of lozenges then a closed orbit $\beta$ is in $\FH(\alpha)$ if
and only if there is a lift $\widetilde \beta$ that is a corner of ${\mathcal{C}}$.
Hence $\FH(\alpha)$ is finite if and only if the string of orbits obtained by 
projecting the corners of ${\mathcal{C}}$ to $M$ is finite, that is, $\FH(\alpha)$ is finite periodic. 
So in particular if $\FH(\alpha)$ is finite, then 
Proposition \ref{prop:finite_periodic_are tori_or_Seifert} implies the result.

Let us now deal with the other direction. Suppose that up to isotopy $\alpha$ is on one of the tori or entirely inside a Seifert
piece of the JSJ decomposition.

If $\alpha$ is on one of the boundary tori then as an element of $\pi_1(M)$, $\alpha$
is in a $\Z^2$ subgroup of $\pi_1(M)$. If $\alpha$ is contained in a 
Seifert piece of the JSJ decomposition, then in $\pi_1(M)$,
$\alpha^2$ 
commutes with an element representing a regular fiber of the Seifert fibration
in the piece.
In either case $\alpha^2$ is an element of a subgroup $G \sim \Z^2$
of $\pi_1(M)$. Let $g \in G$ associated with $\alpha^2$, and 
$\wt \alpha$ a lift of $\alpha$ to $\wt M$ left invariant by $g$.
Let $f \in G$ not leaving $\wt \alpha$ invariant. Then 
\[
g(f(\wt \alpha)) = f(g(\wt \alpha)) = f(\wt \alpha), 
\]
so $\wt \alpha$ and $f(\wt \alpha)$ are distinct orbits of $\wt \phi^t$ that
are invariant
under $g$ non trivial in $\pi_1(M)$. This implies that $\wt \alpha$ and
$f(\wt \alpha)$ are connected by a chain of lozenges $\mathcal{C}_0$.
This chain is a part of a bi-infinite chain $\mathcal{C}$ that is invariant
by $g$. The transformation $f$ acts as a translation in the corners of $\mathcal{C}$,
which shows that these corners project to only finitely many closed orbits 
of $\phi^t$ in $M$. Therefore the string of orbits associated to $\mathcal{C}$ is finite. 
On the other hand, using again that the flow is $\R$-covered, we have that any $\beta \in \FH(\alpha)$ has a coherent lift $\wt \beta$ to $\wt M$ so that $\wt \beta$ is 
a corner of this bi-infinite chain $\mathcal{C}$. 

This ends the proof of Theorem \ref{thm:finite_implies_seifert_or_torus}.
\end{proof}

Now we prove Theorem \ref{thm:finite_homotopy_class}.

\begin{proof}[Proof of Theorem \ref{thm:finite_homotopy_class}]
If a finite lift of $\flot$ is a suspension then $\flot$ itself is a suspension
\cite{Fe1}. So we assume from now on that $M$ is orientable and both stable and unstable
foliations are transversely orientable.

Suppose that every periodic orbit of $\flot$ is
freely homotopic to at most a finite number of other periodic orbits
and also that $\flot$ is not orbit equivalent to a suspension.
We want to show that the flow is, up to finite covers, orbit equivalent to a geodesic flow.
All we have to do is to prove that the manifold $M$ is Seifert-fibered, as a previous result of 
Barbot \cite{Barbot:VarGraphees} (see also Ghys \cite{Ghys:varietes_fibrees_en_cercles}, 
or \cite{BarbotFenley1} for the generalization to pseudo-Anosov flows) yields the orbit equivalence.

Suppose that $M$ is not Seifert-fibered.
If $M$ was hyperbolic then \cite[Theorem 4.4]{Fen:AFM} shows that every free homotopy class
is infinite, contrary to the hypothesis.
It follows that $M$ has at least one torus in its
torus decomposition.
As $\flot$ is $\R$-covered, it is transitive (see \cite{Bar:CFA}), so 
there exists a periodic orbit $\alpha$ that is neither contained in one piece of the modified JSJ
decomposition nor in one of the tori of the decomposition.
To build such a periodic orbit, we can start from a dense orbit and pick a long
orbit segment returning inside one of the interiors of the Birkhoff annuli in a Birkhoff
torus of the torus decomposition. 
Using the Anosov closing lemma this orbit is shadowed by a periodic orbit
with the same properties. 
By Theorem \ref{thm:finite_implies_seifert_or_torus}, $\alpha$ has to have an 
infinite free homotopy class, which gives us a contradiction.
\end{proof}

The second author's  construction in \cite{Fenley:diversified_behavior} 
gave the first explicit examples of Anosov flows such that some orbits have 
infinite free homotopy classes and some have finite free homotopy classes. 
Gathering the results of 
Barbot \cite{Bar:CFA,Barbot:VarGraphees}, Fenley \cite{Fen:AFM} and Theorem 
\ref{thm:finite_homotopy_class}, we can now be a bit more precise. 
Let us say that a flow has a \emph{homogeneous free homotopy type} if either all 
the closed orbits have infinite free homotopy class or they all have finite free homotopy class.

Theorem \ref{thm:finite_implies_seifert_or_torus} immediately implies the following:

\begin{corollary}
 Let $\flot$ be an $\R$-covered Anosov flow on $M$ so that
$\fs$ is transversely orientable. 
Then $\flot$ has a homogeneous free homotopy type if and only if one of the following happens:
\begin{itemize}
 \item $M$ is hyperbolic (and then every closed orbit has an
infinite free homotopy class);
 \item $M$ is Seifert-fibered (and then $\flot$ is orbit equivalent to a 
finite cover of a geodesic flow and there exist $k$ such that all the closed 
orbits have exactly $k$ orbits in their free homotopy class);
 \item The flow is orbit equivalent to a suspension of an Anosov diffeomorphism 
(and then every closed orbit has a free homotopy class that is a singleton).
\end{itemize}
\end{corollary}

If $\fs$ is not transversely orientable the results holds in a double cover of $M$.
In $M$ itself there will be some free homotopy classes that are singletons and 
in the first two cases, other free homotopy classes that are not singletons.

\subsection{Restrictions on infinite free homotopy classes} \label{subsec:top_obstruction}

In this section, we prove Theorem \ref{thmintro:no_periodic_piece} and then use it to show that Theorem \ref{thm:finite_homotopy_class} is ``sharp'', in the sense that the assumption that the flow is $\R$-covered cannot be dropped.

It turns out that every periodic piece except for one special case
is an obstruction to having an infinite free homotopy class
crossing it:
\begin{theorem} \label{thm:no_periodic_piece}
 Suppose that $\FH(\alpha)$ is an \emph{infinite} free homotopy class of a
periodic orbit of an Anosov flow on an orientable manifold $M$. 
Then only finitely many orbits of $\FH(\alpha)$ can be contained in a Seifert-fibered piece
of the modified JSJ decomposition.
No orbit of $\FH(\alpha)$ can cross a \emph{periodic} Seifert-fibered piece unless that
piece is a twisted $I$-bundle over the Klein bottle.

Moreover, there exists a bound $C$, depending only on the flow and the topology of the manifold, 
such that if $\FH(\alpha)$ is a free homotopy class that stays entirely
inside a Seifert piece 
of the modified JSJ decomposition, or crosses a periodic piece that
is not a twisted $I$-bundle over a Klein bottle, then the number of orbits in $\FH(\alpha)$ is less than $C$.
\end{theorem}

In order to prove this theorem, we will use the following result.
Recall that a Birkhoff annulus is an annulus transverse to the flow except for its boundary components that are periodic orbits (see section \ref{subsec:lozenges}).

\begin{theorem}[Barbot, Fenley \cite{BarbotFenley1}, Theorem B and Section 7] \label{thm:BF_periodic_spine}
Suppose that $M$ is orientable.
 Let $P$ be a periodic Seifert piece of the modified JSJ decomposition of $M$. 
There exists a submanifold $Z$, called the spine of $P$, consisting of a \emph{finite} union of Birkhoff annuli with 
boundary periodic orbits that up to powers are freely homotopic to the regular fiber of $P$. 

The submanifold $Z$ is a model for the core of $P$ in the sense that a small neighborhood $N(Z)$ of 
$Z$ is a representative for the piece $P$.
Moreover, the only periodic orbits inside $N(Z)$ are the boundary periodic orbits in $Z$, 
and all the orbits that intersect the piece $P$ are either on one of the periodic orbits, or intersect $Z$ in 
a segment entering and exiting $Z$ transversely to the boundary.

Finally, let $g$ in $\pi_1(M)$ associated with a periodic orbit in $Z$ and let
$\mathcal{C}^{'}_Z$ be the tree of lozenges with corners the fixed points of powers of
$g$ and the lozenges that connect these.
Let $\mathcal{C}_Z$ be the subtree of $\mathcal{C}^{'}_Z$ that contains all the axes of the
elements $f \in \pi_1(P)$ so that $f$ acts freely on $\mathcal{C}^{'}_Z$.
This tree $\mathcal{C}_Z$ of lozenges projects to $Z$ in the appropriate sense.
Then every corner in $\mathcal{C}_Z$ is the corner of at least two lozenges.
In addition unless $P$ is a twisted $I$-bundle over the Klein bottle,
the tree $\mathcal{C}_Z$ is not a linear tree, and
 there exists 
$n$ such that any string of lozenges inside the chain $\mathcal{C}_Z$ contains at most $n$ lozenges.
\end{theorem}

We explain the last statement. The set $Z$ is a {\em{finite}} union of Birkhoff annuli,
suppose there are $m$ such annuli. If there is a string of lozenges of length more than
$m$, it forces $\mathcal{C}_Z$ to be a linear tree. Under the hypothesis of
$M$ orientable, it was shown in \cite{BarbotFenley1} that this implies that
$P$ is a twisted $I$-bundle
over the Klein bottle.

\begin{proof}[Proof of Theorem \ref{thm:no_periodic_piece}]
Let $\FH(\alpha)$ be a finite or infinite free homotopy class. Let $P$ be a Seifert-fibered piece of the 
modified JSJ decomposition of $M$. We suppose that an orbit of $\FH(\alpha)$ intersects
$P$.
We split the proof in two cases, depending on whether $P$ is periodic or free.

\vskip .1in
\noindent
{\bf First case: Suppose that $P$ is periodic.}
 We want to show that there exists a uniform bound on the number of orbits in
$\FH(\alpha)$ that can be contained in $P$. Let $Z$ be the spine of $P$ defined in 
Theorem \ref{thm:BF_periodic_spine}. We denote by $\{\gamma_j\}_{j=1, \dots, k}$ the set of periodic orbits in $Z$.

If $\beta \in \FH(\alpha)$ is contained in $P$, then $\beta$ is one of the $\gamma_j$.
This proves the first statement in the case that $P$ is periodic.

Suppose now that $\beta$ crosses $P$ and that $P$ is not a twisted $I$-bundle over the Klein bottle. Then the geometric intersection number of $\beta$
with one of the boundary tori of $P$ is non zero. Therefore the same is true for any
$\gamma \in \FH(\alpha)$ (by Lemma \ref{lem:cutting_orbits_in_pieces}).

Thanks to Proposition \ref{prop:free_homotopy_class_to_strings}, we can pick a string of orbits $\{\alpha_i\}$ inside $\FH(\alpha)$, and, thanks to the uniform control given by Proposition \ref{prop:free_homotopy_class_to_strings} on the number of such strings and the number of orbits of $\FH_{\textrm{finite}}(\alpha)$, finding a uniform bound for the number of orbits in the string $\{\alpha_i\}$ gives a uniform bound for the number of orbits in $\FH(\alpha)$.

As explained previously we can assume that
the $\alpha_i$ intersect $P$ but are not contained in it. 
Since the $\alpha_i$ are periodic, they cannot be on the stable or unstable leaves of the $\{\gamma_j\}_{j=1, \dots, k}$. Hence, again by Theorem \ref{thm:BF_periodic_spine}, each $\alpha_i$ intersects $Z$ transversely. 
Let $\mathcal{C}_Z$ be a chain of lozenges in $\orb$ that projects to $Z$ given by Theorem \ref{thm:BF_periodic_spine} and $\al i$ be a coherent lift of the string $\alpha_i$. Since $\alpha_0$ intersects $Z$ transversely, 
we can furthermore choose the lift $\al i$ so that $\widetilde \alpha_0$ 
(seen in $\orb$) is inside one of the lozenge in $\mathcal{C}_Z$. We call that lozenge $L_0$ and its corners $c_0$ and $c_1$. Since $\{\al i \}$ are the corners of a string of lozenges, up to renaming $c_0$ and $c_1$, then $c_1$ 
has to be in the lozenge between $\wt \alpha_0$ and $\wt \alpha_1$.
In particular, according to Lemma \ref{lem:orbit_inside_lozenge}, 
$c_1$ can be the corner of at most two lozenges. So, thanks once again to 
Theorem \ref{thm:BF_periodic_spine}, $c_1$ is the corner 
of exactly two lozenges. We call $L_1$ the second lozenge. The orbit $\alpha_1$ is in $L_1$, hence we can iterate the argument above to get a third lozenge $L_2$ such that $L_0 \cup L_1 \cup L_2$ is a string of lozenge. 
Since $P$ is not a twisted $I$-bundle over the Klein bottle, then Theorem \ref{thm:BF_periodic_spine} implies that
the number of elements in a string of lozenges inside $\mathcal{C}_Z$ is bounded above.
Therefore the number of orbits inside the string $\{\al i \}$ is bounded above by a uniform constant.

This finishes the proof when $P$ is periodic.

\vskip .1in
\noindent
{\bf Second case: Suppose that $P$ is free.}
We want to show that 
there exists a uniform bound on the number of orbits in $\FH(\alpha)$ that stay inside $P$.

As explained before we only need to worry about string of orbits in $\FH(\alpha)$.
Let again $\{\alpha_i\}$ be a string of orbits in $\FH(\alpha)$ contained in $P$
and let $\{\al i\}$ be coherent lifts to 
the universal cover. Let $g \in \pi_1(M)$ be the common stabilizer of the $\al i$ and let $h\in \pi_1(M)$ 
be the representative of the fiber of $P$. Since $P$ is a free Seifert piece, we have $hgh^{-1}=g^{\pm 1}$. 
Hence, $g h \cdot \wt \alpha_0 = hg^{\pm 1} \cdot \wt \alpha_0 = h \cdot \wt \alpha_0$, so $g$ stabilizes 
$h\cdot \wt \alpha_0$. And, by iteration, $g$ stabilizes $h^n\cdot \wt \alpha_0$ for any $n\in \Z$. 
Hence all the $h^n\cdot \wt \alpha_0$ are linked by a chain of lozenges. Moreover, since $\wt \alpha_0$ 
cannot be the corner of more than two lozenges, it follows that, for all $n\geq 0$ (or all $n\leq 0$), 
$h^n\cdot \wt \alpha_0 \in \{\al i\}$. In particular, there exists $k \in \N$ such that 
$\wt \alpha_k = h \cdot \wt \alpha_0$ (or $\wt \alpha_k = h^{-1} \cdot \wt \alpha_0$), hence the number of orbits 
in the string $\{\alpha_i\}$ is less than $k$.

All there is to show now is that $k$ does not depend on the string $\{\alpha_i\}$, but only on $h$ (hence only on $P$), which will finish the proof. Let $\tau_{\textrm{max}}(h)$ be the maximum translation length of $h$ inside $P$, i.e.,
\[
 \tau_{\textrm{max}}(h) := \sup_{x\in \wt P} d(x,h \cdot x),
\]
 where $\wt P$ is a lift of $P$ in $\wt M$. Since $P$ is compact, the supremum above is in fact attained, and hence finite.

By Lemma \ref{lem:distance_greater_Ai}, there exists $A>0$ uniform such that $d(\wt \alpha_0, \al i) >Ai$, so 
\[
 Ak < d(\wt \alpha_0, \wt \alpha_k) = d(\wt \alpha_0, h\cdot \wt \alpha_0) \leq \tau_{\textrm{max}}(h).
\]
Hence $k \leq \tau_{\textrm{max}}(h)/A$, so is bounded above by a uniform constant. 
This finishes the proof of Theorem \ref{thm:no_periodic_piece}.
\end{proof}

It is easy to see that this result is also true for the more general case of 
pseudo-Anosov flows.

Now, using the examples of totally periodic Anosov flows constructed in \cite{BarbotFenley1} and Theorem \ref{thm:no_periodic_piece}, we can show that Theorem \ref{thm:finite_homotopy_class} is not true for non $\R$-covered flows.
Pseudo-Anosov flows are a generalization of Anosov flows where one allows finitely many periodic,
singular $p$-prong  orbits where $p \geq 3$ and one only assumes the existence of continuous (weak) stable/unstable
foliations and not the strong stable/unstable foliations, see \cite{Fenley:Ideal_boundaries,Fenley:qg_PA_hyperbolic_manifolds}. Note that by results of Inaba and Matsumoto \cite{InabaMatsumoto} and Paternain \cite{Paternain} (see also 
\cite{Brunella_Surfaces_section_expansive_flows}), for flows on $3$-manifolds, being pseudo-Anosov is equivalent to being expansive.

A (pseudo-)Anosov flow on $M$ is \emph{totally periodic} if $M$ is a graph manifold such that 
all its Seifert pieces are periodic. 
A consequence of Theorem \ref{thm:no_periodic_piece} is the following:

\begin{corollary} \label{cor:totally_periodic}
Suppose that $\phi^t$ is 
a totally periodic (pseudo-)Anosov flow so that no piece of the JSJ decomposition is
a twisted $I$-bundle over the Klein bottle.
Then  every periodic orbit is freely homotopic to at most a finite number of 
other periodic orbits (and there exists a uniform bound on the number of freely homotopic orbits).
\end{corollary}

\begin{proof}
If necessary we can lift to a double cover so that $M$ is orientable.
The fact that this result is true also for pseudo-Anosov flows is just because 
Theorem \ref{thm:BF_periodic_spine} holds for pseudo-Anosov flow, 
hence so does Theorem \ref{thm:no_periodic_piece}.
In the case of totally periodic pseudo-Anosov flows it follows that one can choose the
neighborhoods $N(Z)$ of the spines $Z$ to have boundary transverse to the flow.
As explained in \cite[section 7]{BarbotFenley1} this implies that the tree of lozenges $\mathcal{C}^{'}_Z$ is
equal to the ``pruned'' chain $\mathcal{C}_Z$. In particular if a periodic orbit is
freely homotopic into the Seifert piece $P$, then it is one of the vertical
orbits in $Z$. This proves the result for the vertical orbits in some Seifert piece.
Any other orbit crosses a piece. Then again Theorem \ref{thm:no_periodic_piece}, implies the finiteness of the corresponding free homotopy class, since we assumed that no piece is a twisted $I$-bundle over the Klein bottle. All the bounds are global.
\end{proof}

Note also that Theorem \ref{thm:finite_homotopy_class} for non $\R$-covered flows cannot be true, even if we ask for transitivity.
\begin{corollary}
 There exist (many) non-algebraic transitive Anosov flows such that every periodic orbit is freely homotopic 
to at most finitely many others (and there exists a uniform bound on the number of freely homotopic orbits).
\end{corollary}

\begin{proof}
The construction of totally periodic (pseudo)-Anosov flows described in \cite[section 8]{BarbotFenley1} can be done in such a way that the resulting flow is transitive, and no piece of the JSJ decomposition
is a twisted $I$-bundle over the Klein bottle. This produces the desired examples.
\end{proof}

As explained in \cite{BarbotFenley1}, the generalized Bonatti-Langevin examples 
studied by Barbot \cite{Barbot:generalized_BL} are a particular case of 
totally periodic transitive Anosov flows. They are therefore examples of Anosov flows satisfying 
the above corollary. But if we consider only the original Bonatti-Langevin example 
\cite{BonattiLangevin}, we have something even stronger:
\begin{proposition}
 Every orbit of the Bonatti-Langevin Anosov flow is alone in its free homotopy class. 
\end{proposition}

\begin{proof}
 In the Bonatti-Langevin example, every orbit but one intersects a transverse torus $T$.
If $\beta$ is an orbit intersecting $T$ then $\FH(\beta) = \{ \beta \}$  because
if $\beta$ is freely homotopic to some orbit $\gamma$, there is $\alpha$ periodic orbit
so that $\beta$ is freely homotopic to $\alpha^{-1}$, with orientations induced by
the flow and perhaps up
to powers. But since $T$ is a transverse torus to $\phi^t$ this cannot happen.

On the other hand if $\alpha$ is the single orbit not intersecting $T$, then $\alpha$ is
periodic and by the above $\FH(\alpha) = \{ \alpha \}$. This proves the result.
\end{proof}

\section{More examples of infinite free homotopy classes} \label{sec:more_examples}

Here we produce a variety of non $\R$-covered examples with infinite free homotopy classes. 
The starting point is the geodesic flow $\Phi_0$ in the unit tangent
bundle $M_0$ of a closed, orientable, hyperbolic surface $S$.
In \cite{Fenley:diversified_behavior} the second author constructed the following examples. Let $\gamma$ be a closed
geodesic in $S$ and $S_1$ the subsurface of $S$ that it fills, and we 
assume that $S_1$ is not all of $S$. Let $S_2$ be the closure of $S - S_1$.
For simplicity we also assume that $\gamma$ is
not simple, so $S_1$ is not an annulus either.
Let $\alpha$ be a closed orbit of $\Phi_0$ that projects to $\gamma$ in $S$.
Do Fried Dehn surgery along the orbit $\alpha$ to produce a manifold
$M_1$ and a surgered Anosov flow $\Phi_1$. The orbits of $\Phi_1$ are
in one to one correspondence with the orbits of $\Phi_0$. 
Under a positivity condition on the Dehn surgery (satisfied by infinitely
many Dehn surgery coefficients) the resulting Anosov flow $\Phi_1$ is 
still $\R$-covered.
The following was proved in \cite{Fenley:diversified_behavior}:

\begin{itemize}

\item
Let $\beta$ be a closed orbit of $\Phi_1$. It corresponds to an orbit $\beta_0$
of $\Phi_0$ and that in turn corresponds to a geodesic $\delta$ in $S$.

\item
If $\delta$ is {\emph{not}} isotopic into $S_2$ then the free homotopy
class of $\beta$ with respect to the surgered flow $\Phi_1$ is infinite.

\item
If $\delta$ is isotopic into $S_2$ then the free homotopy class of $\beta$ with
respect to the surgered flow $\Phi_1$ has exactly two elements.

\end{itemize}

\begin{theorem}{}{}
There is an infinite family of Anosov flows satisfying the following property: each
flow is intransitive (hence not $\R$-covered) and has infinitely many orbits so that each one has
infinite free homotopy class. It also has orbits with finite free homotopy classes.
\end{theorem}

\begin{proof}{}
Start with the geodesic flow $\Phi_0$ and do Fried Dehn surgery as above to 
obtain the Anosov flow $\Phi_1$. Now consider a geodesic $\tau$ in $S$ that is
homotopic into $S_2$ and is not peripheral in $S_2$. Peripheral means that the curve
is homotopic to the boundary.
Let $\omega$ be a periodic
orbit of the flow $\Phi_1$ that corresponds to an orbit of $\Phi_0$ that
projects to $\tau$ in $S$. 
Do a blow up of this orbit, using a derived from
Anosov operation. The resulting flow $\Phi_2$ has an expanding orbit. This operation
does not affect the periodic orbits $\beta$ of $\Phi_1$ which correspond to geodesics $\delta$
in $S$ contained in $S_1$ and the free homotopies between the periodic orbits
in $\FH(\beta)$. If the geodesic $\delta$ in $S_1$ corresponding to $\beta$ is not peripheral
then $\FH(\beta)$ is infinite (with respect to the flow $\Phi_1$),
and remains infinite when seeing $\beta$
as an orbit of $\Phi_2$. Now remove a solid torus neighborhood of the 
expanding orbit to produce a semi-flow in a manifold with torus boundary
and the flow incoming along boundary. Now glue a copy of this with a reversed
flow so that it is exiting along the boundary,
as was done by Franks and Williams in \cite{Fr-Wi}. This can be
done to produce a flow that is Anosov, as was carefully proved by
Bonatti, Beguin and Yu in \cite{BBY}. The resulting flow $\Phi_3$ still
has the orbits ``$\beta$'' as above and each of these orbits
has an infinite free homotopy class,
as do infinitely many other orbits of the flow $\Phi_3$. By the construction
the flow $\Phi_3$ is not transitive and hence not $\R$-covered.
It also has diversified homotopic behavior: if the orbit of $\Phi_1$ corresponds 
to a peripheral curve in $S_1$ then the blow up operation does not affect
this orbit and one can easily show that the corresponding orbit of $\Phi_3$ has
a free homotopy class with exactly two elements.

This finishes the proof of the theorem.
\end{proof}

We also obtain the following result:

\begin{theorem}{}{}
There is an infinite family of transitive Anosov flows that are not $\R$-covered
and that have infinitely many orbits, each of which has an infinite free homotopy 
class.
\end{theorem}

\begin{proof}{}
This a modification of the construction in the previous theorem.
We use the geodesic $\tau$ in $S$ that is homotopic into $S_2$ and not peripheral in $S_2$.
For simplicity assume that $\tau$ is simple. Suppose now that $S_2$ has high
enough genus so there is another simple geodesic $\tau'$ homotopic into $S_2$, 
not peripheral and disjoint from $\tau$. Let $\omega'$ be a periodic
orbit of $\Phi_1$ corresponding to the geodesic $\tau'$.
Besides doing the blow up of $\omega$, we also  do the blow up 
of $\omega'$ now to produce an attracting orbit. Remove neighborhoods of $\omega$
and $\omega'$ and glue another copy with a reversed flow. 
The resulting flow is denoted by $\Phi_4$. Exactly as explained for
the flow $\Phi_3$ in the previous
theorem, the flow $\Phi_4$ has infinitely many periodic orbits with
infinite free homotopy classes and also has periodic 
orbits with finite free homotopy classes.
On the other hand since we did the blow up with both a repelling and an attracting
orbit, B\'eguin, Bonatti and Yu \cite{BBY} proved that the resulting flow is transitive.
As it has a transverse torus and is not a suspension, it is not $\R$-covered.

This finishes the proof of the theorem.
\end{proof}

Now it is very easy to see that this can be iterated and blow up finitely many
orbits to obtain more complicated flows with the same properties as in these
two theorems.

Finally we prove the following:

\begin{theorem}{}{}
There is an infinite family of Anosov flows each of which satisfies the following:
the flow $\phi^t$ is transitive, and not $\R$-covered. The underlying manifold is 
not hyperbolic but has atoroidal pieces in its torus decomposition.
Every free homotopy class of periodic orbits of $\phi$
 has at most 4 elements, and every free homotopy
class but one is a singleton.
\end{theorem}

\begin{proof}{}
Let $\phi_0$ be a suspension Anosov flow on a manifold $M_0$ and $\gamma_1, \gamma_2$
two periodic orbits of $\phi_0$ that have stable and unstable
leaves that are annuli. Do a blow up of both of them, turning
one into a repelling orbit $\alpha_1$ and the other an attracting orbit
$\alpha_2$. Remove neighborhoods of $\alpha_1, \alpha_2$ to produce
a manifold $M_1$ with boundary a union of two tori $T_1, T_2$ and a semiflow
in $M_1$ that is entering $T_1$ and exiting $T_2$. Glue $M_1$ to a homeomorphic
manifold $M_2$ with a reversed flow. 
The torus $T_1$ in $M_0$ bounds a solid torus and therefore has a well
defined meridian up to isotopy, that is, a curve in $T_1$ that bounds a disk in the solid torus.
Because the stable and unstable leaves of $\gamma_1$ are annuli, there is also a well
defined longitude in $T_1$ that is a component of the intersection of the local
stable leaf of $\gamma_1$ with the torus $T_1$. Similarly the same happens in
$T_2$.

The resulting flow is $\phi$ in
the manifold $M = M_1 \cup M_2$. By results of B\'eguin, Bonatti and Yu \cite{BBY}
the gluing can be done so that 
the resulting flow is Anosov and transitive. 
In addition $\phi$ admits two
transverse tori $T_1, T_2$ which are not isotopic to each other.
It follows that $\phi$ cannot be orbit equivalent to a suspension Anosov flow
and $\phi$ cannot also be $\R$-covered with skewed type. It follows that
$\phi$ is not $\R$-covered.

We will show that for any periodic orbit $\beta$ of $\phi$ then the
free homotopy class of $\beta$ has at most $4$ elements. In fact we show that
$FH(\beta) = \{ \beta \}$ except for one free homotopy class. We will also show
that $M = M_1 \cup M_2$ is the torus decomposition of $M$ and that 
$M_1, M_2$ are atoroidal. 

We first show that for any periodic orbit $\beta$ of $\phi$, then $FH(\beta)$
has at most $4$ elements.
Suppose that there is an orbit freely homotopic to $\beta$ and distinct from $\beta$.
Then there is an orbit that is freely homotopic to the inverse of $\beta$ as oriented
curves (\cite{Fen:QGAF}). 
Suppose first that $\beta$ intersects $T_1$ or $T_2$.
Since $T_1$ and $T_2$ are transverse to $\phi$ this implies that
$\beta$ cannot be freely homotopic to any other periodic orbit of $\phi$.
This can be done by looking at the algebraic intersection number with $T_1$ and
$T_2$.
Suppose now that $\beta$ is contained in say $M_1$ and let $A$ be a possibly
immersed annulus that realizes a free homotopy from $\beta$ to the inverse
of a periodic orbit $\delta$. Put this free homotopy in general position with
$\partial M_1 = T_1 \cup T_2$. Using the fact that $M$ is irreducible, and
cut and paste techniques we may assume that either $A$ is contained in $M_1$
or there is subannulus $A_1$ of $A$ contained in $M_1$, so that
$A_1$ is a free homotopy from $\gamma$ to a closed curve $\epsilon$ in $\partial M_1$.
Suppose first that $A$ is entirely contained in $M_1$. In this case $\epsilon = \delta$ and
the free homotopy can be blown back down to a free homotopy between orbits of $\phi_0$.
This can only happen if they are the same orbit of $\phi_0$ as $\phi_0$ is a suspension.
In particular this implies that $\beta$ is isotopic in $M_1$ to a longitude of either
$T_1$ or $T_2$.
In the second case the curve $\beta$ is peripheral in $M_1$. In particular if one
blows back down to $M_0$ this produces a free homotopy between an orbit of $\phi_0$
and one of the blow up orbits $\gamma_1$ or $\gamma_2$. Again since $\phi_0$ is a 
suspension we obtain that the orbit blown down from $\beta$ is either $\gamma_1$ or
$\gamma_2$. Again this implies that $\beta$ is isotopic in $M_1$ to a longitude of either
$T_1$ or $T_2$. The same happens from the side of $M_2$ and therefore this
can only happen if the longitudes were glued to each other.
Notice that there are two possible such orbits $\beta$ in $M_1$: these are the two
closed orbits obtained by blowing either $\gamma_1$ or $\gamma_2$ into $3$ periodic
orbits and then 
removing the original orbits $\gamma_1$ or $\gamma_2$ when removing the
solid tori. Therefore this implies that the 
free homotopy class of $\beta$ has at most $4$ elements. Every other free homotopy class 
is a singleton. This proves the statement about free homotopy classes.

Let us now prove the statement about the JSJ decomposition of $M$. We will show that
$M_1$ (and consequently $M_2$) is atoroidal.
Let $T$ be an incompressible torus in $M_1$. Since $T_1, T_2$ are incompressible
in $M$, then $T$ is also an incompressible torus
in $M$. As $\phi$ is not orbit equivalent to a suspension, $T$ can
be homotoped into a Birkhoff torus. In other words, $T$ can be a realized as a free homotopy
from an orbit to itself. But we just proved above that the only free homotopies are 
between the blow up orbits from $\phi_0$. This shows that $T$ is homotopic and hence isotopic into
either $T_1$ or $T_2$. This shows that $M_1$ is atoroidal.
This finishes the proof of the theorem.
\end{proof}

\vskip .2in

\section{Growth of period of  orbits in strings of closed orbits} \label{section:period_growth_strings}

We now start the second part of this article, where we  study  orbits inside a free homotopy class.
This section contains the bulk of the work of the second part of the article.
Here we show Theorem \ref{thmintro:period_growth}, i.e., that inside an infinite string of orbits, the period grows at least linearly and at most exponentially.
We fix a Riemannian metric $g$ on $M$. We denote by $d$ the distance in $M$ for that particular metric. 
Up to reparametrizing the Anosov flow, we suppose that the flow moves the points at unit speed for that choice of metric. 

We will use the following notations. For any curve $c$ in $M$, we write $l(c)$ for the length of the curve. 
In addition if $c$ is a path in the universal cover $\wt M$, which is the a lift of
a closed curve $\alpha$ in $M$, by $l(c)$ we always mean the length of the corresponding
curve $\alpha$.
Recall also that the Hausdorff distance between two sets $S_1,S_2$ is defined in the following way
\[
 d_{\textrm{Haus}}(S_1,S_2):=\max\{ \sup_{x\in S_1} \inf_{y\in S_2} d(x,y), \sup_{y\in S_2} \inf_{x\in S_1} d(x,y)\},
\]
and when talking about the distance between two sets, we mean the minimal distance, i.e.,
\[
 d(S_1,S_2):= \inf\{ d(x,y) \mid x\in S_1, y\in S_2 \}.
\]

We start by stating the result for the upper bound on the length growth, which is the easiest result.

\begin{theorem} \label{thm:upper_bound_length_growth}
Let $\phi^t$ be an Anosov flow in $M^3$.
 Let $\{\alpha_i\}_{i \in I}$ be a string of orbits indexed so that $\alpha_0$ is the shortest. Then the length growth is at most exponential in $i$. More precisely, there exists $C_1,C_2>0$, depending only on the flow and the manifold such that, for all $i\in I$ 
\[
 l(\alpha_i) \ \leq \ C_1 l(\alpha_0) e^{C_2 i}.
\]
\end{theorem}

\begin{proof}
 Let $\al i$ be a coherent lift of the string $\alpha_i$ and $\gamma$ the element of $\pi_1(M)$ fixing all of the $\al i$.

According to Lemma \ref{lem:upper_bound_distance_in_string} 
there exists $B>0$ depending only on the flow and the manifold such that there exists an homotopy $H_i(s,t)$, $0\leq s\leq 1$, $0\leq t\leq 1$, from $\alpha_0 = H_i(0 , \cdot)$ to $\alpha_i = H_i(1, \cdot)$ that moves points a distance at most $Bi$ (that is, for any $t$, the length of $H_i(\cdot, t)$ is bounded above by $Bi$).

Let $\wt H_i$ be a lift of $H_i$ from $\widetilde \alpha_0$ to $\al i$.
Let $x = \wt H_i(0,0) \in \widetilde \alpha_0$ and 
$y = \wt H_i(1,0) \in \al i$ be fixed. Let $c_0$ be the part of 
$\widetilde \alpha_0$ from $x$ to $\gamma \cdot x$ and $c_i$ the part of $\al i$ from $y$ to $\gamma \cdot y$. Then $\wt H_i$ is a free homotopy from $c_0$ to $c_i$ that moves points by at most $Bi$.

Hence, $c_i$ is included in $N(c_0, Bi)$, the tubular neighborhood of $c_0$ of radius $Bi$ in $\widetilde M$. 

We are going to show that the length of $c_i$ (that is, the length of $\alpha_i$) cannot get too big, because it stays in a part of $\wt M$ that has a bounded volume (depending on $i$)

There exists constants $C_1, C_2>0$, depending only on the metric on $M$ (in fact just on a lower bound for the curvature), such that the volume of balls in $\wt M$ of radius $r$ is bounded above by $C_1 e^{C_2 r}$.

Hence,
\[
 \text{Vol}\left(N(c_0, Bi)\right) \ \leq \ l(c_0)C_1 e^{C_2 Bi},
\]
where $l(c_0)$ is the length of $c_0$.

Thanks to Anosov's closing lemma, there exists $\eps>0$ depending only on the flow, such that, for any orbit $\wt \alpha$ in $\wt M$, the tubular neighborhood $N(\wt \alpha, \eps)$ of $\wt \alpha$ of radius $\eps$ 
is an \emph{embedded} solid tube in $\wt M$. Indeed, otherwise Anosov closing lemma would imply the existence of a closed orbit in $\wt M$, which is impossible.

Hence, $N(c_i, \eps)$, the tubular neighborhood of $c_i$  of radius $\eps$ is embedded in, up to replacing $B$ by $B+ \eps$, $N(c_0, Bi)$.
So there exists $\eps'>0$, depending again only on the metric on $M$, such that
\begin{equation*}
 l(c_i)\eps' \ \leq \ \text{Vol}\left(N(c_i, \eps)\right) \ \leq \
 \text{Vol}\left(N(c_0, Bi)\right) \ \leq \ l(c_0)C_1 e^{C_2 Bi}.
\end{equation*}

So up to renaming the constants $C_1$ and $C_2$, we get, as claimed,
\[
 l(\alpha_i) = l(c_i) \ \leq \ l(c_0)C_1 e^{C_2 B i} \ =  \ l(\alpha_0)C_1 e^{C_2 Bi}. \qedhere
\]
\end{proof}

Now we state the result for the lower bound on the growth of period inside a 
string of lozenges. The proof is extremely more involved. It depends in a delicate way on the 
topology and geometry of $M$ or its pieces, and the proof will take the next three subsections.

\begin{theorem} \label{thm:length_growth}
Let $\phi^t$ be an Anosov flow in $M^3$.
 Let $\{\alpha_i\}$ be a string of orbits of $\phi^t$, 
with the indexation chosen so that $\alpha_0$ is the shortest orbit. Then the length growth is at least:
\begin{enumerate}
 \item Exponential in $i$ if the manifold $M$ is hyperbolic;
 \item Quadratic in $i$ if the $\{\alpha_i\}_{i \in \N}$ intersects an atoroidal piece of the
JSJ decomposition of $M$;
 \item Linear in $i$ if $\{\alpha_i\}_{i \in \N}$ goes through two consecutive Seifert-fibered pieces
of the JSJ decomposition of $M$.
\end{enumerate}

\end{theorem}

\begin{rem}
 In some sense the theorem has content only when the string is infinite, 
since with big enough constants this is trivial for any finite string. We will however see in 
the next subsections that we can get explicit bounds on the length of $\alpha_i$ depending only on the length of a shortest orbit in the string.

For $M$ Seifert fibered every free homotopy class is finite and uniformly 
bounded in cardinality.
So in the proof we may assume that $M$ is not Seifert fibered. Therefore, using the geometrization 
theorem, $M$ is either hyperbolic; or $M$ either has an atoroidal piece or has at least one
torus in its torus decomposition.
In the third case of the above theorem, the two consecutive Seifert pieces may 
be the same piece, but in that case it is assumed that the Seifert fibration
does not extend across the gluing torus.

Also up to a double cover we may assume that $M$ is orientable. This does not 
affect possibilities (1)-(3), up to changing the constants involved.
\end{rem}

\begin{rem}
 By Lemma \ref{lem:cutting_orbits_in_pieces} if an orbit $\alpha_j$ 
in the string {\emph{crosses}} a torus $T$ of the JSJ decomposition
then every orbit in its free homotopy class also crosses $T$. The remaining
case is that distinct orbits in the string $\{ \alpha_i \}$ may be contained in distinct
pieces of the modified JSJ decomposition. Again by Lemma 
\ref{lem:cutting_orbits_in_pieces} as we move through the string (say increasing
$i$) the orbits can only change the pieces they are
contained in at most two times. So in any case we may choose a substring
still denoted by $\{ \alpha_i \}$ so that every orbit in this string is contained
in the same piece of the modified JSJ decomposition.
\end{rem}

\subsection{Hyperbolic case} \label{subsec:Hyperbolic_case}

We first start with the hyperbolic case, which is both the easiest and the one for which the period growth is the fastest.

\begin{proposition} \label{prop:length_growth_in_hyperbolic_piece}
Let $\{\alpha_i\}$ be a string of orbits of an Anosov flow on $M$. If $M$ is hyperbolic, then there exists constants $A, B>0$, independent of the homotopy class and $D_{\alpha_0}$ depending on $\alpha_0$ such that 
\begin{equation*}
 l(\alpha_i) \geq B e^{-D_{\alpha_0}} e^{Ai}.
\end{equation*}
\end{proposition}

In order to prove this proposition, we recall the following classical lemma of hyperbolic geometry (see for instance \cite[Proposition 3.9.11]{KlingenbergBook})
\begin{lemma} \label{lem:hyperbolic_geometry}
 Let $c(t)$, $t\in \R$ be a geodesic of $\Hyp^n$. Let $c_1(t)$, $a\leq t\leq b$, be a curve. Let $p$, resp.~$q$, be the orthogonal projection of $c_1(a)$, resp.~$c_1(b)$, onto $c$. 
Suppose that $d(c_1(a),p)= d(c_1(b),q) = k$ and that $d(c_1(t), c)\leq k$, for all $a\leq t\leq b$. Then
\[
 l(c_1) \ \geq \ d(p,q) \cosh k.
\]
\end{lemma}

\begin{proof}[Proof of Proposition \ref{prop:length_growth_in_hyperbolic_piece}]
We first fix a hyperbolic metric on $M$. Let $\{\al i\}$ be a coherent lift of the $\{\alpha_i\}$ and $g$ be 
a generator of the stabilizer in $\pi_1(M)$  of all $\al i$. 
Since $g$ preserves all of the $\al i$, these curves have the same endpoints on the boundary at infinity 
$\partial_{\infty}\mathbb{H}^3$. Let $c_{g}$ be the axis of $g$ acting on $\mathbb{H}^3$, 
or equivalently the geodesic with the same two endpoints as the $\al i$. 
Since $c_{g}$ and $\wt \alpha_0$ have the same endpoints on the boundary at infinity, they are a bounded Hausdorff distance from each other. We denote by $D_{\alpha_0}$ that distance, that is, $D_{\alpha_0} = 
d_{\textrm{Haus}}(c_g,\wt \alpha_0)$.

By Lemma \ref{lem:distance_greater_Ai}, there exists a constant $A>0$, depending only on the flow, such that the minimal distance between $\al i$ and $\wt \alpha_0$ is at least $Ai$. Therefore, the distance between $\al i$ and $c_{\gamma}$ is bounded below by $A i-D_{\alpha_0}$. Let $x$ be a point on $c_g$ and 
let $x_i$ be  a point on $\al i$ 
such that the orthogonal projection of $x_i$ onto $c_{g}$ is $x$, and that is the closest to $c_g$.

By Lemma \ref{lem:hyperbolic_geometry}, we get that
\begin{equation*}
 l(\alpha_i) \ \geq \ d(x,g \cdot x) \cosh(Ai-D_{\alpha_0}) \ \geq \ \frac{l(c_{g})}{2} e^{Ai} e^{-D_{\alpha_0}}.
\end{equation*}

Replacing $l(c_{g})$ by the length of the smallest geodesic in $M$, we obtain the existence of a universal constant $B>0$ such that, for all $i$, 
\begin{equation*}
 l(\alpha_i) \ \geq \ B e^{-D_{\alpha_0}} e^{Ai}. \qedhere
\end{equation*}
\end{proof}

\begin{rem}
In order to later obtain counting results with uniform control, we need to give an explicit control of 
$D_{\alpha_0}$ in terms of $l(\alpha_0)$. 
The concern here is the following. If $D_{\alpha_0}$ is very big, this means that the
curve $\wt \alpha_0$ has pieces at least $D_{\alpha_0}$ away from $c_g$ and possibly
all of $\wt \alpha_0$ is at least $D_{\alpha_0}$ from $c_g$. By the lemma this implies that 
$\alpha_0$ may have a huge length. Therefore the exponential growth of $l(\alpha_i)$ 
{\em {with respect to}} $l(\alpha_0)$ takes much longer to kick in in terms of $i$ and
hence this growth is not uniform amongst strings of orbits. 
Notice that for example in the case of ${\mathbb{R}}$-covered Anosov flows 
in hyperbolic $3$-manifolds 
with $\fs$ transversely oriented, {\em {every}} periodic orbit generates an infinite
string of orbits. Therefore there may be infinitely many different strings of orbits (in fact, our bound will prove that there \emph{must be} infinitely many, see Theorem \ref{thm:counting_orbits}).
To get uniform control we will
split $\alpha_i$ into pieces depending on how big $D_{\alpha_0}$ is and also 
\emph{how much of} $\wt \alpha_i$ is near $c_g$ or far from $c_g$. The downside is that, to get this uniform control, we obtain a worse bound of the growth of $l(\alpha_i)$ than in the proposition above.
\end{rem}

\begin{lemma}\label{lem:control_D_alpha_hyperbolic}
 Let $\{\alpha_i\}$ be a string of orbits as above. Let $a$ be the length of the shortest geodesic in $M$. If $l(\alpha_0)<t$, with $t > \max(4,ae/2)$, then, for all $i$,

\begin{equation*}
 l(\alpha_i) \ \geq \ B e^{-\sqrt{t}\log (2t/a)} e^{Ai}.
\end{equation*}
\end{lemma}

\begin{proof}
Notice that we only have to worry about the case that
$D_{\alpha_0}$ is big, for otherwise the result is immediate given the previous
lemma.
 Let $\wt \alpha_0$ and $c_g$ be as in the preceding proof. Recall that $D_{\alpha_0}$ is the Hausdorff distance between 
$\wt \alpha_0$ and $c_g$.

First, suppose that $d(\wt \alpha_0,c_g)> D_{\alpha_0}/\sqrt{t}$. Then, by Lemma \ref{lem:hyperbolic_geometry}, we have
\[
 l(\alpha_0) \ \geq \ l(c_g) e^{ D_{\alpha_0}/\sqrt{t} },
\]
So, if $a$ is the length of the shortest geodesic in $M$, we get
\[
 e^{ -D_{\alpha_0}} \ \geq \ \left( \frac{l(c_g)}{l(\alpha_0)} \right)^{\sqrt{t}}\ \geq \
 \left( \frac{a}{t} \right)^{\sqrt{t}}
\]

Hence, by Proposition \ref{prop:length_growth_in_hyperbolic_piece}, for all $i$,
\begin{equation*}
 l(\alpha_i) \ \geq \ B e^{-D_{\alpha_0}} e^{Ai}  \ \geq \ B e^{-\sqrt{t}\log(t/a)}e^{Ai}
\end{equation*}
And the lemma is proved in that case.

Now suppose that $d(\wt \alpha_0,c_g)\leq D_{\alpha_0}/\sqrt{t}$. We then cut $\wt \alpha_0$ in two (not necessarily connected) pieces (see Figure \ref{fig:splitting_of_alpha_0}): 
Let $\beta_0$ be the set of points of $\wt \alpha_0$ that are at most $D_{\alpha_0}/\sqrt{t}$ from
$c_g$, and let $\gamma_0$ be the closure of $\wt \alpha_0 \smallsetminus \beta_0$, i.e., $\gamma_0$ is the piece of $\wt \alpha_0$ such that $d(\gamma_0, c_g)\geq D_{\alpha_0}/\sqrt{t}$. By our assumption, $\beta_0$ is not empty. And since $ D_{\alpha_0}$ is the Hausdorff distance between $\wt \alpha_0$ and $c_g$, $\gamma_0$ cannot be empty either (because $t>1$, so $D_{\alpha_0}/\sqrt{t} < D_{\alpha_0}$).

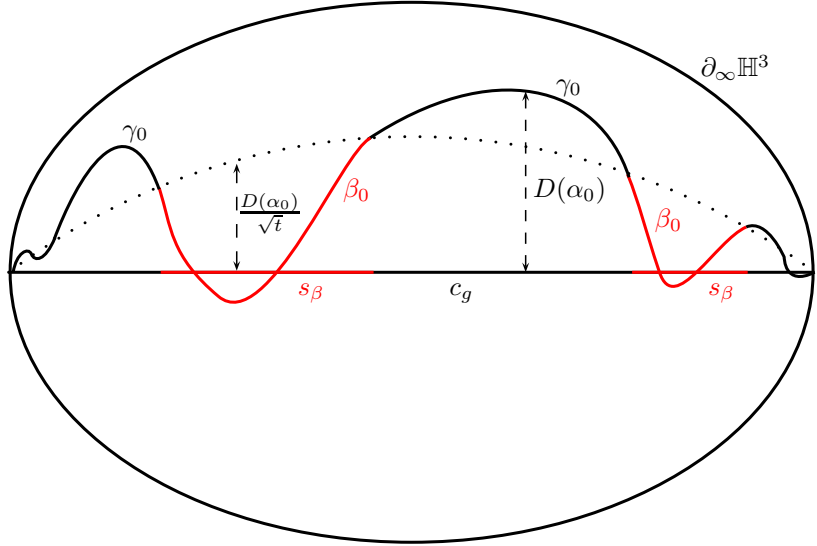
\begin{figure}[h]
 \scalebox{1}{
\begin{pspicture}(0,-3.6)(10.62,3.6)
% \psgrid
\psellipse[linewidth=0.04,dimen=outer](5.3,0.0)(5.3,3.6)
\psline[linewidth=0.04cm](0.0,0.0)(10.6,0.0)
\psbezier[linewidth=0.04,linestyle=dotted,dotsep=0.16cm](0.0,0.0)(3.2,2.4)(7.4,2.4)(10.6,0.0)
\psbezier[linewidth=0.04](10.6,0.0)(10.2,-0.2)(10.2,0.2)(10.2,0.2)(10.2,0.2)(9.98,0.72)(9.7,0.6)
\psbezier[linewidth=0.04,linecolor=red](9.7,0.6)(9.3,0.42)(8.76,-0.6)(8.56,0.0)(8.36,0.6)(8.3,0.9)(8.14,1.3)
\psbezier[linewidth=0.04](8.16,1.26)(7.76,2.5)(6.36,2.86)(4.72,1.76)
\psbezier[linewidth=0.04,linecolor=red](4.76,1.78)(4.38,1.6)(3.4,-0.86)(2.76,-0.32)(2.12,0.22)(2.14,0.6)(1.98,1.12)
\psbezier[linewidth=0.04](1.98,1.1)(1.52,2.38)(0.88,1.2)(0.6,0.44)
\psbezier[linewidth=0.04](0.61,0.46)(0.48,0.14)(0.34,0.14)(0.3,0.24)(0.26,0.34)(0.1,0.26)(0.06,0.0)
\psline[linewidth=0.02cm,linestyle=dashed,dash=0.16cm 0.16cm,arrowsize=0.05291667cm 2.0,arrowlength=1.4,arrowinset=0.4]{<->}(6.8,0.0)(6.8,2.4)
\psline[linewidth=0.02cm,linestyle=dashed,dash=0.16cm 0.16cm,arrowsize=0.05291667cm 2.0,arrowlength=1.4,arrowinset=0.4]{<->}(3.0,0.02)(3.0,1.44)
\psline[linewidth=0.04cm,linecolor=red](4.8,0.0)(2.0,0.0)
\psline[linewidth=0.04cm,linecolor=red](9.72,0.0)(8.2,0.0)
% \psline[linewidth=0.04cm,linecolor=red](10.6,0.0)(10.2,0.0)
% \psline[linewidth=0.04cm,linecolor=red](0.6,0.0)(0.3,0.0)
\put(6.9,1){$D(\alpha_0)$}
\put(3,0.7){$\frac{D(\alpha_0)}{\sqrt{t}}$}
\put(7.2,2.4){$\gamma_0$}
\put(1.5,1.8){$\gamma_0$}
\put(4.4,1){\color{red}$\beta_0$}
\put(8.5,0.6){\color{red}$\beta_0$}
\put(5.8,-0.3){$c_g$}
\put(3.8,-0.3){\color{red}$s_{\beta}$}
\put(9.2,-0.3){\color{red}$s_{\beta}$}
\put(9.1,2.6){$\partial_{\infty}\mathbb{H}^3$}
\end{pspicture} 
}
\caption{The splitting of the orbit $\al 0$ and the geodesic $c_g$}
\label{fig:splitting_of_alpha_0}
\end{figure}

We fix a fundamental domain $\Omega$ of $\wt \alpha_0$ under the action of $g$.
Let $s_{\beta}$ be the orthogonal projection of $(\beta_0 \cap \Omega)$ onto $c_g$ and 
let $s_{\gamma}$ be the orthogonal projection of $(\gamma_0 \cap \Omega)$ onto 
$c_g$. We write $d_{\beta}$ for the length of $s_{\beta}$ and $d_{\gamma}$ for the length of $s_{\gamma}$. Clearly, $d_{\beta}+d_{\gamma} \geq l(c_g)$, so either $d_{\beta}\geq l(c_g)/2$ or $d_{\gamma}\geq l(c_g)/2$.

\vskip .1in
\noindent
{\bf {First case:}} $-$
Suppose that $d_{\beta}\geq l(c_g)/2$. 

In this case we can redo the proof of Proposition \ref{prop:length_growth_in_hyperbolic_piece} for the parts of $\al i$ that are far enough from $c_g$: 
Each curve $\al i$ has the same endpoints as $c_g$, hence the orthogonal projection to $c_g$ is
surjective. For each $i$ let $\beta_i$ be the inverse image of $s_{\beta}$ of this orthogonal
projection.

Since $d(\beta_i, \beta_0) > A i$ and $d_{\textrm{Haus}}(\beta_0, c_g) \leq D_{\alpha_0}/\sqrt{t}$, we obtain that $d(\beta_i, c_g) > A i - D_{\alpha_0}/\sqrt{t}$. Moreover, by construction, the length of the orthogonal projection of $\beta_i$ onto $c_g$ is at least $d_{\beta}$. Then applying Lemma \ref{lem:hyperbolic_geometry} again, we get that 
\[
 l(\alpha_i) \ \geq \ l(\beta_i) \ \geq \ d_{\beta} e^{A i - D_{\alpha_0}/\sqrt{t}} 
\ \geq \ \frac{l(c_g)}{2}e^{Ai}e^{ - D_{\alpha_0}/\sqrt{t}}
\]

Now, since $D_{\alpha_0} = d_{\textrm{Haus}}(\wt \alpha_0, c_g)$, then
 $\gamma_0$ contains a curve that has to go from the annulus of radius $D_{\alpha_0}/\sqrt{t}$ 
around $c_g$ to the annulus of radius $D_{\alpha_0}$ around $c_g$. So,
\[
 D_{\alpha_0}\left(1-\frac{1}{\sqrt{t}}\right) \ \leq \ \frac{l(\gamma_0)}{2} \ \leq \ \frac{l(\alpha_0)}{2} 
\ < \ \frac{t}{2}.
\]
Taking $t>4$, we get that $D_{\alpha_0} < t$. Using this and the previous inequality, we get
\begin{equation*}
  l(\alpha_i) \ \geq \ \frac{l(c_g)}{2}e^{Ai}e^{- D_{\alpha_0}/\sqrt{t}} \ \geq  \ \frac{l(c_g)}{2}e^{Ai}e^{ -\sqrt{t}}.
\end{equation*}
So for some universal constant $B>0$, we get 
\begin{equation*}
  l(\alpha_i)\ \geq  \ Be^{Ai}e^{-\sqrt{t}},
\end{equation*}
hence the lemma follows for $t>ae^1/2$.

\vskip .1in
\noindent
{\bf {Second case:}} $-$
Suppose that $d_{\beta}< l(c_g)/2$.

It follows that $d_{\gamma}\geq l(c_g)/2$. Applying Lemma \ref{lem:hyperbolic_geometry} once again, we get
\[
 l(\alpha_0) \ \geq \ l(\gamma_0) \ \geq \ \frac{l(c_g)}{2} e^{ D_{\alpha_0}/\sqrt{t} },
\]
So,
\[
 e^{ -D_{\alpha_0}} \ \geq \ \left( \frac{l(c_g)}{2l(\alpha_0)} \right)^{\sqrt{t}}\ \geq \ \left( \frac{a}{2t} \right)^{\sqrt{t}}.
\]
And finally,
\begin{equation*}
 l(\alpha_i) \ \geq \ B e^{-D_{\alpha_0}} e^{Ai}  \ \geq \ B e^{-\sqrt{t}\log(2t/a)}e^{Ai}.
\end{equation*}
This finishes the proof of the lemma.
\end{proof}

\begin{rem}
The choice of the function $D_{\alpha_0}/\sqrt{t}$ as the transition function from being near $c_g$ to
being far from $c_g$ is to some extent arbitrary. Possibly different choices of the
transition function could lead to a better inequality in Lemma \ref{lem:control_D_alpha_hyperbolic}.
However, it is not clear how to make a better choice, or if it is even possible with that proof. Indeed, if one takes a bigger transition function, say $D_{\alpha_0}/2$, then we get a better bound (in $1/t$) for the part of 
$\wt \alpha_0$ that is far from $c_g$, but a far worse (in fact exponential) bound for the part that is close. Whereas if one takes a smaller transition function, say $D_{\alpha_0}/t$, then the situation is reversed. 
In particular, none of these other choices would be good enough to obtain 
Theorem \ref{thm:counting_conjugacy_classes}, i.e., the answer to Question 1, even though the constants would still be uniform.
It is very natural to try bounds of the form $D_{\alpha_0}/{t^k}$. The first bound above corresponds essentially
to $k = 0$ and the second to $k = 1$. Neither works for 
Theorem \ref{thm:counting_conjugacy_classes}, and we are lead to $0 < k < 1$.
With the transition function $D_{\alpha_0}/{t^k}$, if one follows the proof of the proposition
from near and far from the geodesic $c_g$ the following happens: One gets a bound
in terms of $\exp(-t^k \log t)$ and another in terms of $\exp(-t^{1 - 1/k}\log t)$. So clearly the optimal bound {\em {for these
types of transition functions}} occurs when $k = 2$.
\end{rem}

\vskip .1in

\subsection{One atoroidal piece} \label{subsec:atoroidal_piece}

\begin{proposition} \label{prop:quadratic_growth_on_atoroidal_piece}
 Suppose that the orbits $\{\alpha_i\}$ are all entirely contained in  an atoroidal piece $N$ of the modified torus decomposition of $M$ or they all cross into this piece $N$. Then there exists $B>0$ depending only on $M$ and the flow, and $D_{\alpha_0}>0$ depending on $\alpha_0$ such that
\begin{equation*}
 l(\alpha_i) \ \geq \ B i^2 e^{-D_{\alpha_0}} .
\end{equation*}
\end{proposition}

In order to prove the above proposition, we need a result on neutered manifolds.

\begin{definition} \label{def:neutered_manifold}
A compact manifold $N$ is a \emph{neutered hyperbolic manifold} if $N = V \smallsetminus H$ 
where $V$ is a complete hyperbolic manifold of finite volume and
$H$ is the interior of a disjoint union of horoball neighborhoods  centered at the cusps.

The \emph{neutered metric} on $N$, denoted by $d_N$, is the path metric obtained from the hyperbolic
Riemannian metric in $N$. We also lift the Riemannian metric to $\wt N$ and again denote by $d_N$
the path metric in $\wt N$ (for this lifted Riemannian metric).

The hyperbolic metric on $V$ induces another metric on $N$, that we denote by $d_H$. We also write $d_H$ for the metric on $\wt N$ induced by the hyperbolic metric on $\wt V = \Hyp^3$.
Here we think of $\widetilde N$ as a subset of 
$\widetilde V \subset {\mathbb{H}^3}$.
\end{definition}

Note that we may always choose the horoball neighborhoods so that they are spaced at least one unit from 
each other and we will always assume that in the following.
Let $\pi\colon \widetilde V \rightarrow V$ be the universal cover.

\begin{lemma} \label{lem:control_neutered_distance}
In $\widetilde N$ the following holds $d_H\leq d_N \leq 2 \sinh(d_H/2)$.
\end{lemma}

\begin{proof}
The first inequality is trivial, so we only prove the second.
Let $x_1, x_2$ be points in $\widetilde N$. Let $c$ be the \emph{hyperbolic}
geodesic arc from  $x_1$ to $x_2$. The hyperbolic distance from $x$ to $y$ is exactly
the hyperbolic length of $c$. As long as $c$ is disjoint from
$\pi^{-1}(H)$ then $c$ is contained 
in $\widetilde N$ and the hyperbolic length along $c$ is the same as the neutered length.
So we suppose this is not the case and let $\beta$ be the closure of a component of $c \cap \pi^{-1}(H)$.
In the upper half space model we can assume that the removed horoball
containing $\beta$ is associated with infinity. We may furthermore assume that the horosphere bounding that horoball is the set of points where $z = 1$. By rotations and translations
we can assume that $\beta$ is actually in ${\bf H}^2$ and connects
the points $a_0 = \frac{x}{2} + i$ and $a_1 = \frac{-x}{2} + i$ in the upper half plane. 
Let $\rho$ be the hyperbolic length of $\beta$, which  is the same as the hyperbolic 
distance from $a_0$ to $a_1$. So $\rho$ is given by
\[
 \sinh \frac{\rho}{2} =  \frac{x}{2}.
\]

Notice that $x$ is exactly the length of a segment in the boundary of the horoball,
and that is also the neutered length of this segment.
Hence any segment $\beta$ of length $\rho$ can be replaced by a segment in 
$\widetilde N$ of neutered length $x = 2 \sinh \frac{\rho}{2}$. 
The inequality of the lemma follows.
\end{proof}

\begin{proof}[Proof of Proposition \ref{prop:quadratic_growth_on_atoroidal_piece}]

We are going to prove that the corresponding
parts of the orbits $\alpha_i$ that are inside the atoroidal piece $N$ grow quadratically with the index $i$. So, if the orbits $\{\alpha_i\}$ are not entirely contained in $N$, we consider the curves $\delta_i$ obtained in the following manner.
First, we fix a generator of the fundamental group of each of the decomposition tori (so this is independent of the orbits $\alpha_i$). Then, by Lemma \ref{lem:cutting_orbits_in_pieces}, for each $i$, there exists $\alpha^N_i$ a connected component of $\alpha_i \cap N$ such that each $\alpha^N_i$ are freely homotopic to each other
relative to the boundary of $N$. Let $T_1, T_2$ be the boundary tori of $N$ that the curves $\alpha^N_i$ intersects. 
\begin{enumerate}
 \item If $T_1=T_2$, then we close each $\alpha^N_i$ along a geodesic segment on the torus between its two endpoints, making sure that we choose each geodesic segments in a coherent way, i.e., making sure that the closed paths $\delta_i$ are 
still pairwise freely homotopic to each other.
 \item If $T_1 \neq T_2$, then we close up $\alpha^N_i$ by adding loops $l^1_i$ and $l^2_i$, starting at the endpoints of $\alpha^N_i$, and in the free homotopy class of the fixed generator chosen above of, respectively, $T_1$ and $T_2$. Moreover, we choose $l^1_i$ and $l^2_i$ to be of minimal length in their homotopy class, so that their length is bounded above by a constant depending \emph{only} on the flow and the manifold.

 The path $\delta_i$ is obtained by concatenation of $l^1_i$, $\alpha^N_i$, $l^2_i$,  and $-\alpha^N_i$.
\end{enumerate}

If the orbits $\{\alpha_i\}$ are contained in $N$, then we write $\delta_i = \alpha_i$.

For convenience, we note that the important features of the $\delta_i$ are:
\begin{itemize}
 \item For all $i$, $\delta_i \subset N$,
 \item The curves $\delta_i$ are freely homotopic in $N$,
 \item The length of $\delta_i \smallsetminus \alpha^N_i$, that is, the length of the pieces of the curve which are not part of an orbit of the flow, are bounded independently of $i$ and independently of the family $\{\alpha_i\}$. Indeed, setting $D'$
\[
 D' = \max_{\epsilon} \sup_{c} \inf \left\{l(d) \mid d \textrm{ is homotopic to } c \textrm{ with fixed base point } c(0) \right\},
\]
where $\epsilon$ runs over the chosen generators of each of 
the decomposition tori, and $c\colon [0,1] \rightarrow N$ runs over all the curves in the free homotopy class of $\epsilon$, then $l(\delta_i \smallsetminus \alpha^N_i) \leq 2 D'$.
\end{itemize}

We choose a metric on $M$ such that the atoroidal piece $N$ is a neutered hyperbolic manifold. Let $d_N$, and $d_H$ be the neutered and hyperbolic metrics in $N$, as in the definition of neutered manifolds.

Let $\wt \delta_i$ be coherent lifts of the $\delta_i$ to the universal cover $\wt N$ of $N$.
Recall that, Lemma \ref{lem:distance_greater_Ai} gives a uniform $A>0$, such that the distance in $\wt M$ 
between $\al i$ and $\wt \alpha_0$ is greater than $A i$ (and hence the same inequality is true for the $d_N$ distance for the parts of $\al i$ and $\wt \alpha_0$ that stays in $\wt N$). So, the minimal separation for the neutered distance $d_N$ in $\wt N \subset \wt M$ between $\wt \delta_i$ and $\wt \delta_0$ is at least $A i - 4D'$, since $l(\delta_i \smallsetminus \alpha^N_i) \leq 2 D'$.
For convenience, we write $d^i_N := d_N (\widetilde \delta_0,\widetilde \delta_i)$, and, setting $D= 4 D'$, we have that $d^i_N >Ai -D$.

Since $\wt N$ is contained in $\wt V = \Hyp^3$, we can use the boundary at infinity of $\Hyp^3$: Let $g$ be the element of $\pi_1(N)$ that leaves invariant every $\wt \delta_i$. Since $N = V \smallsetminus H$, we can see $\pi_1(N)$ as a group of isometries of $\Hyp^3$. More precisely each element of $\pi_1(N)$ seen as a covering translation of $\wt N$ is the restriction of a hyperbolic isometry of $\Hyp^3$ to $\wt N \subset \Hyp^3$. So we call $c_g$ the hyperbolic geodesic in $\Hyp^3$ representing $g$. 
The geodesic $c_g$ is in general not contained in $\wt N$, but $c_g$ has the same endpoints as the $\delta_i$ on the boundary at infinity $\partial_{\infty} \Hyp^3$.

Let $d^{i,0}_H$ be the minimum hyperbolic distance between points in $\widetilde \delta_i$ and
$c_g$. Let $l_H(c_g)$ be the hyperbolic length of $c_g/g$, or in other words, $l_H(c_g)$ is the translation length of the hyperbolic element $g$.
Let $l_N(\delta_i)$ be the neutered length of $\delta_i$ (which is the same as its length for the 
Riemannian metric on $M$), and $l_H(\delta_i)$ be its hyperbolic length. Note that, since $\delta_i$ stays in $N$, its neutered and hyperbolic length are the same. 
Thanks to Lemma \ref{lem:hyperbolic_geometry}, we then have 

\begin{equation*}
 l_N(\delta_i) \ = \ l_H(\delta_i) \ \geq \ l_H(c_g) \cosh d^{i,0}_H \ \geq \ l_H(c_g) \frac{e^{d^{i,0}_H}}{2}.
\end{equation*}
In addition if $d^i_N$ is the minimum neutered distance from 
$\widetilde \delta_i$ to $\widetilde \delta_0$ and $d^i_H$ is the corresponding
minimum hyperbolic distance then by Lemma \ref{lem:control_neutered_distance}, $d^i_N \leq 2 \sinh(d^i_H/2)$. 
Hence 
\begin{equation*}
e^{d^i_H/2}  \ \geq \ d^i_N 
\end{equation*}

Let $D_{\alpha_0}:= d_{\textrm{Haus},H}(\widetilde \delta_0,c_g)$ be the Hausdorff distance for the hyperbolic metric between $\widetilde \delta_0$ and $c_g$. Then 
\[
d^i_H \ = \ d_H(\wt \delta_0, \wt \delta_i) 
\ \leq \ d_H(\wt \delta_i, c_g) + D_{\alpha_0} \ = \ d^{i,0}_H + D_{\alpha_0},
\]
so 
\[
 l_N(\delta_i) \ \geq  \ l_H(c_g) \frac{e^{d^{i,0}_H}}{2} \ \geq \ \frac{l_H(c_g)}{2}e^{d^i_H}e^{-D_{\alpha_0}} \ \geq \ \frac{l_H(c_g)}{2}e^{-D_{\alpha_0}}(d^i_N)^2 \ \geq \ \frac{l_H(c_g)}{2}e^{-D_{\alpha_0}}(A i - D)^2.
\]
Replacing $l_H(c_g)$ by the minimal translation length of the hyperbolic isometries in $\pi_1(N)$, we can find a constant $B>0$ depending only on the manifold and the flow such that 
\begin{equation*}
 l_N(\delta_i) \ \geq  \ Be^{-D_{\alpha_0}} i^2.
\end{equation*}

Since $l_N(\delta_i \smallsetminus \alpha^N_i)$ has length bounded by a uniform constant, we have that for some uniform constant $C>0$, $l(\alpha_i)\geq C l_N(\delta_i)$. So, replacing the constant $B$ above by $B/C$, we obtain
\begin{equation*}
 l(\alpha_i) \geq  Be^{-D_{\alpha_0}} i^2.\qedhere
\end{equation*}
\end{proof}

We will later need to control that the constant $D_{\alpha_0}$ obtained in the Proposition \ref{prop:quadratic_growth_on_atoroidal_piece} does not get too big with $l(\alpha_0)$. The following Lemma deals with that and its proof is essentially the same as the Lemma \ref{lem:control_D_alpha_hyperbolic} 
\begin{lemma} \label{lem:control_D_alpha_neutered}
 Let $\{\alpha_i\}$ and $\beta$ be as in Proposition \ref{prop:quadratic_growth_on_atoroidal_piece}. There exists a uniform constant $C>0$ such that, if $l(\alpha_0)<t$, with $t> \max(4, Ce^{1/C})$, then, for all $i$,
\begin{equation*}
 l(\alpha_i) \ \geq \ B e^{-\sqrt{t}\log (t/C)} i^2.
\end{equation*}
\end{lemma}

\begin{proof}
We use the same notations as in the proof of Proposition \ref{prop:quadratic_growth_on_atoroidal_piece}. In particular, $H$ indices refer to distance computed in the hyperbolic metric, while $N$ indices refer to the neutered metric. Length without any index refers to the length in $M$. 

First, suppose that $d_H(\wt\delta_0,c_g)> D_{\alpha_0}/\sqrt{t}$. Then, applying once more Lemma \ref{lem:hyperbolic_geometry}, we have
\[
 l_N(\delta_0) \ = \ l_H(\delta_0)  \ \geq \ l_H(c_g) e^{ D_{\alpha_0}/\sqrt{t} },
\]
where the first equality comes from the fact that $\delta_0$ is entirely in $N$, hence its neutered and hyperbolic length are equal.
Recall that by construction of $\delta_i$, there exists a uniform constant $C>0$ such that $l(\alpha_i) \geq C l(\delta_i)$. Setting 
$a$ to be the smallest translation length of hyperbolic elements in $\pi_1(N)$, we get
\[
 e^{ -D_{\alpha_0}} \ \geq \ \left( \frac{l_H(c_g)}{l_N(\delta_0)} \right)^{\sqrt{t}}
\ \geq \ \left( \frac{Ca}{t} \right)^{\sqrt{t}}
\]

Hence, by Proposition \ref{prop:quadratic_growth_on_atoroidal_piece}, for all $i$,
\begin{equation*}
 l(\alpha_i) \ \geq \ B i^2 e^{-D_{\alpha_0}}  \ \geq \ Bi^2 e^{-\sqrt{t}\log(t/Ca)}
\end{equation*}
And the lemma is proved in that case, up to changing $C$ to $aC$.

\vskip .1in

Now suppose that $d_H(\wt\delta_0,c_g)\leq D_{\alpha_0}/\sqrt{t}$. Then, as in the proof of Lemma \ref{lem:control_D_alpha_hyperbolic}, we let $\beta_0$ be the piece of $\wt\delta_0$ such that its Hausdorff hyperbolic distance is at most $D_{\alpha_0}/\sqrt{t}$, and let $\gamma_0$ be the closure of $\wt\delta_0 \smallsetminus \beta_0$, i.e., $\gamma_0$ is the 
piece of $\wt \alpha_0$ such that $d_H(\gamma_0, c_g)\geq D_{\alpha_0}/\sqrt{t}$ (see Figure \ref{fig:split_of_delta_0_in_neutered}). By our assumption, $\beta_0$ is not empty. And since $ D_{\alpha_0}$ is the Hausdorff distance between $\wt\delta_0$ and $c_g$, $\gamma_0$ cannot be empty either (because we can assume $t>1$, so $D_{\alpha_0}/\sqrt{t} < D_{\alpha_0}$).

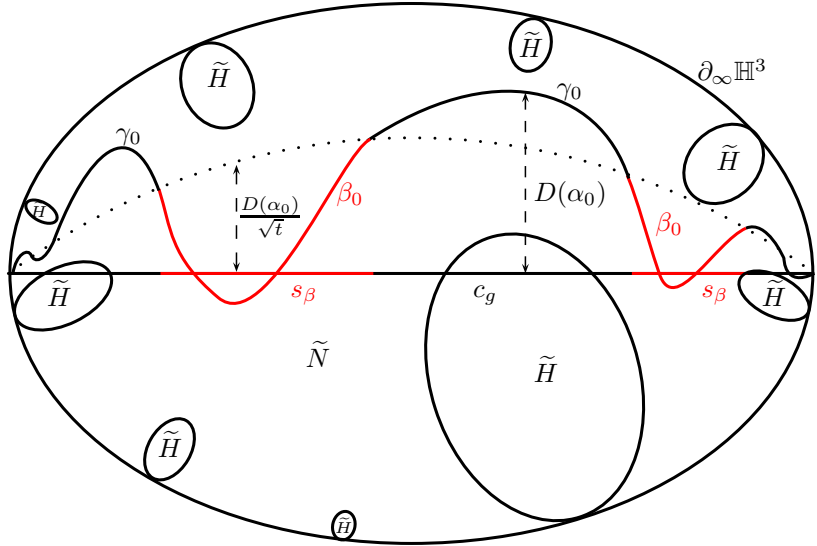
\begin{figure}[h]
 \scalebox{1}{
\begin{pspicture}(0,-3.5643353)(10.709215,3.6356647)
% \psgrid
\psellipse[linewidth=0.04,dimen=outer](5.3892155,0.035664614)(5.3,3.6)
\psline[linewidth=0.04cm](0.08921539,0.035664614)(10.689216,0.035664614)
\psbezier[linewidth=0.04,linestyle=dotted,dotsep=0.16cm](0.08921539,0.035664614)(3.2892153,2.4356647)(7.4892154,2.4356647)(10.689216,0.035664614)
\psbezier[linewidth=0.04](10.689216,0.035664614)(10.289215,-0.16433538)(10.289215,0.23566462)(10.289215,0.23566462)(10.289215,0.23566462)(10.069216,0.7556646)(9.789215,0.63566464)
\psbezier[linewidth=0.04,linecolor=red](9.789215,0.63566464)(9.389215,0.4556646)(8.8492155,-0.5643354)(8.649216,0.035664614)(8.449215,0.63566464)(8.389215,0.9356646)(8.229216,1.3356646)
\psbezier[linewidth=0.04](8.249215,1.2956647)(7.8492155,2.5356646)(6.4492154,2.8956647)(4.8092155,1.7956647)
\psbezier[linewidth=0.04,linecolor=red](4.8492155,1.8156646)(4.4692154,1.6356646)(3.4892154,-0.8243354)(2.8492153,-0.28433537)(2.2092154,0.25566462)(2.2292154,0.63566464)(2.0692153,1.1556646)
\psbezier[linewidth=0.04](2.0692153,1.1356646)(1.6092154,2.4156647)(0.9692154,1.2356646)(0.68921536,0.47566462)
\psbezier[linewidth=0.04](0.6992154,0.49566463)(0.56921536,0.17566462)(0.4292154,0.17566462)(0.38921538,0.27566463)(0.3492154,0.37566462)(0.18921539,0.2956646)(0.14921539,0.035664614)
\psline[linewidth=0.02cm,linestyle=dashed,dash=0.16cm 0.16cm,arrowsize=0.05291667cm 2.0,arrowlength=1.4,arrowinset=0.4]{<->}(6.8892155,0.035664614)(6.8892155,2.4356647)
\psline[linewidth=0.02cm,linestyle=dashed,dash=0.16cm 0.16cm,arrowsize=0.05291667cm 2.0,arrowlength=1.4,arrowinset=0.4]{<->}(3.0892153,0.055664614)(3.0892153,1.4756646)
\psline[linewidth=0.04cm,linecolor=red](4.8892155,0.035664614)(2.0892153,0.035664614)
\psline[linewidth=0.04cm,linecolor=red](9.809216,0.035664614)(8.289215,0.035664614)
\rput{17.585947}(-0.0777352,-2.1808026){\psellipse[linewidth=0.04,dimen=outer](7.010414,-1.3416746)(1.3908548,1.9656266)}
\rput{63.012943}(5.319135,-9.194727){\psellipse[linewidth=0.04,dimen=outer](10.159882,-0.25844386)(0.28518012,0.5208078)}
\rput{116.05324}(0.9286281,-1.1030467){\psellipse[linewidth=0.04,dimen=outer](0.8085879,-0.26168758)(0.39583153,0.7065293)}
\rput{150.10461}(2.9787002,-5.3871713){\psellipse[linewidth=0.04,dimen=outer](2.2084594,-2.2959723)(0.30040237,0.45937324)}
\rput{166.42757}(8.0997305,-7.5812073){\psellipse[linewidth=0.04,dimen=outer](4.500942,-3.3086746)(0.16594714,0.21129633)}
\rput{166.42757}(14.443521,4.412568){\psellipse[linewidth=0.04,dimen=outer](6.959215,3.0656645)(0.28560406,0.3734037)}
\rput{135.50894}(17.314747,-4.104791){\psellipse[linewidth=0.04,dimen=outer](9.496843,1.4886407)(0.47966433,0.6023625)}
\rput{203.87234}(4.4076524,5.9890256){\psellipse[linewidth=0.04,dimen=outer](2.8368437,2.5286407)(0.47966433,0.6023625)}
\rput{240.19801}(0.042927217,1.7356547){\psellipse[linewidth=0.04,dimen=outer](0.52450544,0.85538584)(0.14174393,0.24691801)}
\put(7,1){$D(\alpha_0)$}
\put(3.1,0.7){$\frac{D(\alpha_0)}{\sqrt{t}}$}
\put(7.3,2.4){$\gamma_0$}
\put(1.5,1.8){$\gamma_0$}
\put(4.4,1){\color{red}$\beta_0$}
\put(8.6,0.6){\color{red}$\beta_0$}
\put(6.2,-0.28){$c_g$}
\put(3.8,-0.28){\color{red}$s_{\beta}$}
\put(9.2,-0.28){\color{red}$s_{\beta}$}
\put(9.15,2.6){$\partial_{\infty}\mathbb{H}^3$}
\put(4,-1.2){$\widetilde N$}
\put(7,-1.4){$\widetilde H$}
\put(2.05,-2.4){$\widetilde H$}
\put(0.6,-0.4){$\widetilde H$}
\put(2.7,2.5){$\widetilde H$}
\put(9.4,1.4){$\widetilde H$}
\put(10,-0.4){$\widetilde H$}
\put(6.8,2.9){$\widetilde H$}
\rput(4.5,-3.3){\tiny $\widetilde H$}
\rput(0.49,0.89){\tiny $\widetilde H$}
\end{pspicture}}
\caption{The orbit $\wt \delta_0$ in $\wt N= \wt V \smallsetminus \wt H$ and the geodesic $c_g$ in $\wt V = \Hyp^3$}
\label{fig:split_of_delta_0_in_neutered}
\end{figure}

Let $\Omega$ be a fundamental domain of $\wt \delta$ under the action of $g$.
Let $s_{\beta}$ be the orthogonal projection (for the hyperbolic metric) of 
$(\beta_0 \cap \Omega)$ onto $c_g$ and $s_{\gamma}$ the orthogonal projection 
of $(\gamma_0 \cap \Omega)$ onto $c_g$ (see Figure \ref{fig:split_of_delta_0_in_neutered}). 
We write $l_{H,\beta}$ for the hyperbolic length of $s_{\beta}$ and $l_{H,\gamma}$ for the 
hyperbolic length of $s_{\gamma}$. Clearly, $l_{H,\beta}+l_{H,\gamma} \geq l_H(c_g)$, 
so either $l_{H,\beta}\geq l_H(c_g)/2$ or $l_{H,\gamma}\geq l_H(c_g)/2$.

\vskip .1in
\noindent
{\bf {First case:}} $-$
Suppose that $l_{H,\beta}\geq l_H(c_g)/2$.

Here we redo the proof of Proposition \ref{prop:quadratic_growth_on_atoroidal_piece} for the parts of $\wt\delta_i$ that are far enough from $c_g$. 
Let $\pi_{c_g}: {\mathbb{H}}^3 \rightarrow c_g$ be the orthogonal projection
and let $\beta_i = \wt \delta_i \cap (\pi_{c_g})^{-1}(s_{\beta})$.
Notice that $\beta_i$ is not necessarily connected.

Let $d_{H,\beta}^{i,0}=d_H(\beta_i, c_g)$. Let $l_N(\beta_i)$ be the neutered length of $\beta_i$ and $l_H(\beta_i)$ its hyperbolic length. Once again, since we are talking about lengths of curves in $N$ (and not distances between points), $l_N(\beta_i)= l_H(\beta_i)$, but we keep the different subscripts to help remembering which metric we are considering at each time.
Thanks to Lemma \ref{lem:hyperbolic_geometry}, and our assumption that $l_{H,\beta}\geq l_H(c_g)/2$, we have 
\begin{equation}\label{eq1}
l_N(\delta_i) \ \geq \
 l_N(\beta_i) = l_H(\beta_i)  \ \geq \ l_{H,\beta} \cosh d^{i,0}_{H,\beta} 
\ \geq \ \frac{l_H(c_g)}{4} e^{d^{i,0}_{H,\beta}}.
\end{equation}
In addition if $d^i_{N,\beta}$ is the minimum neutered distance from 
$\widetilde \beta_i$ to $\widetilde \beta_0$ and $d^i_{H, \beta}$ is the corresponding
minimum hyperbolic distance then by Lemma \ref{lem:control_neutered_distance}, $d^i_{N,\beta} \leq 2 \sinh(d^i_{H,\beta}/2)$. 
Hence 
\begin{equation*} 
e^{d^i_{H,\beta}/2}  \ \geq \ d^i_{N,\beta}.
\end{equation*}
Recall that, by construction of $\delta_i$, we have $d_N(\delta_i, \delta_0) > Ai- D$, where $D$ is a uniform constant (see proof of Proposition \ref{prop:quadratic_growth_on_atoroidal_piece}). Therefore, the same inequality holds for $d^i_{N,\beta} = d_N(\beta_i, \beta_0) $, so we have
\begin{equation*}
e^{d^i_{H,\beta}/2}  \ \geq \ Ai-D
\end{equation*}

Now, by construction, $\beta_0$ is such that $d_{\textrm{Haus},H}(\beta_0,c_g)\ \leq \ D_{\alpha_0}/\sqrt{t}$, so 
\[
d^i_{H,\beta} \leq d^{i,0}_{H,\beta} + \frac{D_{\alpha_0}}{\sqrt{t}}. 
\]
Hence, using \eqref{eq1}, we get
\begin{multline*}
 l_N(\delta_i) \ \geq \ \frac{l_H(c_g)}{4} e^{d^{i,0}_{H,\beta}} \ \geq \ \frac{l_H(c_g)}{4}e^{d^i_{H,\beta}}e^{-D_{\alpha_0}/\sqrt{t}} \\ \geq \ \frac{l_H(c_g)}{4}e^{-D_{\alpha_0}/\sqrt{t}}(d^i_{N,\beta})^2  \ \geq 
\ \frac{l_H(c_g)}{4}e^{-D_{\alpha_0}/\sqrt{t}}(A i - D)^2.
\end{multline*}

Since $D_{\alpha_0} = d_{\textrm{Haus}, H}(\wt\delta_0, c_g)$, $\gamma_0$ contains a curve that has to go from the annulus of radius $D_{\alpha_0}/\sqrt{t}$ around $c_g$ to the annulus of radius $D_{\alpha_0}$ around $c_g$, which implies that
\[
 D_{\alpha_0}\left(1-\frac{1}{\sqrt{t}}\right) \ \leq \ \frac{l_N(\gamma_0)}{2}.
\]
Taking once again $C>0$ to be a uniform constant such that $l(\alpha_i) \geq C l(\delta_i)$, we obtain, for $t>4$,
\[
 D_{\alpha_0} \ \leq \ l_N(\gamma_0) \ \leq \ l_N(\delta_0) \ \leq \ \frac{l(\alpha_0)}{C} \ < \ \frac{t}{C}.
\]

Using this and the previous inequality, we obtain
\begin{equation*}
  l(\alpha_i) \ \geq \ C l_N(\delta_i) \ \geq \ 
C\frac{l_H(c_g)}{4}e^{-D_{\alpha_0}/\sqrt{t}}(A i - D)^2 \ \geq \ B i^2 e^{-\sqrt{t}/C},
\end{equation*}
where $B>0$ is some universal constant. The lemma follows for $t \geq C e^{1/C}$.

\vskip .1in
\noindent
{\bf {Second case:}} $-$
Suppose now that $l_{H,\gamma}\geq l_H(c_g)/2$.

Here we apply Lemma \ref{lem:hyperbolic_geometry} once again, and obtain
\[
 \frac{l(\alpha_0)}{C} \ \geq  \ l_N(\delta_0) \ \geq \ l_H(\delta_0)
\ \geq \ l_H(\gamma_0) \ \geq \ \frac{l_H(c_g)}{2} e^{ D_{\alpha_0}/\sqrt{t} },
\]
So,
\[
 e^{ -D_{\alpha_0}} \ \geq \ \left( \frac{C l_H(c_g)}{2l(\alpha_0)} \right)^{\sqrt{t}}
\ \geq \ \left( \frac{C a}{2t} \right)^{\sqrt{t}}.
\]
Using Proposition \ref{prop:quadratic_growth_on_atoroidal_piece}, we get
\begin{equation*}
 l(\alpha_i) \ \geq  \ B i^2 e^{-D_{\alpha_0}} \ \geq \ B i^2 e^{-\sqrt{t}\log(2t/Ca)},
\end{equation*}
which yields the lemma after changing $C$ to $Ca/2$.
\end{proof}

\subsection{Two Seifert-fibered pieces} \label{subsec:two_Seifert_pieces}

\begin{proposition} \label{prop:linear_growth_seifert_pieces}
Let $\flot$ be an Anosov flow on $M$. 
There are uniform constants $A_1, A_2 > 0$ so that the following happens:
Let $\{\alpha_i\}$ be a string of orbits. Let $S_1$ and $S_2$ be two Seifert piece glued in $M$ along one of the decomposition tori (so $S_1$ and $S_2$ are allowed to be the same with two boundary tori glued together).
 Suppose that $\alpha_i$ intersects both Seifert pieces $S_1$ and $S_2$ consecutively. Then
\begin{equation*}
 l(\alpha_i) \ \geq \ A_1 i - A_2 - l(\alpha_0).
\end{equation*}
\end{proposition}

\begin{proof}
 
This proof will split into several different cases, depending on the topological 
type of the Seifert pieces $S_1$ and $S_2$, and on the dynamical type of the flow (i.e., free or periodic) on them.

Suppose first that either $S_1$ or $S_2$ is periodic and is not a twisted $I$-bundle over the
Klein bottle.
This is the easy case, because
by Theorem \ref{thm:no_periodic_piece}, there exists a uniform bound on the number of orbits in $\{\alpha_i\}$. 
The result follows therefore trivially.

The remaining possibilities are either that the flow is free on both $S_1$ and $S_2$ or that one is free and the other is a twisted $I$-bundle over the Klein bottle, or both are twisted $I$-bundles over the Klein bottle.

We first show that the last situation cannot happen. If $S_1$ and $S_2$ are both twisted $I$-bundles over the Klein bottle, then $S_1$ and $S_2$ have a unique boundary torus $T$. Hence we have that $M = S_1 \cup S_2$. Since $\pi(T)$ is a subgroup of 
index $2$ in $\pi_1(S_1)$ and in $\pi_1(S_2)$, it follows that $\pi_1(T)$ is a subgroup of index at most $4$ in $\pi_1(M)$. 
This is impossible since a $3$-manifold supporting an Anosov flow cannot have a 
finite index subgroup homeomorphic to $\Z^2$. Indeed, this would contradict the fact that
the fundamental group of a $3$-manifold supporting an Anosov flow has exponential growth \cite{PlanteThurston}.

So we can now suppose that either the flow is free on both $S_1$ and $S_2$ or that one is free and the other is a twisted $I$-bundle over the Klein bottle. This situation is quite complicated and requires a long proof.

Note for future reference that since one of $S_1$ or $S_2$ is free then
the tori constituting the common boundaries of $S_1$ and $S_2$ 
are quasi-transverse but cannot be transverse. Indeed, let $T$ be such a
torus and consider $\mathcal{C}$ a bi-infinite chain of lozenges that is $\pi_1(T)$-invariant.
As explained in \cite{BarbotFenley1,BarbotFenley2} if $T$ is homotopic to a transverse torus, then all consecutive lozenges in
$\mathcal{C}$ are adjacent (i.e., $\mathcal{C}$ is a scalloped chain). This contradicts the fact that $S_1$ (or $S_2$) is a free piece
\cite{BarbotFenley2}.

Now, if $T$ is an incompressible torus in the boundary of $S_1$ and $S_2$,
then either there exists a 
loop on $T$ that is freely homotopic to a periodic orbit of the flow or the flow is a suspension Anosov flow 
\cite{Fenley:Incompressible_tori}. By hypothesis, the flow is not a suspension, so any boundary tori of $S_1$ and $S_2$ must have one generator of their fundamental group freely homotopic to a periodic orbit. Since we can assume that $T$ is
a quasi-transverse torus, there is a periodic orbit contained in $T$. 
If both $S_1$ and $S_2$ are free Seifert pieces, this orbit cannot be freely homotopic to the Seifert fiber direction.
This remark will be important for us in the following way: Since $S_1$ and $S_2$ 
are glued along the quasi-transverse torus $T$, the gluing is a Dehn twist that has to preserve the
periodic orbits in $T$.
Therefore the gluing is a Dehn twist around the orbits on $T$. If $(1,0)$ represents the closed orbits
in $T$ then the gluing is given by the matrix $\begin{bmatrix}
                                                         1 & n \\ 0 & 1
                                                        \end{bmatrix}$.

If both $S_1$ and $S_2$ are free, there is a lot of structure of the flow when restricted to these
pieces. First of all we choose models for $S_1, S_2$ that have every boundary a quasi-transverse
torus. In \cite{BarbotFenley3} the following facts are proved: the stable foliation restricted
to $S_1$ (or $S_2$) is transverse to the boundary and it is an ${\mathbb{R}}$-covered
foliation. In addition since $S_1$ and $S_2$ are free one can choose the Seifert fibration
in the respective piece to be transverse to the stable foliation.

As $S_1, S_2$ are Seifert fibered spaces, let $B_1$ and $B_2$ be the bases of respectively $S_1$ and $S_2$. 
In other words $B_1, B_2$ are the quotients of $S_1, S_2$ by the respective Seifert fibrations.
If $S_1$ or $S_2$ is periodic then, at this point, it is a twisted $I$-bundle over the
Klein bottle, and its base is not a hyperbolic orbifold.

\vskip .1in
\noindent
{\bf First case:} Suppose that $B_1$ and $B_2$ are hyperbolic orbifolds.

So in particular, both $S_1$ and $S_2$ are free.

Choose a hyperbolic metric on $B_1$ and $B_2$ and lift this to metrics in $S_1, S_2$ respectively,
so that the leaves of the stable foliation are hyperbolic surfaces and local holonomy along
the Seifert fibers is a hyperbolic isometry. The Seifert fibrations do not agree along the
common boundary of $S_1$ and $S_2$. So in $S_2$ we make an interpolation
between the two Riemannian metrics near these boundaries.

Let $\delta_i^1$ and $\delta^2_i$ be connected components of, respectively, $\alpha_i \cap S_1$ and $\alpha_i \cap S_2$ such that $\delta_i = \delta_i^1 \cup \delta_i^2$ is (possibly a subset of) a connected component of $\alpha_i \cap (S_1\cup S_2)$. We as usual choose all these connected components in such a way that they are freely homotopic to each other inside each piece. Let $T$ be the decomposition torus in between $\delta_i^1$ and $\delta_i^2$, and $x_i$ be the point on $T$ that separates $\delta_i$ between $\delta_i^1$ and $\delta_i^2$.

Let $\wt S_1$ be the universal cover of $S_1$. Note that $\wt S_1 = \wt B_1 \times \R$, where $\wt B_1$ is the universal cover of $B_1$. Note that $\wt B_1 \subset {\mathbb{H}}^2$. Let $\wt \delta^1_i$ be a coherent lift of the $\delta^1_i$ and $\wt x_i = \wt \delta^1_i \cap \wt T$ be the endpoint on $\wt T$.

The horizontal foliation is the stable foliation. The vertical foliation is the Seifert
fibration in each piece.
For any $x,y \in \wt S_1$, we write $d_{\textrm{Hor}}^1(x,y)$ for the distance along the horizontal foliation and $d_{\textrm{Ver}}^1(x,y)$ for the distance along the vertical foliation. Since the pieces $S_1$ and $S_2$ are glued together by a Dehn twist along $T$, there exists a constant $C_1>0$ such that a vertical leaf for the fibration on $S_1$ is sent to a line of slope $C_1$  in the coordinates given by the horizontal and vertical foliations on $S_2$.

Once more using Lemma \ref{lem:distance_greater_Ai}, we know that for some uniform $A>0$, $d(\wt \delta_0, \wt \delta_i) \geq Ai$, so 
\[
 Ai \ \leq \ d_{\textrm{Hor}}^1(\wt x_0, \wt x_i) + d_{\textrm{Ver}}^1(\wt x_0, \wt x_i).
\]

Suppose that
\[
 d_{\textrm{Hor}}^1(\wt x_0, \wt x_i) \ \geq \ \frac{AC_1i}{2+C_1},
\]
then the result follows from the following claim.
\begin{claim} \label{claim:horizontal_distance_implies_big_length}
 Let $C>0$. There exists a uniform constant $C_2\geq 0$ such that, if $d_{\textrm{Hor}}^1(\wt x_0, \wt x_i) \geq C i$ then $l(\delta^1_i) \geq C i - l(\delta^1_0) - C_2$.
\end{claim}

\begin{proof}
We first fix one horizontal leaf $\wt B_1 = \wt B_1 \times \{ 0 \}$ 
inside $\wt S_1 = \wt B_1 \times \R$ and write $\wt\beta_i$ for the projection (through the vertical foliation) of $\wt \delta_i$ onto $\wt B_1$. Let $y_i$ be the projection along the Seifert fibration direction of $\wt x_i$ onto $B_1$. We write $d_{B_1}$ for the hyperbolic distance on $\wt B_1 \subset \Hyp^2$. By definition, $d_{B_1}(y_0, y_i) = d_{\textrm{Hor}}^1(\wt x_0, \wt x_i) \geq Ci$.

Let $T_1$ be the decomposition torus containing the other endpoint of $\delta^1_i$. Let $\wt T_1$ be the coherent lift of $T_1$, i.e., the lift such that the other endpoint of $\wt \delta_i$ is on $\wt T_1$. Note that $T_1$ and $T$ might be the same torus, however, $\wt T_1 \neq \wt T$ since the $\delta_i$ are not homotopically trivial in $S_1$ relative to the boundary. This last fact comes from the quasi-transversality of $T$ and $T_1$.
Let $\gamma =\wt T \cap \wt B_1$ and $\gamma_1 = \wt T_1 \cap \wt B_1$. By our choice of metric, $\gamma$ and $\gamma_1$ are geodesics boundaries of $\wt B_1$ inside $\Hyp^2$. Let $c$ be the geodesic in $B_1$ realizing the minimal distance between $\gamma$ and $\gamma_1$ and let $y = c \cap \gamma$ the endpoint of $c$ on $\gamma$ (see Figure \ref{fig:large_horizontal_distance}).

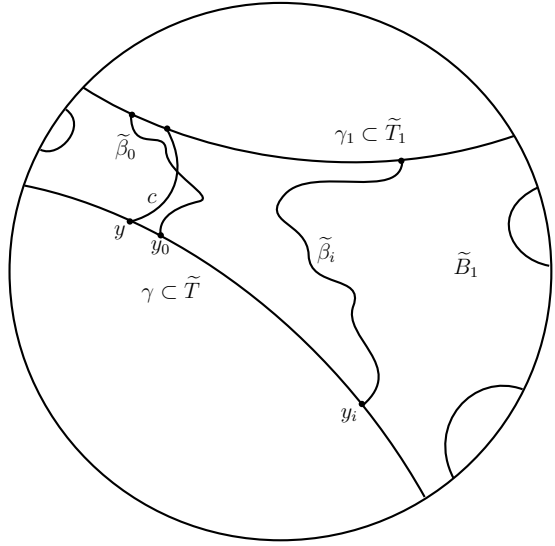
\begin{figure}[h]
\scalebox{.7}{
\begin{pspicture}(0,-5.11)(10.22,5.11)
% \psgrid
\pscircle[linewidth=0.04,dimen=outer](5.11,0.0){5.11}
\psbezier[linewidth=0.04](0.3,1.63)(3.0,1.09)(6.04,-1.07)(7.82,-4.27)
\psbezier[linewidth=0.04](1.4,3.49)(3.3,2.25)(6.5,1.53)(9.5,2.57)
\psbezier[linewidth=0.04](2.3,0.95)(3.04,1.15)(3.44,1.83)(2.98,2.69)
\psbezier[linewidth=0.04](7.4,2.11)(7.4,1.31)(5.72,2.03)(5.2,1.47)(4.68,0.91)(5.62,0.95)(5.64,0.31)(5.66,-0.33)(6.702361,-0.050708603)(6.48,-0.65)(6.257639,-1.2492914)(7.5,-1.79)(6.66,-2.53)
\psbezier[linewidth=0.04](2.3,2.99)(2.3,2.19)(2.9500706,2.7316937)(2.96,2.29)(2.9699295,1.8483063)(3.92,1.41)(3.58,1.33)(3.24,1.25)(2.8,1.07)(2.86,0.69)
\psbezier[linewidth=0.04](9.94,1.59)(9.02,1.15)(9.4,0.25)(10.16,0.11)
\psbezier[linewidth=0.04](9.66,-2.23)(8.74,-1.89)(7.84,-2.89)(8.36,-3.91)
\psbezier[linewidth=0.04](1.08,3.07)(1.5,2.77)(0.96,2.13)(0.6,2.31)
% \psbezier[linewidth=0.04,linecolor=red](7.38,2.09)(5.94,1.93)(4.06,1.99)(3.8,1.37)(3.54,0.75)(5.94,-1.35)(6.64,-2.51)
% \psbezier[linewidth=0.04,linecolor=red](7.38,2.09)(5.94,1.93)(4.4,1.79)(4.14,1.15)(3.88,0.51)(5.94,-1.35)(6.64,-2.51)
%
%
\psdots[dotsize=0.12](7.38,2.1)
\psdots[dotsize=0.12](2.98,2.71)
\psdots[dotsize=0.12](2.28,0.95)
\uput{3pt}[-135](2.28,0.95){\Large $y$}
\psdots[dotsize=0.12](6.64,-2.51)
\uput{3pt}[-135](6.64,-2.51){\Large $y_i$}
\psdots[dotsize=0.12](2.32,2.97)
\psdots[dotsize=0.12](2.86,0.69)
\uput{3pt}[-90](2.86,0.69){\Large $y_0$}

\put(5.8,0.235){\Large $\widetilde \beta_i$}
\put(6.1314063,2.495){\Large $\gamma_1 \subset \widetilde T_1$}
\put(8.361406,-0.045){\Large $\widetilde B_1$}
\put(2.5,-0.5){\Large $\gamma \subset \widetilde T$}
\put(2,2.2){\Large $\widetilde \beta_0$}
\put(2.6,1.3){\Large $c$}
\end{pspicture}}
\caption{Large horizontal distance between $y_0$ and $y_i$}
\label{fig:large_horizontal_distance}
\end{figure}

We have that $d_{B_1}(y, y_i) \geq Ci- d_{B_1}(y,y_0)$. Now, since the $\wt \beta_i$ are curves from $\gamma_1$ to $\gamma$, that $c$ is the geodesic arc perpendicular to both
$\gamma_1$ and $\gamma$, that the metric on $B_1$ is hyperbolic, and that $l(c)$ is greater than some constant depending only on the manifold $M$ (so up to scaling the metric, we can suppose that $l(c)\geq 1$), we get that for some uniform constant $C_0$,
\[
 l(\wt \beta_i) \ \geq \ d_{B_1}(y, y_i) + d(\gamma, \gamma_1) - C_0 \ \geq \ C i -  d_{B_1}(y,y_0) - C_2,
\]
where $C_2$ can be chosen to depend only on the manifold and the JSJ decomposition. 
Note that we also have
\[
 l(\wt \beta_0) \ \geq \ d_{B_1}(y, y_0) + d(\gamma, \gamma_1) - C_0,
\]
so, $d_{B_1}(y, y_0) \leq l(\wt \beta_0) + C_0$. So finally, since $l(\delta^1_i) \geq l(\wt \beta_i)$, we obtain
\[
 l(\delta^1_i) \ \geq \ l(\beta_i) \ \geq \ C i -  l(\delta^1_0) - (C_2 + C_0),
\]
 and the claim is proved.
\end{proof}

So the Proposition is proved if $d_{\textrm{Hor}}^1(\wt x_0, \wt x_i) \geq  AC_1i/(2+C_1)$, with a constant $A_1$ that can be taken to be $A_1 := AC_1/(2+C_1)$. We now suppose that $d_{\textrm{Hor}}^1(\wt x_0, \wt x_i) < AC_1i/(2+C_1)$. Then, 
\[
 d_{\textrm{Ver}}^1(\wt x_0, \wt x_i) \ \geq \ Ai - \frac{AC_1i}{2+C_1} = \frac{2Ai}{2+C_1}.
\]
Using the fact that $S_1$ and $S_2$ are glued together by a Dehn twist on $T$ such that the slope in the horizontal/vertical coordinates is $C_1$, we get that
\[
 d_{\textrm{Hor}}^2(\wt x_0, \wt x_i) \ \geq \ C_1  d_{\textrm{Ver}}^1(\wt x_0, \wt x_i)  -  d_{\textrm{Hor}}^1(\wt x_0, \wt x_i).
\]
Hence,
\begin{equation*}
 d_{\textrm{Hor}}^2(\wt x_0, \wt x_i) \ \geq \ C_1 \frac{2Ai}{2+C_1} - \frac{AC_1i}{2+C_1} = \frac{AC_1i}{2+C_1}
\end{equation*}
Therefore, we can apply the previous claim in the piece $S_2$ and thus finish the proof, with the same constant $A_1= AC_1/(2+C_1)$ as the first case.

Let us recap what we showed up to now: The proposition is proven if either $S_1$ or $S_2$ is a periodic piece that
is not a twisted $I$-bundle over the Klein bottle, 
or if both $S_1$ and $S_2$ are free and have an underlying orbifold 
admitting an hyperbolic structure. We assume now that this is not the case.
Whether $S_1, S_2$ are periodic or not, let $B_1, B_2$ be their base spaces.

\vskip .1in
\noindent
{\bf Second case:} Up to renaming $S_1$ and $S_2$, we can suppose that either
$S_1$ is periodic and a twisted $I$-bundle over the Klein bottle
or $S_1$ is free with $B_1$ not an hyperbolic orbifold.

So whether $S_1$ is periodic or free, $B_1$ is not an hyperbolic orbifold. Denoting by $\chi_O$ the orbifold Euler characteristic, by $\chi$ the topological Euler characteristic, and by $n_j$ the order of the cone points of $B_1$, we then have (since $S_1$ is Seifert, and so the singular points of $B_1$ can only be elliptic points)
\[
 \chi_O(B_1) = \chi(B_1) - \sum_{j} \left(1-\frac{1}{n_j}\right) \geq 0.
\]
Since $\chi_O(B_1) \geq 0$, we must have $\chi(B_1) \geq \sum_{j} \left(1-\frac{1}{n_j}\right) \geq 0$. Hence the list of topological types for $B_1$ is: the disk, the sphere, the real projective plane, the annulus, the M\"obius band, the Klein bottle and the torus. Now, by assumption, $S_1$ has at least one boundary torus, so $B_1$ can only be a disk, an annulus or a M\"obius band.

If $B_1$ is an annulus or a M\"obius band, then $\chi(B_1) = 0$, so $B_1$ cannot have any cone point. Suppose that $B_1$ is an annulus, then, since $M$ is orientable 
$S_1$ has to be orientable, so the Seifert 
fibration in $S_1$ has to be orientable. Hence $S_1$ has to be a torus times an interval, but no Seifert piece of a JSJ decomposition can be $\mathbb{T}^2 \times I$, so we have a contradiction in this case.

Hence, if $\chi(B_1) = 0$, then $B_1$ has to be a M\"obius band. And in that case, since $M$ is orientable, the $S^1$ fibration has to be non-orientable. So $S_1$ is a regular neighborhood of a one sided Klein bottle. In particular, $S_1$ has only one boundary component, and $\pi_1(S_1) = \pi_1(K)$, where $K$ is the Klein bottle. We let $N_1$ be a manifold which is a regular neighborhood of a one sided Klein bottle. 
In particular $N_1$ is a twisted $I$-bundle over the Klein bottle. 

Before continuing with the case $S_1 = N_1$, let us consider the other case left.

Suppose that $B_1$ is a disk, then, since $\chi_O(B_1) \geq 0$, $B_1$ has at most 2 singular fibers. Suppose that $B_1$ has either $0$ or $1$ singular fibers. Then $S_1$ is a solid torus, so its boundary torus $T$ is compressible, which is ruled out since $T$ is a tori of the JSJ decomposition of an irreducible manifold, so in particular $T$ has to be 
incompressible \cite{He,Ja-Sh}. 

So $B_1$ has two singular fibers of order $2$ each. In addition, since the disk is orientable and $M$ is assumed to be orientable too, the $S^1$ fibration has to be orientable. Call $N_2$ that Seifert manifold. A presentation of the fundamental group of $N_2$ is the following
\[
 \pi_1(N_2)= \langle c,d,h \mid  [c,h] = [d,h] = 1, c^2 = d^2 = h \rangle.
\]
In particular, setting $a= c$ and $b = d^{-1}c$, we see that
\[
 \pi_1(N_2)= \langle a,b \mid  a^{-1}ba = b^{-1} \rangle = \pi_1(K).
\]
So $N_2$ has also the fundamental group of the Klein bottle, and $N_2$ is also a regular neighborhood of a Klein bottle. 
It follows that the manifolds $N_1, N_2$ are homeomorphic, but the Seifert fibrations are 
different. This is one of the few manifolds where this happens.

Therefore, independently of whether $S_1 = N_1$ or $S_1= N_2$, we always have that $S_1$ has a \emph{unique} boundary torus, that we call $T$ and $\pi_1(S_1) = \pi_1(K)$. 
Moreover, $B_1$ always admit a finite cover which is an annulus
and the universal cover of $S_1$ is 
$\wt S_1 = \wt B_1 \times \R$.
We can replace the Seifert fibration in $N_2$ so that
the $\R$ factor is always the fiber direction and $\wt B_1$ is homeomorphic
to $[0,1] \times \R$.

We now turn our attention towards $S_2$. If we suppose that $B_2$, the base orbifold of $S_2$ is not hyperbolic, then, the argument just above shows that $S_2 = N_1$ or $S_2 = N_2$. In particular, $S_2$ has only one boundary torus $T$ and $M = S_1 \cup S_2$. As we saw earlier in the proof, this is impossible since otherwise $\pi_1(T)$ would be a subgroup of index at most $4$ in $\pi_1(M)$, which would contradict the fact that $\pi_1(M)$ has exponential growth.

Therefore, $B_2$ is an hyperbolic orbifold.

So in either case we obtain, up to switching $S_1$ and $S_2$, that $S_1$ is a twisted
$I$-bundle over the Klein bottle, and that $S_2$ is a free Seifert
piece with hyperbolic orbifold base.

In order to finish the proof we will apply a trick that will allow us to reduce the proof to what we did in the first case.
The manifold $S_1$ is a twisted $I$-bundle over the Klein bottle. This can be described
as follows. First let $T$ with coordinates $(x,y)$ defined $mod \ 1$. Let $j:T \rightarrow T$ be
the free involution of $T$ given by $j(x,y) = (x + 1/2, 1 - y)$. Now define $S_1$ to be the quotient
of $T \times [0,1]$ by $g(p,t) = (j(p), 1 - t)$. 
Up to isotopy there are two Seifert fibrations in $S_1$: the first $\mathcal{F}_1$ is given
by the curves $y = const, t = const$, the second $\mathcal{F}_2$ is defined by the 
curves $x = const, t = const$. 
The stable foliation when restricted to $S_1$ has up to isotopy an annular leaf
$y = 0$. The Seifert fibration $\mathcal{F}_1$ cannot be made transverse to the
stable foliation in $S_1$, but $\mathcal{F}_2$ can be made transverse. 
%If necessary we choose a different Seifert fibration in $S_1$, to make it transverse to the stable foliation in this piece. 

Let $f \colon T \subset S_2 \rightarrow T \subset S_1$ be the Dehn twist giving the gluing between $S_2$ and $S_1$. 
Now, if $V$ is a vertical fiber of $S_2$, it travels through $S_1$ and comes back to $S_2$,
it becomes the curve $f^{-1} \circ j \circ f (V)$. In particular, since $f(V)$ is a curve of slope $C_1>0$ in the horizontal/vertical coordinates in $S_1$, $j\circ f (V)$ will be of \emph{negative} slope and hence $f^{-1} \circ j \circ f (V)$ will have an even more negative slope.

The trick is the following: $S_1$ has a double cover $T \times [0,1]$. Since $K$ is one
sided in $S_1$ this produces a double cover $M_2$ of $M$ made up of $T \times [0,1]$ and
two copies of $M - int(S_1)$ glued along $T \times \{ 0 \}$ and 
$T \times \{ 1 \}$. In particular the Seifert piece $S_2$ lifts to two Seifert 
fibered spaces $S_3$ and $S_4$ contained in $M_2$ and each one is homeomorphic to $S_2$.
In addition since $S_1$ lifts to $T \times [0,1]$ then $S_3 \cup (T \times [0,1])$
is also a Seifert fibered space. But the Seifert fibration cannot be extended to
$S_4$ because as explained in the paragraph above the Seifert fiber $V'$ in
$S_3$ (a lift of the Seifert fiber $V$ in $S_2$) moves through $T \times [0,1]$
(corresponding to the curve $V$ moving through $S_1$) to a curve $V"$ which is a 
lift of $f^{-1} \circ j \circ f(V)$) and as explained in the paragraph above
this is {\emph{not}} a curve isotopic to a lift of the fiber $V$ in $S_2$.
This shows that the Seifert fibration in $S_3 \cup (T \times [0,1])$ cannot
extend into $S_4$. Notice that $S_3 \cup (T \times [0,1])$ is homeomorphic to
$S_3$ which is in turn homeomorphic to $S_2$ and hence has hyperbolic 
base orbifold. In addition the Seifert fibration in $S_3$ cannot extend to any
other parts of $M_2$ or else the projection to $M$ would extend the Seifert fibration
of $S_2$ in $M$. The important conclusion for us is that $S_3 \cup (T \times [0,1])$
and $S_4$ are Seifert fibered pieces of the JSJ decomposition of $M_2$. 
The Anosov flow in $M$ lifts to an Anosov flow in $M_2$ and the string of orbits
also does, with a factor of at most $2$ in the periods of the orbits.
Now $S_3 \cup (T \times [0,1])$ is free and with hyperbolic base orbifold.
The same is true of $S_4$.
The lifted string of orbits crosses through $S_3 \cup (T \times [0,1])$ into $S_4$. 
So we reduced the last possibility to the first case and the result follows.

This finally ends the proof of Proposition \ref{prop:linear_growth_seifert_pieces}.
\end{proof}

\vskip .1in

\section{Consequences for counting orbits} \label{section:consequences}

\subsection{Counting orbits in free homotopy classes}

First, it is a classical result that the number of periodic orbits of an Anosov flow grows exponentially fast with the period \cite{Bowen:periodic_orbits,MargulisThesis}. Moreover (when the flow is transitive and not a suspension of an Anosov diffeomorphism) the exponential rate of growth is the topological entropy of the flow. Several authors also studied the growth of periodic orbits when restricted to a given \emph{homology} class (see for instance \cite{BabillotLedrappier,KatSun,PhillipsSarnak,Sharp:closed_orbits_homology_classes} and references therein). In a fixed homology class, the rate of growth is still exponential, and they obtain some precise expression of the exponential rate.

Thanks to our previous results, we can give bounds for the rate of growth of the number of orbits inside a fixed 
\emph{free homotopy} class. Let us first explain what we exactly mean by that: The free homotopy class of an orbit is only well determined up to conjugacy, so when talking about a fixed free homotopy class, we fix a conjugacy class in $\pi_1 (M)$. We will write $\Cl(h)$ for the conjugacy class of an element $g\in \pi_1(M)$. If $\alpha$ is a closed orbit of $\flot$, then we write $\Cl(\alpha)$ for the conjugacy class in the fundamental group of $M$ that represents $\alpha$. So 
\[
\mathcal{FH}(\alpha) = \lbrace \beta \text{ closed orbit} \mid \Cl(\beta)=\Cl(\alpha) \rbrace. 
\]

\begin{theorem} \label{thm:counting_orbits}
Let $\flot$ be an Anosov flow on a $3$-manifold $M$, and $h$ be an element of the fundamental group of $M$.
\begin{enumerate}
 \item If $M$ is hyperbolic, then there exists a uniform constant $A_1>0$ and a constant $C_{1,h}$ depending on $h$ such that, for $t$ big enough,
\begin{equation*}
 \sharp \{ \alpha \text{ {\rm  closed orbit}} \mid \Cl(\alpha) = \Cl(h), \quad l(\alpha) <t \} \leq \frac{1}{A_1} \log(t) + C_{1,h}.
\end{equation*}
 \item If the JSJ decomposition of $M$ is such that no decomposition tori bounds a Seifert-fibered piece on both sides (so in particular, if all the pieces are atoroidal), then there exists a constant $C_{1,h}$ depending on $h$ such that, for $t$ big enough,
\begin{equation*}
 \sharp \{ \alpha \text{ {\rm  closed orbit}} \mid \Cl(\alpha) = \Cl(h), \quad l(\alpha) <t \} \leq  C_{1,h}\sqrt{t}.
\end{equation*}

 \item Otherwise, there exist constants $A_1>0$ and $B\geq 0$, that do not depend on $h$ if $M$ is a graph manifold, such that, for $t$ big enough, 
\begin{equation*}
 \sharp \{ \alpha \text{ {\rm  closed orbit}} \mid \Cl(\alpha) = \Cl(h), \quad l(\alpha) <t \} \leq  A_1t + B.
\end{equation*}
\end{enumerate}
So, in any case, the orbit growth inside a conjugacy class is at most linear in the period.

Moreover, independently of the topology of $M$, there exists a uniform constant $A_2>0$ and a constant $C_{2,h}$ depending on $h$ such that, if the set $\{ \alpha \text{ {\rm  closed orbit}} \mid \Cl(\alpha) = \Cl(h) \}$ is infinite, then for any $t$

\begin{equation*}
 \sharp \{ \alpha \text{ {\rm  closed orbit}} \mid \Cl(\alpha) = \Cl(h), \quad l(\alpha) <t \} \geq  \frac{1}{A_2} \log(t) - C_{2,h}.
\end{equation*}
\end{theorem}

\begin{rem}
 Note that, when $M$ is hyperbolic, the growth rate of the number of orbits inside a free homotopy class is exactly logarithmic.
\end{rem}

\begin{proof}
First note that, if $\left\{ \alpha \text{ closed orbit} \mid \Cl(\alpha) = \Cl(h) \right\}$ is empty or finite, then the first parts of the Theorem follows trivially, so we restrict our attention to elements $h$ that yields an infinite free homotopy class of orbits.
 
In order to prove the result, we also note that, thanks to Proposition \ref{prop:free_homotopy_class_to_strings}, counting the number of orbits in a free homotopy class is, up to a uniform factor, the same thing as counting orbits inside an infinite string. Hence the above result is just a transcription of Theorem \ref{thm:length_growth} and Theorem \ref{thm:upper_bound_length_growth}, if
we bound the worst case scenario in the finitely many strings of orbits in any free homotopy class.

Let $\{\alpha_i\}_{i\in \N}$ be an infinite string of orbits inside $\left\{ \alpha \text{ closed orbit} \mid \Cl(\alpha) = \Cl(h) \right\}$ and suppose that $\alpha_0$ is the shortest orbit in the string. We will prove the three different cases of the theorem separately and then prove the lower bound.
\begin{enumerate}
 \item If $M$ is hyperbolic, then, according to Proposition \ref{prop:length_growth_in_hyperbolic_piece}, there exist $A, B>0$ independent of the homotopy class and $D_{\alpha_0}$ depending on the length of $\alpha_0$ such that 
\begin{equation*}
 l(\alpha_i) \ \geq \ B e^{-D_{\alpha_0}} e^{Ai}.
\end{equation*}
So, if $l(\alpha_i)<t$, then 
\[
 i \ < \ \frac{1}{A}\log\left( \frac{te^{D_{\alpha_0}}}{B} \right) \ = \ \frac{1}{A} \log(t) + C_{\alpha},
\]
where $C_{\alpha}= (D_{\alpha_0} - \log(B))/A$ is a constant depending only on the infinite string chosen. Hence,
\[
  \sharp \{ \alpha_i \mid l(\alpha) <t \} \ \leq \ \frac{1}{A} \log(t) + C_{\alpha}.
\]
Using the above inequality and Proposition \ref{prop:free_homotopy_class_to_strings}, we get, up to a renaming of constants that, for $t$ big enough,
\[
  \sharp \{ \alpha \text{ closed orbit} \mid \Cl(\alpha) = \Cl(h), \quad l(\alpha) <t \} 
\ \leq \ \frac{1}{A} \log(t) + C_{h}.
\]
So the first case is proven.

  \item Suppose now that $M$ is such that no decomposition tori bounds a Seifert-fibered piece on both side. Since no infinite free homotopy class can stay entirely in a unique Seifert piece (by Theorem \ref{thm:no_periodic_piece}), the $\{\alpha_i\}_{i\in \N}$ has to go through one of the decomposition tori or be contained in an atoroidal piece of
the JSJ decomposition. And since the tori do not bound a Seifert-fibered piece on both sides, the orbits $\{\alpha_i\}$ have to enter an atoroidal piece of the modified JSJ decomposition or be contained in an atoroidal piece
of the JSJ decomposition. We can hence apply Proposition \ref{prop:quadratic_growth_on_atoroidal_piece}. There exists $B>0$ depending only on $M$ and the flow, and $D_{\alpha_0}>0$ depending on $\alpha_0$ such that
\begin{equation*}
 l(\alpha_i) \ \geq \ B i^2 e^{-D_{\alpha_0}}.
\end{equation*}
So, if $l(\alpha_i)<t$, then 
\[
 i^2  \ <  \ \frac{te^{D_{\alpha_0}}}{B},
\]
so,
\[
  \sharp \{ \alpha_i \mid l(\alpha_i) <t \} \ \leq \ \sqrt{t} \left( \frac{ e^{D_{\alpha_0} }}{B}\right)^{1/2}.
\]
Which implies, using again Proposition \ref{prop:free_homotopy_class_to_strings}, that for some constant $C_h$, depending only on $h$ and for $t$ big enough,
\begin{equation*}
 \sharp \{ \alpha \text{ closed orbit} \mid \Cl(\alpha) = \Cl(h), \quad l(\alpha) <t \} \ \leq \ C_{h}\sqrt{t}.
\end{equation*}

 \item Since linear growth is faster than a square root growth, the last case is proven by what we did above as soon as the orbits $\alpha_i$ enters an atoroidal piece. As mentioned above, thanks to Theorem \ref{thm:no_periodic_piece}, an infinite string cannot stay entirely in a unique Seifert piece. So the only case we are left to deal with is when the string crosses a decomposition torus that bounds two Seifert pieces (note that the torus can also bound the same Seifert piece on both side, but this is the same for us). We can then apply Proposition \ref{prop:linear_growth_seifert_pieces} and deduce the result in the same manner as above.
\end{enumerate}

Finally, to prove the lower bound on the growth rate, we use Theorem \ref{thm:upper_bound_length_growth}; There exist uniform constants $C_1, C_2>0$ such that 
\[
 l(\alpha_i) \ \leq \ C_1 l(\alpha_0) e^{C_2 i}.
\]
Hence, for any $i$ such that $i < \left(\log t - \log (C_1 l(\alpha_0) ) \right)/C_2$, we have $l(\alpha_i) <t$. So
\[
 \sharp \{ \alpha_i \mid \quad l(\alpha) <t \} \ \geq \ \frac{\log t}{C_2} - \frac{\log (C_1 l(\alpha_0) )}{C_2},
\]
and this finishes the proof.
\end{proof}

As we saw, the constants given in the preceding theorem depend on the free homotopy class we start with, 
but thanks to the Lemmas \ref{lem:control_D_alpha_hyperbolic} and \ref{lem:control_D_alpha_neutered}, 
we can also obtain some uniform growth rate for the upper bounds.
\begin{theorem} \label{thm:Uniform_control_growth_rate}
Let $\flot$ be an Anosov flow on a $3$-manifold $M$. There exist uniform constants 
$A_1, \dots, A_7 >0$ and $t_0$ such that, if $h$ is an element of the fundamental group of $M$, then for $t\geq t_0$,
\begin{enumerate}
 \item If $M$ is a graph manifold, then
\begin{equation*}
 \sharp \{ \alpha \text{ {\rm closed orbit}} \mid \Cl(\alpha) = \Cl(h), \quad l(\alpha) <t \} \ \leq \ A_1t + A_2.
\end{equation*}

 \item If $M$ is hyperbolic, then
\begin{equation*}
 \sharp \{ \alpha \text{ {\rm closed orbit}} \mid \Cl(\alpha) = \Cl(h), \quad l(\alpha) <t \} \ \leq \ A_3 \log(t) + 
A_3 \sqrt{t} \log(A_4t) + A_5
\end{equation*}

 \item Otherwise,
\begin{equation*}
 \sharp \{ \alpha \text{ {\rm closed orbit}} \mid \Cl(\alpha) = \Cl(h), \quad l(\alpha) <t \} \ \leq  \ A_6\sqrt{t} e^{\frac{\sqrt{t}}{2} \log (t/A_7)}
\end{equation*}
\end{enumerate}

So, independently of the topology of $M$, we can rename $t_0$ and $A_6, A_7$ so that
we always have, for $t\geq t_0$,
\begin{equation*}
 \sharp \{ \alpha \text{ {\rm closed orbit}} \mid \Cl(\alpha) = \Cl(h), \quad l(\alpha) <t \} \ \leq \ A_6\sqrt{t} e^{\frac{\sqrt{t}}{2} \log (t/A_7)}.
\end{equation*}
\end{theorem}

Note that Theorem \ref{thm:Uniform_control_growth_rate} is not trivial even when looking at finite free homotopy classes,
as opposed to Theorem \ref{thm:counting_orbits}. Indeed, one consequence of Theorem \ref{thm:Uniform_control_growth_rate} is that there exists constants $A_6, A_7>0$ uniform, such that if $t_{\text{max}}$ is the longest orbit in a given 
finite free homotopy class, then this class has cardinality bounded above by $ A_6\sqrt{t_{\text{max}}} e^{\sqrt{t_{\text{max}}} \log (t_{\text{max}}/A_7)/2}$.

\begin{proof}
 The proof of this theorem is almost the same as Theorem \ref{thm:counting_orbits}, we just replace the use of Proposition \ref{prop:length_growth_in_hyperbolic_piece} by Lemma \ref{lem:control_D_alpha_hyperbolic} and Proposition \ref{prop:quadratic_growth_on_atoroidal_piece} by Lemma \ref{lem:control_D_alpha_neutered}. The only difference is that, when we want a uniform control, the worse, i.e., fastest, control we get is in the case of manifolds containing one or more atoroidal pieces in their JSJ decomposition. 
Here are the details.

As before, choose a string of orbits $\{\alpha_i\}$ contained in the free homotopy class associated
with $h$, that is, $\left\{ \alpha \text{ closed orbit} \mid \Cl(\alpha) = \Cl(h) \right\}$. 
Note that, in order to get uniform controls on the constants and on the size of $t$, we do have to deal with finite string of orbits also.

Recall that there is a uniform bound on the numbers of strings of orbits in any free homotopy class
and a uniform bound on the number of orbits in a free homotopy class outside a string (by Proposition \ref{prop:free_homotopy_class_to_strings}). Therefore counting the number of orbits of length less than $t$ inside a string implies the result 
of Theorem \ref{thm:Uniform_control_growth_rate} up to a change of (uniform) constants.
\begin{enumerate}
 \item Suppose that $M$ is a graph-manifold. Then, either $\{\alpha_i\}$ stays in a Seifert fibered
piece, or it intersects at least two different Seifert-fibered pieces. In the first case, by Theorem \ref{thm:no_periodic_piece}, there exists a uniform bound on the number of orbits in $\{\alpha_i\}$, and the result follows trivially for $t$ big enough 
(and independently of $h$ in $\pi_1(M)$). In the latter case, we can apply Proposition \ref{prop:linear_growth_seifert_pieces} and get that, for some uniform constants $A_1, A_2 >0$, if $l(\alpha_i)<t$, then $l(\alpha_0)<t$ and
\[
 i \ < \ \left({\frac{A_2}{A_1}}\right) + \left({\frac{2}{A_1}}\right) t,
\]
which implies the result up to renaming $A_1, A_2$.

 \item If $M$ is hyperbolic, then we can apply Lemma \ref{lem:control_D_alpha_hyperbolic} and get that, 
if $l(\alpha_i)<t$ and $t > \max(4,\frac{ae}{2})$ where $a$ is the length of the shortest 
geodesic in $M$, then
\[
 i \ < \ \frac{1}{A} \log \left( \frac{t}{B} e^{\sqrt{t}\log(2t/a)} \right),
\]
where $A$, $B$ and $a$ are uniform constants.
Then

\begin{align*}
i &< \frac{1}{A} \left( \log t - \log B + \sqrt{t} \log \left(\frac{2t}{a}\right) \right) \\
  &\leq \ A_3 \log t + A_4 \sqrt{t} \log(A_4 t) + A_5,
\end{align*}
where $A_3 = \frac{1}{A}$, $A_4 = \frac{2}{A}$ and $A_5 = | \frac{\log B}{A} |$.
This implies the result in the hyperbolic case.

 \item In the general case, we have three possibilities: 
The first possibility is that the orbits of the string stay in a Seifert piece. 
Then Theorem \ref{thm:no_periodic_piece} yields the result.  Notice that there is a global
bound on the number of orbits.
The second option is that the orbits intersect two consecutive Seifert-fibered pieces.
In this case we have uniform growth bounded above by a linear function,
by Proposition \ref{prop:linear_growth_seifert_pieces}. 
The third and final option is that 
the orbits have to enter (cross) an atoroidal piece.
In this final case,
we can use Lemma \ref{lem:control_D_alpha_neutered},
 and we obtain that for some uniform constants, if $l(\alpha_i)<t$, then 
\[
 i^2 \ < \ \frac{t}{B}e^{\sqrt{t} \log(t/C)},
\]
therefore 
\[
i \ < \ \frac{\sqrt{t}}{\sqrt{B}} e^{\sqrt{t} \log{\left(t/C\right)} /2}.
\]
Then take $A_6 = \frac{1}{\sqrt{B}}$ and $A_7 = C$.

The third function is eventually bigger than the other two so up to changing $A_6, A_7$ and $t_0$,
we obtain the final statement of the theorem. \qedhere
\end{enumerate}
\end{proof}

\subsection{Counting the conjugacy classes}

We can now use Theorem \ref{thm:Uniform_control_growth_rate} to prove Theorem \ref{thmintro:exponential_growth_conjugacy_classes}, and thereby answer, in the case of Anosov flows on a $3$-manifold, the question raised by Plante and Thurston in \cite{PlanteThurston}.

If $\Cl(h)$ is the conjugacy class of an element $h \in \pi_1(M)$, we write $\alpha_{\Cl(h)}$ for the shortest periodic orbit in the conjugacy class $\Cl(h)$ if there is a periodic orbit in that conjugacy class.
We write
\[
 \Conj(t):= \left\{ \Cl(h) \mid h \in \pi_1(M), \; l\left( \alpha_{\Cl(h)}\right) <t \right\}
\]
for the set of conjugacy class in $\pi_1(M)$ that admit a periodic orbit representative of length less than $t$. Plante and Thurston \cite{PlanteThurston} asked if the number of elements in $\Conj(t)$ grew exponentially with $t$. We have

\begin{theorem} \label{thm:counting_conjugacy_classes}
 Let $\flot$ be an Anosov flow on a $3$-manifold $M$.
 
Then the number of conjugacy classes in $\pi_1(M)$ grows exponentially fast with the length of the shortest representative. Moreover, the exponential growth rate is equal to the exponential growth rate of the number of periodic orbits.

More precisely, there exists constants $A_6, A_7 >0$ and $t_0>0$, given by Theorem
\ref{thm:Uniform_control_growth_rate},  such that,

\begin{equation*}
 \sharp \Conj(t) \ \leq \ \sharp \{ \alpha \text{ {\rm closed orbit}} \mid l(\alpha) <t \},
\end{equation*}
and, for $t\geq t_0$,
\begin{equation*}
 \sharp \Conj(t) \ \geq \ \frac{1}{A_6\sqrt{t}} e^{-\frac{\sqrt{t}}{2} \log (t/A_7)} \ \sharp \{ \alpha \text{ {\rm closed orbit}} \mid l(\alpha) <t \}.
\end{equation*}

Moreover:
\begin{itemize}
 \item If $M$ is hyperbolic, then there exist $A_3, A_4, A_5>0$ such that, for all $t\geq t_0$,
\begin{equation*}
 \sharp \Conj(t) \ \geq \ \frac{1}{A_3\log(t) + A_3\sqrt{t}\log(A_4t) + A_5} \ \sharp \{ \alpha \text{ {\rm closed orbit}} \mid l(\alpha) <t \}.
\end{equation*}
 \item If $M$ is a graph manifold, then there exist $A_1, A_2 >0$ such that, for all $t\geq t_0$,
\begin{equation*}
 \sharp \Conj(t) \ \geq \ \frac{1}{A_1t + A_2} \ \sharp \{ \alpha \text{ {\rm closed orbit}} \mid l(\alpha) <t \}.
\end{equation*}
\end{itemize}
\end{theorem}

With this result, and Margulis' \cite{MargulisThesis} or Bowen's \cite{Bowen:periodic_orbits} counting results we
obtain the following:
\begin{corollary} \label{cor:counting_conjugacy_equal_topological_entropy}
 Let $\flot$ be a \emph{transitive} Anosov flow on a $3$-manifold $M$. Then
\begin{equation*}
 \lim_{t \rightarrow +\infty} \frac{1}{t} \log \sharp \Conj(t) = h_{\textrm{top}},
\end{equation*}
where $h_{\textrm{top}}$ is the topological entropy of the flow.
\end{corollary}

Note that, if we use more precise asymptotics of the number of periodic orbits (see for instance \cite{PollicottSharp_error_terms}), we could deduce a more precise control on $\sharp \Conj(t)$. However, even in the best possible case, i.e., when $M$ is hyperbolic, our results are not quite enough to deduce an actual asymptotic formula for $\sharp \Conj(t)$.

Note finally that Plante and Thurston asked the question about the growth of conjugacy classes in the setting of Anosov flows on a manifold of any dimension, so in a setting much more general than ours.
It is possible that parts of our method can be extended directly to higher dimensions
for codimension one Anosov flows. However, we are not aware of any previous results on that particular question.
Moreover, as previously mentioned, if the Verjovsky conjecture is true, then that question is void for codimension one Anosov flow in higher dimensional manifolds.

Before proving Theorem \ref{thm:counting_conjugacy_classes}, 
we also state another easy consequence, which is that the shortest orbit representatives of conjugacy classes are equidistributed. For $\alpha$ a periodic orbit of $\flot$, we can define a probability measure supported on it by setting
\[
 \delta_{\alpha} := \frac{1}{l(\alpha)} \mathrm{Leb}_{\alpha},
\]
where $\mathrm{Leb}_{\alpha}$ is the image of the Lebesgue measure on $[0, l(\alpha)]$ under the map $x \mapsto \flot x$, with $x\in \alpha$.
\begin{corollary} \label{cor:equidistribution}
 Let $\flot$ be a \emph{transitive} Anosov flow on a $3$-manifold $M$. 
 Then, the Bowen-Margulis measure $\mu_{BM}$ of $\flot$ (i.e., measure of maximal entropy) can be obtained as
\[
 \mu_{BM} = \lim_{t\rightarrow +\infty} \frac{1}{\sharp \Conj(t)}\sum_{\Cl(h) \in \Conj(t)} \delta_{\alpha_{\Cl(h)}}.
\]
\end{corollary}

\begin{proof}[Proof of Theorem \ref{thm:counting_conjugacy_classes}]

The first inequality in the Theorem is trivial: If a conjugacy class has its smallest representative of length less than $t$, then there exists at least an orbit with length less than $t$, so
\begin{equation*}
 \sharp \left\{ \Cl(h) \mid h \in \pi_1(M), \; l\left( \alpha_{\Cl(h) }\right) <t \right\} \ \leq \ \sharp \{ \alpha \mid l(\alpha) <t \}.
\end{equation*}

We now prove the second inequality.
For any $h\in \Gamma$, we set
\begin{equation*}
 N(\Cl(h), t) := \sharp \{ \alpha \text{ closed orbit of }\flot \mid \Cl(\alpha) =\Cl(h) , \quad l(\alpha) <t \}.
\end{equation*}

By Theorem \ref{thm:Uniform_control_growth_rate}, there exist uniform constants $A_6,A_7>0$, such that, for $t$ big enough 
$-$ that is $t \geq t_0$ 
(note that this is where we need to know that the $t$ does not depend on the conjugacy class of $h$ for the control given in Theorem \ref{thm:Uniform_control_growth_rate} to work)
\begin{align*}
 \sharp \{ \alpha \mid l(\alpha) <t \} &=  \sum_{\Cl(h) \in \Conj(t)}N(\Cl(h),t)\\
   &\leq \ A_6\sqrt{t} e^{\frac{\sqrt{t}}{2} \log (t/A_7)} \sharp \Conj(t).
\end{align*}
Which gives the second inequality.
The other inequalities follow in the same manner.

The exponential growth of the number of closed orbits of an Anosov flow is always positive (even when the flow is not transitive \cite{Bowen:periodic_orbits}). Therefore $\sharp \Conj(t)$ has exponential growth as well.
This is because $\{ \alpha \mid l(\alpha) < t \}$ grows at least as fast as
$e^{bt}$ for some $b > 0$ and $bt - \frac{\sqrt{t}}{2} \log\left(\frac{t}{A_7}\right) \geq ct$ for some $c > 0$ and
for all $t \geq t_1$ for some uniform time $t_1$. Therefore 
the number of conjugacy classes also grows exponentially fast with the length of the shortest representative.
\end{proof}

\begin{proof}[Proof of Corollary \ref{cor:counting_conjugacy_equal_topological_entropy}]
The second inequality in Theorem \ref{thm:counting_conjugacy_classes} yields
\begin{equation*}
 \frac{1}{t} \log \sharp \{ \alpha \mid l(\alpha) <t \} 
\ \leq \ \frac{1}{t} \left(\log (A_6 \sqrt{t}) + \frac{\sqrt{t}}{2} \log (t/A_7) \right)  + \frac{1}{t}\log \sharp \Conj(t).
\end{equation*}

Passing to the limit (if it exists), and also using
the first inequality of Theorem \ref{thm:counting_conjugacy_classes}, gives
\begin{equation*}
 \lim_{t \rightarrow +\infty} \frac{1}{t} \log \sharp \{ \alpha \mid l(\alpha) <t \} \ = \  \lim_{t \rightarrow +\infty} \frac{1}{t} \log \sharp \Conj(t),
\end{equation*}

And, since the flow is transitive, then Bowen's result in \cite{Bowen:periodic_orbits} (or Margulis \cite{MargulisThesis}) shows that the above limit exists and
\[
 \lim_{t \rightarrow +\infty} \frac{1}{t} \log \sharp \{ \alpha \mid l(\alpha) <t \} = h_{\textrm{top}},
\]
so this proves Corollary \ref{cor:counting_conjugacy_equal_topological_entropy}.
\end{proof}

To prove Corollary \ref{cor:equidistribution}, one could follow Bowen's original proof \cite{Bowen:periodic_orbits} that the closed orbits of an Anosov flow are equidistributed. Instead we copy the proof of equidistribution of closed orbits under homological constraints given by Babillot and Ledrappier in \cite{BabillotLedrappier}. Their proof is based on the following result of Kifer \cite{Kifer:Large_deviations94}
\begin{theorem}[Kifer \cite{Kifer:Large_deviations94}]
 If $\mathcal K$ is a closed subset of the set of $\flot$-invariant probability measures (equipped with the weak$^{\ast}$-topology), then
\begin{equation*}
 \limsup_{t \rightarrow +\infty} \frac{1}{t} \log \sharp \left\{ \alpha \mid \delta_{\alpha} \in \mathcal{K} , \: l(\alpha) <t \right\} \ \leq \ \sup_{\mu \in \mathcal{K}} h_{\mu},
\end{equation*}
where $h_{\mu}$ is the measure-theoretic entropy of $\mu$.
\end{theorem}

\begin{proof} [Proof of Corollary \ref{cor:equidistribution}]
Let $U$ be an open neighborhood of $\mu_{BM}$ and write $U^{c}$ for its complementary set. Since $\mu_{BM}$ is the measure of maximal entropy and it is unique, there exists $\eps_0>0$ such that $\sup_{\mu \in U^c} h_{\mu} \leq h_{\textrm{top}} - \eps_0$.

Recall that $h_{top} \geq h_{\mu}$ for any $\phi^t$-invariant probability
measure $\mu$.
Hence, by Kifer's result, for $t$ big enough
\begin{align*}
 \sharp \left\{\Cl(h) \mid \Cl(h)\in \Conj(t) , \text{ and } \delta_{\alpha_{\Cl(h)}} \in U^c \right\}\ &\leq \ \sharp \left\{ \alpha \mid \delta_{\alpha} \in U^c , \text{ and } l(\alpha) <t \right\} \\
   &\leq \ e^{t( h_{\textrm{top}} - \eps_0/2)}.
\end{align*}
Now, recall from Theorem \ref{thm:counting_conjugacy_classes}, that for $t$ big enough,
\begin{equation*}
 \sharp \Conj(t) \ \geq \ \frac{1}{A_6\sqrt{t}} e^{-\frac{\sqrt{t}}{2} \log (t/A_7)} \sharp \{ \alpha \mid l(\alpha) <t \},
\end{equation*}
so, since the flow is transitive, for $t$ big enough,
\begin{align*}
 \sharp \Conj(t)\  &\geq \ \frac{1}{A_6\sqrt{t}} e^{-\frac{\sqrt{t}}{2} \log (t/A_7)} e^{t(h_{\textrm{top}} - \eps_0/10)} \\
   &\geq \ e^{t(h_{\textrm{top}} - \eps_0/5)}.
\end{align*}
These two equations imply that
\[
\frac{1}{\sharp \Conj(t)} \sharp \left\{\Cl(h) \mid \Cl(h)\in \Conj(t) , \text{ and } \delta_{\alpha_{\Cl(h)}} \in U^c \right\} \ < \ e^{-3 t\eps_0/10},
\]
for $t$ big enough. 
Consider the sum
\[
\frac{1}{\sharp \Conj(t)} \sum_{Cl(h) \in \Conj(t) \cap U} \delta_{\alpha_{\Conj(t)}}  \ \  + \ \ 
\frac{1}{\sharp \Conj(t)} \sum_{Cl(h) \in \Conj(t) \cap U^c} \delta_{\alpha_{\Conj(t)}}
\]
By the above, the total mass of the second sum tends to zero when $t \rightarrow \infty$.
Hence any weak$^*$-limit of the sum of the two terms has to be in $U$ since the first sum
is in $U$ and the second part will converge to the zero measure.

Since $U$ is arbitrary, this shows that any weak limit of the original total sum
has be the Bowen Margulis measure.
\end{proof}

\section{Quasigeodesic behavior and $\R$-covered Anosov flows} \label{sec:quasigeodesic}

Let us first recall that a quasigeodesic is a quasi-isometric embedding of the real line or a segment
of the real line into an open complete manifold.
A quasi-isometry between metric spaces $(X,d), (Y,d')$ is a map $f: X \rightarrow Y$ so that there
are constants $k, c > 0$ so that for any $a,b$  in $X$ then
\[
\frac{1}{k} d(a,b) - c  \ \leq \ d'(f(a),f(b))  \ \leq \ k d(a,b) + c
\]
The map $f$ need not be continuous.
A flow on a compact manifold $M$ is quasigeodesic if the orbits of its lift to the universal cover are quasigeodesics with universal constants, that is, independent of the particular flow line.
The quasigeodesic question for flows is particularly important if the manifold is hyperbolic, because
in $\Hyp^3$ a quasigeodesic is a bounded distance from a minimal geodesic, with the bound
depending only on the $k,c$ of the associated quasi-isometry \cite{Thurston_book,Gromov_Hyperbolic_groups}. As such quasigeodesics are extremely important and
useful in the whole theory of hyperbolic $3$-manifolds.

Surprisingly enough the quasigeodesic question for flows in closed hyperbolic $3$-manifolds is easier to deal with
for certain classes of pseudo-Anosov flows, rather than the more restrictive Anosov flows.
In particular there is a huge amount of examples of pseudo-Anosov quasigeodesic flows in closed,
hyperbolic $3$-manifolds. 
For examples suspensions of pseudo-Anosov diffeomorphisms on surfaces \cite{Zeghib}. In addition any transversely oriented, $\R$-covered foliation in a closed hyperbolic $3$-manifold admits a transverse quasigeodesic pseudo-Anosov flow \cite{Calegari:geometry_of_R_covered,Fen:Foliations_TG3M,Fenley:Ideal_boundaries}.
The quasigeodesic property is then used to study the asymptotic behavior of the leaves
of the foliation lifted to the universal cover $\Hyp^3$ \cite{CannonThurston,Fenley:Ideal_boundaries}.

As for the ``supposedly" much  simpler case of Anosov flows,
the only examples of Anosov flows that are known to be quasigeodesic are the geodesic flows and suspensions of Anosov diffeomorphisms, and in each case the underlying manifold is not hyperbolic. 
The second author proved 20 years ago, in \cite{Fen:AFM}, that $\R$-covered Anosov flows on hyperbolic manifolds cannot be 
quasigeodesic. We prove here that the only $\R$-covered Anosov flows that could 
possibly be quasigeodesic are on graph-manifolds.
\begin{theorem} \label{thm:quasigeodesic}
 Let $\flot$ be an $\R$-covered Anosov flow on a $3$-manifold $M$. If $M$ admits an atoroidal piece in its JSJ decomposition, then $\flot$ is not quasigeodesic.
\end{theorem}

It seems likely however that $\R$-covered Anosov flows on graph-manifolds are indeed quasigeodesic, so we make the following:
\begin{conjecture}
 Let $\flot$ be an $\R$-covered Anosov flow on a $3$-manifold $M$. The flow is quasigeodesic if and only if $M$ is a graph-manifold.
\end{conjecture}

In order to prove Theorem \ref{thm:quasigeodesic}, we will first need to prove that there always exist periodic orbits in the interior of an atoroidal piece.
\begin{proposition} \label{prop:existence_periodic_orbits}
 Let $\flot$ be a $\R$-covered Anosov flow on a $3$-manifold $M$. Let $P$ be a piece of a modified JSJ decomposition
of $M$.
Then there exists periodic orbits in the interior of $P$.
\end{proposition}

A different proof than the one we are going to present would show that the above result holds for any transitive Anosov flow, not only the $\R$-covered ones. But since we do not need the more general result, and the proof in the $\R$-covered case is much nicer, we only present this one here.

We stress that the whole point of this proposition is that the periodic orbit is contained
in the {\emph{interior}} of $P$ as opposed to the boundary of $P$. Since $\partial P$ is 
made up of Birkhoff tori, there are always periodic orbits contained in $\partial P$.

\begin{proof}
Let $T$ be a quasi-transverse boundary torus of $P$ and $A$ be an open Birkhoff annulus (i.e., transverse) of $T$ such that orbits through $A$ enter the piece $P$.
We are going to show that there exist orbits through $A$ that stay in the interior of the piece $P$ and do not accumulate on the boundary. For any such orbit, there must exists a subsegment of the orbit
 in the interior of $P$ that comes back close to itself. We can then apply the Anosov closing lemma to it and get a 
periodic orbit in the interior of $P$.

Recall that, since the flow is $\R$-covered, each leaf space is homeomorphic to $\R$ and the orbit space is homeomorphic to a diagonal band in $\leafs \times \leafu$ (see Proposition \ref{prop:eta_s_eta_u}). We will be using this fact in all the proof.

The annulus $A$ lifts to a lozenge in the orbit space $\orb$ that we denote by $\wt A$. 
What we mean by that is that if $V$ is a lift of $A$ to the universal cover $\widetilde M$ then 
the {\emph{set of orbits}} intersected by $V$ is a lozenge in $\orb$.
Let $\alpha$ be an orbit intersecting $A$, $\alpha$ leaves the piece $P$ if and only if it intersects one of the exiting annuli
of $P$. Lifting that to the orbit space, it means that, if $\wt \alpha$ is a lift of $\alpha$ in $\wt A$, then $\alpha$ exit $P$ if and only if $\wt \alpha$ is also inside one of the lozenge in $\orb$ that projects to one of the exiting annuli. Moreover, $\alpha$ accumulates on one of the boundary tori of $P$ if and only if it is on the stable leaf of one of the periodic orbits of the boundary tori.

In the same way let $\{\wt B_i\}_{i}$ be the (countable) set of lozenges 
in $\orb$ that are all the lifts of the exiting annuli of $P$. We are going to show that $\wt A \smallsetminus \cup_i \wt B_i$ is an uncountable set in $\orb$, that is an uncountable set of orbits.
In addition the set $\wt A \smallsetminus \cup_i \overline{\wt B_i}$, where we remove also the sides of the lozenges, is still uncountable, so any orbit in that set projects to an orbit of $\phi^t$ that enters $P$ through $A$ and never leaves
$P$, and never accumulates on $\partial P$.

Since we are interested in the set $\wt A \smallsetminus \cup_i \wt B_i$, we can already remove all the $\wt B_i$ that do not intersect $\wt A$ from our considerations. So from now on, $\{\wt B_i\}_{i \in \N}$ is the set of all the lifts of the exiting annuli of $P$ such that $\wt A \cap \wt B_i \neq \emptyset$.

The first thing to remark is that $\wt A \cap \wt B_i$ is an open set in $\orb$ and it
 cannot contain one of the corners of $A$ or $\wt B_i$. Indeed, the corners of $\wt A$ and $\wt B_i$ are periodic orbits on the boundary tori, so in particular when projected to $M$ these closed orbits do not intersect any of the open Birkhoff annuli 
contained in the boundary tori of $P$ $-$ entering or exiting. 
This same remark applies to $\wt B_i \cap \wt B_j$ for any $i \neq j$.

Therefore, either $\wt B_i$ intersects all the stable leaves in $\wt A$, or it intersects all the unstable ones (see Figure \ref{fig:vert_and_horizontal_lozenges}). We say that $\wt B_i$ is \emph{vertical} if it intersects all the stable leaves of $\wt A$ (as the red lozenge $\wt B_j$ in Figure \ref{fig:vert_and_horizontal_lozenges}), and \emph{horizontal} otherwise (as the blue lozenge $\wt B_i$ in Figure \ref{fig:vert_and_horizontal_lozenges}).

\begin{figure}[h]
\begin{pspicture}(-0.5,-0.5)(5,5)
% \psgrid
\psline[linewidth=0.04cm,arrowsize=0.05cm 2.0,arrowlength=1.4,arrowinset=0.4]{->}(1,0)(5,0)
\psline[linewidth=0.04cm,arrowsize=0.05cm 2.0,arrowlength=1.4,arrowinset=0.4]{->}(0,1)(0,5)
\rput(5.2,-0.2){$\mathcal{L}^s$}
\rput(-0.2,5.2){$\mathcal{L}^u$}
\psline[linewidth=0.04cm](2,2)(3.5,2)
 \psline[linewidth=0.04cm](2,2)(2,3.5)
\psline[linewidth=0.04cm](3.5,2)(3.5,3.5)
\psline[linewidth=0.04cm](2,3.5)(3.5,3.5)
%  \psline[linewidth=0.04cm](0.5,2)(3.5,2)
%  \psline[linewidth=0.04cm](2,0.5)(2,3.5)
% \psline[linewidth=0.04cm](3.5,2)(3.5,5)
% \psline[linewidth=0.04cm](2,3.5)(5,3.5)
\psdots[dotsize=0.16](2,2)
\psdots[dotsize=0.16](3.5,3.5)
% \rput(1.1,3.9){$\eta^s\left(\hfs(o)\right)$}
% \put(3.6,3.6){$\eta\left(o\right)$ }
\put(2.6,2.6){$\widetilde A$}
\psline[linewidth=0.04cm,linecolor=blue](1.5,2.4)(3.9,2.4)
\psline[linewidth=0.04cm,linecolor=blue](3.9,3)(3.9,2.4)
\psline[linewidth=0.04cm,linecolor=blue](1.5,3)(3.9,3)
\psline[linewidth=0.04cm,linecolor=blue](1.5,3)(1.5,2.4)
\psdots[dotsize=0.16,linecolor=blue](1.5,2.4)
\psdots[dotsize=0.16,linecolor=blue](3.9,3)
\psline[linewidth=0.04cm,linecolor=red](2.4,1.5)(2.4,3.9)
\psline[linewidth=0.04cm,linecolor=red](3,3.9)(2.4,3.9)
\psline[linewidth=0.04cm,linecolor=red](3,1.5)(3,3.9)
\psline[linewidth=0.04cm,linecolor=red](3,1.5)(2.4,1.5)
\psdots[dotsize=0.16,linecolor=red](2.4,1.5)
\psdots[dotsize=0.16,linecolor=red](3,3.9)
\psline[linewidth=0.04cm,linestyle=dotted](2,0.5)(5,3.5)
\psline[linewidth=0.04cm,linestyle=dotted](0.5,2)(3.5,5)
\psline[linewidth=0.06cm,arrowsize=0.005]{[-]}(2,0)(3.5,0)
\psline[linewidth=0.06cm,arrowsize=0.005]{[-]}(0,2)(0,3.5)
\uput{3pt}[-135](2,0){$l_0^s$}
\uput{3pt}[-45](3.5,0){$l_1^s$}
\uput{3pt}[-135](0,2){$l_0^u$}
\uput{3pt}[135](0,3.5){$l_1^u$}
\put(1.6,2.6){\small \color{blue} $\widetilde B_i$}
\put(2.6,1.6){\small \color{red} $\widetilde B_j$}
\end{pspicture} 
\caption{The lozenge $\wt A$ with a vertical and an horizontal intersecting lozenges}
\label{fig:vert_and_horizontal_lozenges}
\end{figure}
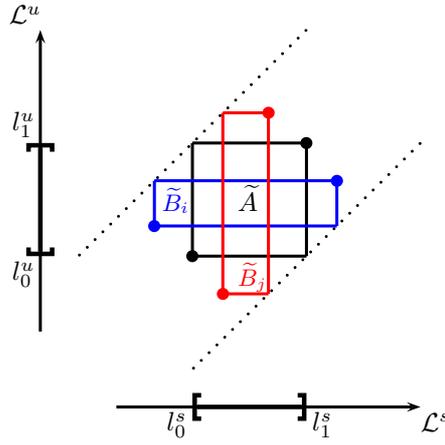

Since the intersection $\wt B_i \cap \wt B_j$ cannot contain any of their corners, then, up to switching $i$ and $j$, we have:
\begin{enumerate}
 \item either $\wt B_i$ is horizontal and $\wt B_j$ is vertical, (see Figure \ref{fig:vert_and_horizontal_lozenges})
 \item or $\wt B_i \cap A  \subset \wt B_j \cap A$, in particular $\wt B_i$ and $\wt B_j$ are both vertical or both horizontal (see Figure \ref{fig:nested_lozenges})
 \item or $\wt B_i$ and $\wt B_j$ are disjoint, and $\wt B_i$ and $\wt B_j$ are both vertical or both horizontal (see Figure \ref{fig:disjoint_intersection}).
\end{enumerate}

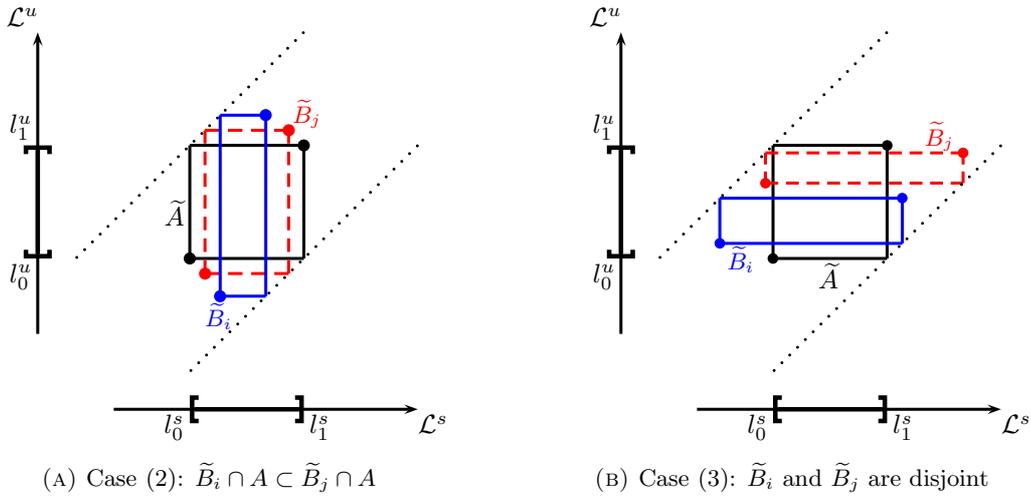
\begin{figure}[h]
 \begin{subfigure}[b]{0.45\textwidth}
\centering
 \scalebox{1} { 
  \begin{pspicture}(-0.5,-0.5)(5,5)
% \psgrid
\psline[linewidth=0.04cm,arrowsize=0.05cm 2.0,arrowlength=1.4,arrowinset=0.4]{->}(1,0)(5,0)
\psline[linewidth=0.04cm,arrowsize=0.05cm 2.0,arrowlength=1.4,arrowinset=0.4]{->}(0,1)(0,5)
\rput(5.2,-0.2){$\mathcal{L}^s$}
\rput(-0.2,5.2){$\mathcal{L}^u$}
 \psline[linewidth=0.04cm](2,2)(3.5,2)
 \psline[linewidth=0.04cm](2,2)(2,3.5)
\psline[linewidth=0.04cm](3.5,2)(3.5,3.5)
\psline[linewidth=0.04cm](2,3.5)(3.5,3.5)
%  \psline[linewidth=0.04cm](0.5,2)(3.5,2)
%  \psline[linewidth=0.04cm](2,0.5)(2,3.5)
% \psline[linewidth=0.04cm](3.5,2)(3.5,5)
% \psline[linewidth=0.04cm](2,3.5)(5,3.5)
\psdots[dotsize=0.16](2,2)
\psdots[dotsize=0.16](3.5,3.5)
% \put(2.6,2.6){$\widetilde A$}
\rput(1.8,2.6){$\widetilde A$}
% \psline[linewidth=0.04cm,linecolor=blue](0.9,2.4)(3.9,2.4)
% \psline[linewidth=0.04cm,linecolor=blue](3.9,5)(3.9,2.4)
% \psline[linewidth=0.04cm,linecolor=blue](1.5,3)(4.5,3)
% \psline[linewidth=0.04cm,linecolor=blue](1.5,3)(1.5,0.5)
% \psdots[dotsize=0.16,linecolor=blue](1.5,2.4)
% \psdots[dotsize=0.16,linecolor=blue](3.9,3)
\psline[linewidth=0.04cm,linecolor=red,linestyle=dashed](2.2,1.8)(3.3,1.8)
\psline[linewidth=0.04cm,linecolor=red,linestyle=dashed](2.2,1.8)(2.2,3.7)
\psline[linewidth=0.04cm,linecolor=red,linestyle=dashed](3.3,1.8)(3.3,3.7)
\psline[linewidth=0.04cm,linecolor=red,linestyle=dashed](3.3,3.7)(2.2,3.7)
\psdots[dotsize=0.16,linecolor=red](2.2,1.8)
\psdots[dotsize=0.16,linecolor=red](3.3,3.7)
\psline[linewidth=0.04cm,linecolor=blue](2.4,1.5)(2.4,3.9)
\psline[linewidth=0.04cm,linecolor=blue](3,3.9)(2.4,3.9)
\psline[linewidth=0.04cm,linecolor=blue](3,1.5)(3,3.9)
\psline[linewidth=0.04cm,linecolor=blue](3,1.5)(2.4,1.5)
\psdots[dotsize=0.16,linecolor=blue](2.4,1.5)
\psdots[dotsize=0.16,linecolor=blue](3,3.9)
\psline[linewidth=0.04cm,linestyle=dotted](2,0.5)(5,3.5)
\psline[linewidth=0.04cm,linestyle=dotted](0.5,2)(3.5,5)
\psline[linewidth=0.06cm,arrowsize=0.005]{[-]}(2,0)(3.5,0)
\psline[linewidth=0.06cm,arrowsize=0.005]{[-]}(0,2)(0,3.5)
\uput{3pt}[-135](2,0){$l_0^s$}
\uput{3pt}[-45](3.5,0){$l_1^s$}
\uput{3pt}[-135](0,2){$l_0^u$}
\uput{3pt}[135](0,3.5){$l_1^u$}
\put(2.2,1.1){\small \color{blue} $\widetilde B_i$}
\put(3.35,3.8){\small \color{red} $\widetilde B_j$}
% \put(2.6,1.55){\small $\widetilde B_i$}
% \put(2.95,3.2){\small $\widetilde B_j$}
\end{pspicture} }
\caption{Case (2): $\wt B_i \cap A  \subset \wt B_j \cap A$}
\label{fig:nested_lozenges}
 \end{subfigure}
\quad
\begin{subfigure}[b]{0.45\textwidth}
\centering
  \scalebox{1} { 
 \begin{pspicture}(-0.5,-0.5)(5,5)
% \psgrid
\psline[linewidth=0.04cm,arrowsize=0.05cm 2.0,arrowlength=1.4,arrowinset=0.4]{->}(1,0)(5,0)
\psline[linewidth=0.04cm,arrowsize=0.05cm 2.0,arrowlength=1.4,arrowinset=0.4]{->}(0,1)(0,5)
\rput(5.2,-0.2){$\mathcal{L}^s$}
\rput(-0.2,5.2){$\mathcal{L}^u$}
 \psline[linewidth=0.04cm](2,2)(3.5,2)
 \psline[linewidth=0.04cm](2,2)(2,3.5)
\psline[linewidth=0.04cm](3.5,2)(3.5,3.5)

\psline[linewidth=0.04cm](2,3.5)(3.5,3.5)
\psdots[dotsize=0.14](2,2)
\psdots[dotsize=0.14](3.5,3.5)
\put(2.6,1.6){$\widetilde A$}
% \put(2.6,2.6){$\widetilde A$}
\psline[linewidth=0.04cm,linecolor=blue](1.3,2.2)(3.7,2.2)
\psline[linewidth=0.04cm,linecolor=blue](3.7,2.8)(3.7,2.2)
\psline[linewidth=0.04cm,linecolor=blue](1.3,2.8)(3.7,2.8)
\psline[linewidth=0.04cm,linecolor=blue](1.3,2.8)(1.3,2.2)
\psdots[dotsize=0.14,linecolor=blue](1.3,2.2)
\psdots[dotsize=0.14,linecolor=blue](3.7,2.8)

\psline[linewidth=0.04cm,linecolor=red,linestyle=dashed](1.9,3.4)(4.5,3.4)
\psline[linewidth=0.04cm,linecolor=red,linestyle=dashed](1.9,3.4)(1.9,3)
\psline[linewidth=0.04cm,linecolor=red,linestyle=dashed](1.9,3)(4.5,3)
\psline[linewidth=0.04cm,linecolor=red,linestyle=dashed](4.5,3)(4.5,3.4)
\psdots[dotsize=0.14,linecolor=red](1.9,3)
\psdots[dotsize=0.14,linecolor=red](4.5,3.4)

\psline[linewidth=0.04cm,linestyle=dotted](2,0.5)(5,3.5)
\psline[linewidth=0.04cm,linestyle=dotted](0.5,2)(3.5,5)
\psline[linewidth=0.06cm,arrowsize=0.005]{[-]}(2,0)(3.5,0)
\psline[linewidth=0.06cm,arrowsize=0.005]{[-]}(0,2)(0,3.5)
\uput{3pt}[-135](2,0){$l_0^s$}
\uput{3pt}[-45](3.5,0){$l_1^s$}
\uput{3pt}[-135](0,2){$l_0^u$}
\uput{3pt}[135](0,3.5){$l_1^u$}

\put(1.4,1.85){\small \color{blue}  $\widetilde B_i$}
% \put(1.4,2.3){\small $\widetilde B_i$}
% \put(4,3.05){\small $\widetilde B_j$}
\put(4,3.5){\small \color{red} $\widetilde B_j$}
\end{pspicture} 
}
\caption{Case (3): $\wt B_i$ and $\wt B_j$ are disjoint}
\label{fig:disjoint_intersection}
 \end{subfigure}
\caption{Possible types of intersections of the lozenges $\wt B_i$ (blue) and $\wt B_j$ (red) with $\wt A$ (black)}
\label{fig:disjoint_or_nested_intersections}
\end{figure}

It turns out that there are no vertical lozenges in $\{\wt B_i\}$ (vertical lozenges would appear if we were considering $A$ as an \emph{exiting} Birkhoff annulus of some other piece and took the intersection with some entering Birkhoff annulus). However, we do not really need this fact and continue as if there were some, since it saves us some work.

Let $l_0^s$ and $l_1^s$ be the stable sides of $\wt A$, and let $l_0^u$ and $l_1^u$ be the unstable sides. Note that, if $l^s$ is the stable side of any horizontal lozenge $\wt B_i$,
then $l^s \in [l_0^s, l_1^s] \subset \leafs = \R$ (see Figure \ref{fig:vert_and_horizontal_lozenges}). And, similarly, if $l^u$ is the unstable side of any horizontal lozenge, then $l^u \in [l_0^u, l_1^u]$.

We claim that 
for any $\wt B_i$, there are at most finitely many $j$ such that $\wt B_i \cap A  \subset \wt B_j \cap A$: Suppose that this is not the case. Then we can suppose that $\{\wt B_j \cap A\}$ is an increasing sequence. Call $\delta^0_j$ and $\delta^1_j$ their corners. The sequence $\{\delta^1_j\}$ stays in a compact part of the orbit space (see Figure \ref{fig:nested_lozenges}). More precisely, $\{\delta^1_j\}$ stays in the compact rectangle delimited by the stable and unstable leaves of $\delta^1_0$ and of the top corner of $\wt A$. Hence, $\{\delta^1_j\}$ admits a converging subsequence. But this is impossible since the $\delta^1_j$ are lifts of a finite number of periodic orbits.

Therefore from the family $\{\wt B_i\}_{i \in \N}$, we can extract a subfamily
$\{\wt B_i\}_{i \in I}$ such that the intersection of $\wt B_i$ with $\wt A$ is maximal for $i\in I$. So, for any $i\neq j \in I$, either $\wt B_i$ and $\wt B_j$ are not of the same type, or they are disjoint.
In fact for any $\wt B_j$ of the original family, the intersection
$\wt B_j \cap \wt A$ is contained in some $\wt B_i \cap \wt A$, where $i \in I$.

For any vertical lozenge $\wt B_i$, $i \in I$, we set $I^u_i$ to be the closed interval consisting of the 
closure of the set of unstable leaves of $\wt B_i$. And for any horizontal lozenge, we set $I^{s}_i$ to be the 
closed interval consisting of the closure of the set of its stable leaves.
We claim that $I^u_i \cap I^u_j \not = \emptyset$ if $i \not = j$ in $I$.
Otherwise $\wt B_i, \wt B_j$ share a side. If this is true then they also share a corner orbit as
they are lozenges with periodic corners. But, because of the structure of skewed $\R$-covered Anosov flows,
there are no adjacent lozenges. It follows that if $\wt B_i \cap A \not = \emptyset$, then
$\wt B_j \cap A = \emptyset$. This is a contradiction.

The above implies that each of $[l_0^s, l_1^s] \smallsetminus \cap_{i \in I} I_i^s$ 
and $[l_0^u, l_1^u] \smallsetminus \cap_{i \in I} I_i^u$ is either a Cantor set or contains an open interval. Hence,
\[
 \wt A \smallsetminus \cup_i \wt B_i = \left([l_0^s, l_1^s] \smallsetminus \cap_{i \in I} I_i^s \right) \times \left([l_0^u, l_1^u] \smallsetminus \cap_{i \in I} I_i^u \right),
\]
is uncountable, which finishes the proof. 

In fact, the above set is a Cantor set times an interval. 
Indeed 
$[l_0^s, l_1^s] \smallsetminus \cap_{i \in I} I_i^s$ 
cannot contain an open set because the flow is transitive, so there exist dense orbits 
that intersects $A$ and any exiting annulus.
\end{proof}

\begin{proof}[Proof of Theorem \ref{thm:quasigeodesic}]

Suppose that the flow is quasigeodesic, but $M$ is not a graph-manifold. 
Up to taking a double cover we may assume that ${\mathcal F}^s$ is transversely
orientable.
Then either $M$ is hyperbolic or there exist an atoroidal piece in its JSJ decomposition. The first case was already dealt with by the second author in \cite{Fen:AFM}. So we suppose that there exists an atoroidal piece $P$
that is not all of $M$. By Proposition \ref{prop:existence_periodic_orbits}, there exists a periodic orbit $\alpha$ in the interior of $P$. Since the flow is $\R$-covered, $P$ is atoroidal, and $\alpha$ is not on the boundary tori, Theorem \ref{thm:finite_implies_seifert_or_torus} shows that $\alpha$ has an infinite free homotopy class. Let $\{\alpha_i\}_{i\in \Z}$ be the infinite free homotopy class of $\alpha$, indexed 
so that $\alpha_0$ is the shortest, and let $\al i$ be coherent lifts to the universal cover.

The idea now is to use what we did in the proof of Proposition \ref{prop:quadratic_growth_on_atoroidal_piece}: We showed in that proof that the length of the $\al i$ grows at least quadratically in the distance between $\al i$ and a certain geodesic $c_g$, but since $\al i$ is a quasigeodesic, its length cannot grow more than linearly in that distance, and we obtain a contradiction.

Let us be more precise. We use the same notations as in the proof of Proposition \ref{prop:quadratic_growth_on_atoroidal_piece}. In particular, $P$ is equipped with a neutered metric $d_N$ and $\wt P$ can be seen inside the hyperbolic space $\Hyp^3$. We use an $N$ subscript to refer to the neutered distance and $H$ subscript for the hyperbolic distance.

Let $g\in \pi_1(P)$ be the stabilizer of the $\al i$. Since $P$ is a neutered manifold, we can think of
$g$ as a hyperbolic isometry. Let $c_g$ be the geodesic in $\Hyp^3$ associated to $g$. Let $x$ be a point 
on $c_g$ which projects to a point inside $P$, and let $H_x$ be the hyperbolic hyperplane through $x$ and orthogonal to $c_g$. Finally, let $x_i$ be the closest point on $\al i \cap H_x$. Using Lemma \ref{lem:hyperbolic_geometry}, we get
\[
 l_N(\al i) \ = \ l_H(\al i) \ \geq \ l_H(c_g)\frac{e^{d_H(x_i,x)}}{2}.
\]
And, using Lemma \ref{lem:control_neutered_distance}, we have 
\[
 e^{d_H(x_i,x)/2} \ \geq \ d_N(x_i,x).
\]
So, we have 
\begin{equation} \label{eq:for_proof_not_quasigeodesic}
 l_N(\al i)  \ \geq \ \frac{l_H(c_g)}{2} d_N(x_i,x)^2.
\end{equation}

Now, since we assumed that $\flot$ is a quasigeodesic flow, the orbits $\al i$ are quasigeodesics and since they stay in the atoroidal piece $P$, they are quasigeodesics for the neutered distance. So there exist constants $C_1\geq 1$ and $C_2\geq 0$ such that 
\[
 l_N(\al i) \ \leq \ C_1 d_N(x_i, g\cdot x_i) + C_2.
\]
And the triangle inequality gives that 
\begin{align*}
 d_N(x_i, g\cdot x_i) \ &\leq \ d_N(x_i,x) + d_N(x,g\cdot x) + d_N(g\cdot x, g\cdot x_i) \\
& = \ 2d_N(x_i,x) + d_N(x,g\cdot x).
\end{align*}

So, setting $C_3 = C_1 d_N(x,g\cdot x) + C_2$, we get
\begin{equation*}
 l_N(\al i) \ \leq \ 2 C_1 d_N(x_i,x) + C_3.
\end{equation*}
Together with equation \eqref{eq:for_proof_not_quasigeodesic}, this gives, for all $i$
\begin{equation} \label{eq:to_use_in_proof2}
 2 C_1 d_N(x_i,x) + C_3 \ \geq \ \frac{l_H(c_g)}{2} d_N(x_i,x)^2.
\end{equation}

But $d_N(\al i, \wt \alpha_0) \geq Ai$ for some uniform $A>0$, thanks to Lemma \ref{lem:distance_greater_Ai}. So $ d_N(x_i,x) \geq Ai - d_{\textrm{Haus}}(\wt \alpha_0 , c_g)$, hence the equation \eqref{eq:to_use_in_proof2} cannot hold for big $i$, and we obtained our contradiction.
\end{proof}

\bibliographystyle{amsplaineprint_mine}
\bibliography{R_covered}

\def\cprime{$'$}
\providecommand{\bysame}{\leavevmode\hbox to3em{\hrulefill}\thinspace}
\providecommand{\MR}{\relax\ifhmode\unskip\space\fi MR }
\providecommand{\MRhref}[2]{%
  \href{http://www.ams.org/mathscinet-getitem?mr=#1}{#2}
}
\providecommand{\href}[2]{#2}
\begin{thebibliography}{10}

\bibitem{Anosov}
D.~V. Anosov, \emph{Geodesic flows on closed {R}iemannian manifolds of negative
  curvature}, Trudy Mat. Inst. Steklov. \textbf{90} (1967), 209.

\bibitem{BabillotLedrappier}
M.~Babillot and F.~Ledrappier, \emph{Lalley's theorem on periodic orbits of
  hyperbolic flows}, Ergodic Theory Dynam. Systems \textbf{18} (1998), no.~1,
  17--39.

\bibitem{Bar:CFA}
T.~Barbot, \emph{Caract\'erisation des flots d'{A}nosov en dimension 3 par
  leurs feuilletages faibles}, Ergodic Theory Dynam. Systems \textbf{15}
  (1995), no.~2, 247--270.

\bibitem{Bar:MPOT}
T.~Barbot, \emph{Mise en position optimale de tores par rapport \`a un flot
  d'{A}nosov}, Comment. Math. Helv. \textbf{70} (1995), no.~1, 113--160.

\bibitem{Barbot:VarGraphees}
T.~Barbot, \emph{Flots d'{A}nosov sur les vari\'et\'es graph\'ees au sens de
  {W}aldhausen}, Ann. Inst. Fourier (Grenoble) \textbf{46} (1996), no.~5,
  1451--1517.

\bibitem{Barbot:generalized_BL}
T.~Barbot, \emph{Generalizations of the {B}onatti-{L}angevin example of
  {A}nosov flow and their classification up to topological equivalence}, Comm.
  Anal. Geom. \textbf{6} (1998), no.~4, 749--798.

\bibitem{Bar:PAG}
T.~Barbot, \emph{Plane affine geometry and {A}nosov flows}, Ann. Sci. \'Ecole
  Norm. Sup. (4) \textbf{34} (2001), no.~6, 871--889.

\bibitem{Bar:HDR}
T.~Barbot, \emph{De l'hyperbolique au globalement hyperbolique}, Habilitation
  \`a diriger des recherches, Universit\'e Claude Bernard de Lyon, 2005,
  \url{http://www.univ-avignon.fr/fileadmin/documents/Users/Fiches_X_P/memoire%
CRY.pdf}.

\bibitem{BarbotFenley1}
T.~Barbot and S.~Fenley, \emph{Pseudo-{A}nosov flows in toroidal manifolds},
  Geom. Topol. \textbf{17} (2013), no.~4, 1877--1954.

\bibitem{BarbotFenley2}
T.~Barbot and S.~Fenley, \emph{Classification and rigidity of totally periodic
  pseudo-anosov flows in graph manifolds}, Ergodic Theory Dynam. Systems
  (2015), \href {http://arxiv.org/abs/1211.7327 [math.GT]}
  {\path{arXiv:1211.7327 [math.GT]}}.

\bibitem{BarbotFenley3}
T.~Barbot and S.~Fenley, \emph{Pseudo-{A}nosov flows and free {S}eifert fibered
  pieces},  (In preparation).

\bibitem{BBY}
F.~B\'{e}guin, C.~Bonatti, and B.~Yu, \emph{Building anosov flows on
  3-manifolds}, \href {http://arxiv.org/abs/1408.3951}
  {\path{arXiv:1408.3951}}.

\bibitem{BonattiLangevin}
C.~Bonatti and R.~Langevin, \emph{Un exemple de flot d'{A}nosov transitif
  transverse \`a un tore et non conjugu\'e \`a une suspension}, Ergodic Theory
  Dynam. Systems \textbf{14} (1994), no.~4, 633--643.

\bibitem{Bowen:periodic_orbits}
R.~Bowen, \emph{Periodic orbits for hyperbolic flows}, Amer. J. Math.
  \textbf{94} (1972), 1--30.

\bibitem{Brunella_Surfaces_section_expansive_flows}
M.~Brunella, \emph{Surfaces of section for expansive flows on three-manifolds},
  J. Math. Soc. Japan \textbf{47} (1995), no.~3, 491--501.

\bibitem{Calegari:geometry_of_R_covered}
D.~Calegari, \emph{The geometry of {${\bf R}$}-covered foliations}, Geom.
  Topol. \textbf{4} (2000), 457--515 (electronic).

\bibitem{CannonThurston}
J.~W. Cannon and W.~P. Thurston, \emph{Group invariant {P}eano curves}, Geom.
  Topol. \textbf{11} (2007), 1315--1355.

\bibitem{Ep}
D.~B.~A. Epstein, \emph{Curves on {$2$}-manifolds and isotopies}, Acta Math.
  \textbf{115} (1966), 83--107.

\bibitem{Fen:AFM}
S.~R. Fenley, \emph{Anosov flows in {$3$}-manifolds}, Ann. of Math. (2)
  \textbf{139} (1994), no.~1, 79--115.

\bibitem{Fen:QGAF}
S.~R. Fenley, \emph{Quasigeodesic {A}nosov flows and homotopic properties of
  flow lines}, J. Differential Geom. \textbf{41} (1995), no.~2, 479--514.

\bibitem{Fenley:Incompressible_tori}
S.~R. Fenley, \emph{Incompressible tori transverse to {A}nosov flows in
  {$3$}-manifolds}, Ergodic Theory Dynam. Systems \textbf{17} (1997), no.~1,
  105--121.

\bibitem{Fen:SBAF}
S.~R. Fenley, \emph{The structure of branching in {A}nosov flows of
  {$3$}-manifolds}, Comment. Math. Helv. \textbf{73} (1998), no.~2, 259--297.

\bibitem{Fe1}
S.~R. Fenley, \emph{Surfaces transverse to pseudo-{A}nosov flows and virtual
  fibers in {$3$}-manifolds}, Topology \textbf{38} (1999), no.~4, 823--859.

\bibitem{Fen:Foliations_TG3M}
S.~R. Fenley, \emph{Foliations, topology and geometry of 3-manifolds:
  $\mathbf{R}$-covered foliations and transverse pseudo-{A}nosov flows},
  Comment. Math. Helv. \textbf{77} (2002), no.~3, 415--490.

\bibitem{Fenley:Ideal_boundaries}
S.~R. Fenley, \emph{Ideal boundaries of pseudo-{A}nosov flows and uniform
  convergence groups, with connections and applications to large scale
  geometry}, Geom. Topol. \textbf{16} (2012), no.~1, 1--110.

\bibitem{Fenley:qg_PA_hyperbolic_manifolds}
S.~R. Fenley, \emph{Quasigeodesic pseudo-{A}nosov flows in hyperbolic
  $3$-manifolds and connections with large scale geometry},  (2014), \href
  {http://arxiv.org/abs/1405.4542} {\path{arXiv:1405.4542}}.

\bibitem{Fenley:diversified_behavior}
S.~R. Fenley, \emph{Diversified homotopic behavior of closed orbits of some
  $\mathbb{R}$-covered {A}nosov flows}, Ergodic Theory Dynam. Systems (2015),
  \href {http://arxiv.org/abs/1403.0310 [math.GT]} {\path{arXiv:1403.0310
  [math.GT]}}.

\bibitem{FouHassel:contact_anosov}
P.~Foulon and B.~Hasselblatt, \emph{Contact {A}nosov flows on hyperbolic
  3-manifolds}, Geom. Topol. \textbf{17} (2013), no.~2, 1225--1252.

\bibitem{Fr-Wi}
J.~Franks and B.~Williams, \emph{Anomalous {A}nosov flows}, Global theory of
  dynamical systems ({P}roc. {I}nternat. {C}onf., {N}orthwestern {U}niv.,
  {E}vanston, {I}ll., 1979), Lecture Notes in Math., vol. 819, Springer,
  Berlin, 1980, pp.~158--174.

\bibitem{Gabai:conv_groups_are_Fuchsian}
D.~Gabai, \emph{Convergence groups are {F}uchsian groups}, Ann. of Math. (2)
  \textbf{136} (1992), no.~3, 447--510.

\bibitem{Ghys:varietes_fibrees_en_cercles}
{\'E}.~Ghys, \emph{Flots d'{A}nosov sur les {$3$}-vari\'et\'es fibr\'ees en
  cercles}, Ergodic Theory Dynam. Systems \textbf{4} (1984), no.~1, 67--80.

\bibitem{Ghys:codim_one}
{\'E}.~Ghys, \emph{Codimension one {A}nosov flows and suspensions}, Dynamical
  systems, {V}alparaiso 1986, Lecture Notes in Math., vol. 1331, Springer,
  Berlin, 1988, pp.~59--72.

\bibitem{Gromov_Hyperbolic_groups}
M.~Gromov, \emph{Hyperbolic groups}, Essays in group theory, Math. Sci. Res.
  Inst. Publ., vol.~8, Springer, New York, 1987, pp.~75--263.

\bibitem{HandelThurston}
M.~Handel and W.~P. Thurston, \emph{Anosov flows on new three manifolds},
  Invent. Math. \textbf{59} (1980), no.~2, 95--103.

\bibitem{He}
J.~Hempel, \emph{{$3$}-{M}anifolds}, Princeton University Press, Princeton, N.
  J.; University of Tokyo Press, Tokyo, 1976, Ann. of Math. Studies, No. 86.

\bibitem{InabaMatsumoto}
T.~Inaba and S.~Matsumoto, \emph{Nonsingular expansive flows on {$3$}-manifolds
  and foliations with circle prong singularities}, Japan. J. Math. (N.S.)
  \textbf{16} (1990), no.~2, 329--340.

\bibitem{Ja-Sh}
W.~Jaco and P.~Shalen, \emph{Seifert fibered spaces in {$3$}-manifolds}, Mem.
  Amer. Math. Soc. \textbf{21} (1979), no.~220, viii+192.

\bibitem{KatSun}
A.~Katsuda and T.~Sunada, \emph{Closed orbits in homology classes}, Inst.
  Hautes \'Etudes Sci. Publ. Math. (1990), no.~71, 5--32.

\bibitem{Kifer:Large_deviations94}
Y.~Kifer, \emph{Large deviations, averaging and periodic orbits of dynamical
  systems}, Comm. Math. Phys. \textbf{162} (1994), no.~1, 33--46.

\bibitem{KlingenbergBook}
W.~Klingenberg, \emph{Riemannian geometry}, second ed., de Gruyter Studies in
  Mathematics, vol.~1, Walter de Gruyter \& Co., Berlin, 1995.

\bibitem{MargulisThesis}
G.~A. Margulis, \emph{On some aspects of the theory of {A}nosov systems},
  Springer Monographs in Mathematics, Springer-Verlag, Berlin, 2004, With a
  survey by Richard Sharp: Periodic orbits of hyperbolic flows, Translated from
  the Russian by Valentina Vladimirovna Szulikowska.

\bibitem{ParryPollicott}
W.~Parry and M.~Pollicott, \emph{An analogue of the prime number theorem for
  closed orbits of {A}xiom {A} flows}, Ann. of Math. (2) \textbf{118} (1983),
  no.~3, 573--591.

\bibitem{Paternain}
M.~Paternain, \emph{Expansive flows and the fundamental group}, Bol. Soc.
  Brasil. Mat. (N.S.) \textbf{24} (1993), no.~2, 179--199.

\bibitem{PhillipsSarnak}
R.~Phillips and P.~Sarnak, \emph{Geodesics in homology classes}, Duke Math. J.
  \textbf{55} (1987), no.~2, 287--297.

\bibitem{Plante:verjovsky_conjecture}
J.~F. Plante, \emph{Anosov flows, transversely affine foliations, and a
  conjecture of {V}erjovsky}, J. London Math. Soc. (2) \textbf{23} (1981),
  no.~2, 359--362.

\bibitem{PlanteThurston}
J.~F. Plante and W.~P. Thurston, \emph{Anosov flows and the fundamental group},
  Topology \textbf{11} (1972), 147--150.

\bibitem{PollicottSharp_error_terms}
M.~Pollicott and R.~Sharp, \emph{Error terms for closed orbits of hyperbolic
  flows}, Ergodic Theory Dynam. Systems \textbf{21} (2001), no.~2, 545--562.

\bibitem{Ro}
H.~Rosenberg, \emph{Foliations by planes}, Topology \textbf{7} (1968),
  131--138.

\bibitem{Serre}
J.-P. Serre, \emph{Trees}, Springer-Verlag, Berlin-New York, 1980, Translated
  from the French by John Stillwell.

\bibitem{Sharp:closed_orbits_homology_classes}
R.~Sharp, \emph{Closed orbits in homology classes for {A}nosov flows}, Ergodic
  Theory Dynam. Systems \textbf{13} (1993), no.~2, 387--408.

\bibitem{Simic:codim_one}
S.~Simi{\'c}, \emph{Codimension one {A}nosov flows and a conjecture of
  {V}erjovsky}, Ergodic Theory Dynam. Systems \textbf{17} (1997), no.~5,
  1211--1231.

\bibitem{Thurston_3manifolds_kleinian_groups}
W.~P. Thurston, \emph{Three-dimensional manifolds, {K}leinian groups and
  hyperbolic geometry}, Bull. Amer. Math. Soc. (N.S.) \textbf{6} (1982), no.~3,
  357--381.

\bibitem{Thurston_book}
W.~P. Thurston, \emph{Three-dimensional geometry and topology. {V}ol. 1},
  Princeton Mathematical Series, vol.~35, Princeton University Press,
  Princeton, NJ, 1997, Edited by Silvio Levy.

\bibitem{Ver:codim1}
A.~Verjovsky, \emph{Codimension one {A}nosov flows}, Bol. Soc. Mat. Mexicana
  (2) \textbf{19} (1974), no.~2, 49--77.

\bibitem{Zeghib}
A.~Zeghib, \emph{Sur les feuilletages g\'eod\'esiques continus des vari\'et\'es
  hyperboliques}, Invent. Math. \textbf{114} (1993), no.~1, 193--206.

\end{thebibliography}

\end{document}